\documentclass[12pt]{amsart}
\usepackage{verbatim,amscd,amssymb}
\usepackage{graphicx}

\newtheorem{thm}{Theorem}[section]
\newtheorem{prop}[thm]{Proposition}
\newtheorem{lem}[thm]{Lemma}

\newtheorem{cor}[thm]{Corollary}

\newtheorem*{thmA}{Theorem A}
\newtheorem*{thmB}{Theorem B}
\newtheorem*{thmC}{Theorem C}

\newtheorem*{thmm}{Main Theorem}
\newtheorem*{thmS}{Straightening Theorem}

\theoremstyle{definition}
\newtheorem{dfn}[thm]{Definition}

\def\C{\mathbb{C}}   

\def\R{\mathbb{R}}

\def\0{\emptyset}
  \def\Cc{\mathcal{C}}

\def\Fc{\mathcal{F}}   \def\Mc{\mathcal{M}}
\def\Pc{\mathcal{P}} \def\Rc{\mathcal{R}}  
\def\Uc{\mathcal{U}} \def\Vc{\mathcal{V}}  \def\Wc{\mathcal{W}}
 \def\Zc{\mathcal{Z}}  \def\tWc{\widehat{\mathcal{W}}}
\def\lam{\mathcal{L}} \def\Li{\mathcal{LI}}
\def\cf{\mathfrak{c}} 
\def\gf{\mathfrak{g}}  
 \def\hf{\mathfrak{h}}
  \def\Mf{\mathfrak{M}}
\def\Qf{\mathfrak{Q}}
\def\Uf{\mathfrak{U}} \def\Vf{\mathfrak{V}}

\def\Z{\mathbb{Z}}
\renewcommand\emptyset{\varnothing}
\newcommand{\sm}{\setminus}

\def\eps{\varepsilon}
\def\ol{\overline}

\def\si{\sigma}  \def\ta{\theta} \def\Ta{\Theta} \def\Ga{\Gamma}
\def\al{\alpha}  \def\be{\beta}  \def\la{\lambda} \def\ga{\gamma}
\def\om{\omega}
\def\vk{\varkappa}
\def\vp{\varphi}
\def\le{\leqslant}
\def\ge{\geqslant}
\def\imp{\mathrm{Imp}}
\def\cu{\mathcal{CU}}
\def\cuc{\mathcal{CU}}
\def\ch{\mathrm{CH}}
\def\uc{\mathbb{S}^1}
\def\thu{\mathrm{Th}}
\def\bd{\mathrm{Bd}}

\newcommand\3{\frac13}
\newcommand\4{\frac23}

\newcommand{\pr}{\mathrm{Pr}}

\def\disk{\mathbb{D}}
\def\cdisk{\overline{\mathbb{D}}}
\newcommand{\fg}{\mathrm{FG}}

\def\phd{\mathcal{PHD}}

\begin{document}
\date{August 23, 2016}

\title[Slices of Parameter Space of Cubic Polynomials]
{Slices of the Parameter Space of\\ Cubic Polynomials}

\author[A.~Blokh]{Alexander~Blokh}

\begin{comment}
\thanks{The first named author was partially
supported by NSF grant DMS--0901038}
\end{comment}

\author[L.~Oversteegen]{Lex Oversteegen}

\begin{comment}
\thanks{The second named author was partially  supported
by NSF grant DMS-0906316}
\end{comment}

\author[V.~Timorin]{Vladlen~Timorin}

\address[Alexander~Blokh, Lex~Oversteegen]
{Department of Mathematics\\ University of Alabama at Birmingham\\
Birmingham, AL 35294-1170}

\address[Vladlen~Timorin]
{Faculty of Mathematics\\
National Research University Higher School of Economics, Russian Federation\\
6 Usacheva St., 119048 Moscow
}

\address[Vladlen~Timorin]
{Independent University of Moscow\\
Bolshoy Vlasyevskiy Pereulok 11, 119002 Moscow, Russia}

\email[Alexander~Blokh]{ablokh@math.uab.edu}
\email[Lex~Oversteegen]{overstee@math.uab.edu}
\email[Vladlen~Timorin]{vtimorin@hse.ru}

\subjclass[2010]{Primary 37F20; Secondary 37C25, 37F10, 37F50}

\keywords{Complex dynamics; Julia set; Mandelbrot set}

\begin{abstract}
In this paper, we study slices of the parameter space of cubic
polynomials, up to affine conjugacy, given by a fixed value of the multiplier at a
non-repelling fixed point. In particular, we study the location of
the \emph{main cubioid} in this parameter space. The \emph{main
cubioid} is the set of affine conjugacy classes of complex cubic
polynomials that have certain dynamical properties generalizing
those of polynomials $z^2+c$ for $c$ in the filled main cardioid.
\end{abstract}

\maketitle

\section{Introduction}
By \emph{classes} of polynomials, we mean affine conjugacy classes. For
a polynomial $f$, let $[f]$ be its class. For any polynomial $f$, we
write $K(f)$ for the filled Julia set of $f$, and $J(f)$ for the Julia
set of $f$. The \emph{connectedness locus $\Mc_d$ of degree $d$} is the
set of classes of degree $d$ polynomials whose critical points \emph{do
not escape} (i.e., have bounded orbits). Equivalently, $\Mc_d$ is the
set of classes of degree $d$ polynomials $f$ whose Julia set $J(f)$ is
connected. The connectedness locus $\Mc_2$ of degree $2$ is otherwise
called the \emph{Mandelbrot set}; the connectedness locus $\Mc_3$ of
degree $3$ is also called the \emph{cubic connectedness locus}.  The
\emph{principal hyperbolic domain} $\phd_3$ of $\Mc_3$ can be defined
as the set of classes of hyperbolic cubic polynomials with Jordan curve
Julia sets. Equivalently, we have $[f]\in\phd_3$ if both critical
points of $f$ are in the immediate attracting basin of the same
attracting (or super-attracting) fixed point. Recall that a polynomial
is \emph{hyperbolic} if the orbits of all its critical points converge
to attracting or super-attracting cycles.

We define the \emph{main cubioid} $\cu$ as the set of classes $[f]\in\Mc_3$
with the following properties: $f$ has a non-repelling fixed point,
$f$ has no repelling periodic cutpoints in $J(f)$, and all non-repelling
periodic points of $f$, except at most one fixed point, have multiplier 1.
The main cubioid is a cubic analogue of the main cardioid
(by the \emph{main cardioid} we mean the closure of the family of all polynomials $z^2+c$
that have an attracting fixed point).
In this paper, we discuss properties of \emph{$\cu$-polynomials}, i.e.,
cubic polynomials $f$ such that $[f]\in\cu$.

\begin{thm}[\cite{bopt14}]\label{t:nonlin}
We have that $\ol\phd_3\subset\cu$.
\end{thm}

It is conjectured in \cite{bopt14} that $\ol\phd_3=\cu$. Some
properties of the complement $\cu\sm\ol\phd_3$ are discussed in
\cite{bopt14b}.

Let $\Fc$ be the space of polynomials
$$
f_{\lambda,b}(z)=\lambda z+b z^2+z^3,\quad \lambda\in \C,\quad b\in \C.
$$
An affine change of variables reduces any cubic polynomial $f$ to
the form $f_{\lambda,b}$. Clearly, $0$ is a fixed point for every
polynomial in $\Fc$.
Define the \emph{$\la$-slice} $\Fc_\lambda$ of $\Fc$ as the space of
all polynomials $g\in\Fc$ with $g'(0)=\lambda$. The $\lambda$-slice
$\Fc_\lambda$ maps onto the space of classes of all cubic polynomials
with a fixed point of multiplier $\lambda$ as a finite branched
covering. Indeed, it is easy to see that polynomials $f_{\lambda, b}$
and $f_{\lambda, -b}$ are affinely conjugate and belong to the same
class consisting of \emph{exactly} these polynomials. Hence, this
branched covering is equivalent to the map $b\mapsto a=b^2$, i.e.,
classes of polynomials $f_{\la,b}\in\Fc_\la$ are in one-to-one
correspondence with the values of $a$. Thus, if we talk about, say,
points $[f]$ of $\Mc_3$, then it suffices to take $f\in\Fc_\lambda$ for
some $\lambda$. There is no loss of generality in that we consider only
perturbations of $f$ in $\Fc$. The family $\Fc$ has been studied by
Zakeri \cite{Z}, Buff and Henriksen \cite{BH}. The main result of this
paper is a description of $\cu$ through $\lambda$-slices where
$|\lambda|\le 1$.

We use calligraphic (script) font for parameter space objects like
$\Fc$, $\Mc_3$, etc., to distinguish them from the dynamical plane
objects. We mostly use German Gothic fonts for various objects in the
closed disk related to \emph{laminations} and used in the combinatorial
models of polynomials (laminations will be introduced in
Subsection~\ref{ss:geolam}). We mostly use Greek letters for
\emph{angles}, i.e. elements of $\R/\Z$.

We need a few combinatorial concepts.
Given an angle $\al\in\R/\Z$, we write $\ol\al$ for the corresponding point $e^{2\pi i\al}$ of the unit circle $\uc=\{z\in\C\,|\, |z|=1\}$.
The angle tripling map $\si_3:\R/\Z\to\R/\Z$ identifies with the self-map of $\uc$ taking $z$ to $z^3$.
We write $(\ol\al,\ol\be)$ for an open arc of the unit circle with endpoints $\ol\al$ and $\ol\be$ if the direction from $\ol\al$ to $\ol\be$ within the arc is positive.
A closed chord of the closed unit disk $\cdisk=\{z\in\C\,|\, |z|\le 1\}$ with endpoints $\ol\al$, $\ol\be\in\uc$ is denoted by $\ol{\al\be}$.
Given a closed set $X\subset\uc$, define \emph{holes} of $X$ as components of $\uc\sm X$.

Let $\Uf\subset\ol\disk$ be the convex hull of $\Uf'=\Uf\cap\uc$.
By definition, \emph{holes} of $\Uf$ are holes of $\Uf'$; \emph{edges} of $\Uf$ are
chords on the boundary of $\Uf$. The set $\Uf$ is said to be a {\em
$($stand alone$)$ invariant gap} if $\si_3(\Uf')=\Uf'$, and, for every
hole $(\ol\al,\ol\be)$ of $\Uf$, we either have
$\si_3(\ol\al)=\si_3(\ol\be)$ (then the chord $\ol{\al\be}$ is called
\emph{critical}), or the circular arc $(\si_3(\ol\al),\si_3(\ol\be))$
is also a hole of $\Uf$. Extend the map $\si_3:\Uf'\to \Uf'$ to every
edge of $\Uf$ linearly so that a critical edge maps to a point, and a
non-critical edge $\ol{\al\be}$ maps to
$\ol{(3\al)(3\be)}=\si_3(\al)\si_3(\be)$ and keep the notation $\si_3$
for this extension. The \emph{degree} of $\Uf$ is the number of its
edges mapping \emph{onto} a non-degenerate edge of $\Uf$; it is
well-defined. Degree two gaps are also called \emph{quadratic gaps}.

We measure arc length in $\uc$ so that the total length of the entire
circle is 1. The \emph{length of a chord} $\ell$ of $\ol\disk$ is the
length of the shorter circle arc in $\uc$ connecting the endpoints of
$\ell$. A hole of $\Uf$ is called a \emph{major hole} if its length is
greater than or equal to $\frac13$. The edge of $\Uf$ connecting the
endpoints of a major hole is called a \emph{major edge}, or simply a
\emph{major}, of $\Uf$. Any \emph{quadratic} invariant gap $\Uf$ has
exactly one major that is either critical or periodic \cite{BOPT}.
By \cite{BOPT}, there exists a Cantor set $Q\subset\uc$ with the following property.
If we collapse every hole of $Q$ to a point, we obtain a topological circle
whose points are in one-to-one correspondence with \emph{all} quadratic
invariant gaps $\Uf$ such that $\Uf\cap \uc$ is a Cantor set.
Moreover, the following holds:

\begin{enumerate}

\item for each point $\ol\theta\in\uc$ of $Q$ that is not an endpoint
    of a hole of $Q$, the critical chord $\ol{(\theta+\frac
    13)(\theta+\frac 23)}$ is the major of a quadratic invariant gap
    $\Uf$ such that $\Uf\cap \uc$ is a Cantor set;

\item for each hole $(\ol\theta_1,\ol\theta_2)$ of $Q$ the chord
    $\ol{(\theta_1+\frac 13)(\theta_2+\frac 23)}$ is the periodic
    major of a quadratic invariant gap $\Uf$ such that $\Uf\cap \uc$
    is a Cantor set.

\end{enumerate}

The convex hull $\Qf$ of $Q$ in the plane is called the \emph{Principal
Quadratic Parameter Gap} (see Figure \ref{fig:Qf}); it is similar to the Main Cardioid of
$\Mc_2$.

\begin{figure}
  \includegraphics[height=6cm]{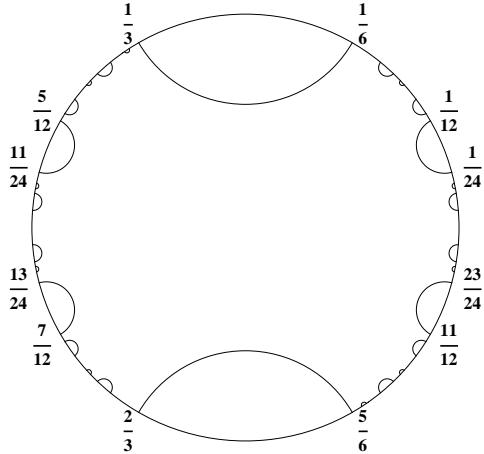}
  \caption{The Principal Quadratic Parameter Gap}
  \label{fig:Qf}
\end{figure}

Let us now introduce some analytic notions. Fix $\la$ with $|\la|\le
1$. The \emph{$\la$-connectedness locus $\Cc_\lambda$} is defined as the set of
all $f\in\Fc_\la$ such that $K(f)$ is connected, equivalently, such
that $[f]\in\Mc_3$. This is a full continuum \cite{BrHu,Z} (a compact
set $X\subset\C$ is \emph{full} if $\C\sm X$ is connected).
For every polynomial $f\in\Fc_\la$ and every angle $\al$, we define the
\emph{dynamic ray} $R_f(\al)$.
Also, for every angle $\theta$, in the parameter plane of $\Fc_\la$ we define the \emph{parameter ray} $\Rc_\lambda(\theta)$. We use rays to show that the picture in
$\Fc_\la$ resembles the picture in the parameter plane of quadratic
polynomials. Let $\cuc_\la$ be the set of all polynomials $f\in\Fc_\la$
with $[f]\in\cu$.

\begin{thmm}
The set $\cuc_\la$ is a full continuum. The set $\Cc_\la$ is the union
of $\cuc_\la$ and a countable family of \emph{limbs} $\Li_H$ of
$\Cc_\la$ parameterized by holes $H$ of $\Qf$.
The union is disjoint $($except when $\la=1)$.
For a hole $H=(\theta_1,\theta_2)$ of $\Qf$, the following holds.

\begin{enumerate}

\item The parameter rays $\Rc_\lambda(\theta_1)$ and
    $\Rc_\lambda(\theta_2)$ land at the same point $f_{root(H)}$.

\item Let $\Wc_\lambda(H)$ be the component of $\C\sm
    \ol{\Rc_\lambda(\theta_1)\cup \Rc_\lambda(\theta_2)}$ containing
    the parameter rays with arguments from $H$. Then, for every
    $f\in\Wc_\la(H)$, the dynamic rays $R_f(\theta_1+\frac 13)$,
    $R_f(\theta_2+\frac 23)$ land at the same point, either periodic
    and repelling for all $f\in\Wc_\la(H)$, or $0$ for all
    $f\in\Wc_\la(H)$. Moreover, $\Li_H=\Wc_\lambda(H)\cap \Cc_\la$.

\item The dynamic rays $R_{f_{root(H)}}(\theta_1+\frac 13)$,
    $R_{f_{root(H)}}(\theta_2+\frac 23)$ land at the same parabolic
    periodic point, and $f_{root(H)}$ belongs to $\cuc_\la$.
\end{enumerate}
\end{thmm}

Figure \ref{fig:Fc1d3} shows the parameter slice $\Fc_{e^{2\pi i/3}}$ in which several
parameter rays and several wakes are shown.

\begin{figure}
  \includegraphics[height=8cm]{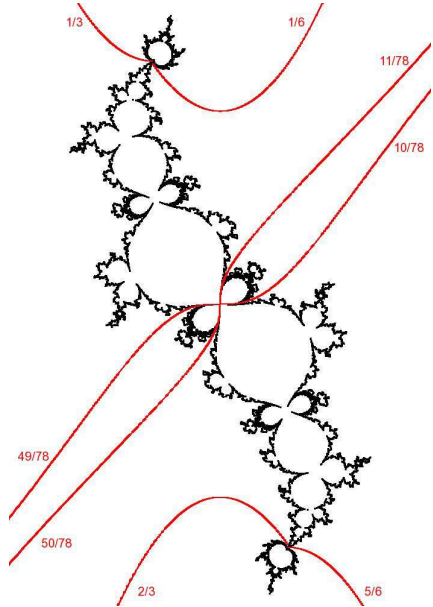}
  \caption{Parameter slice $\Fc_{e^{2\pi i/3}}$ with some rays}
  \label{fig:Fc1d3}
\end{figure}

\section{Detailed statement of the main results}\label{s:detail}

In this section, we break our Main Theorem into steps called Theorem A,
Theorem B and Theorem C.

\subsection{The structure of $\Fc_\la$}\label{ss:strufc}
A \emph{$($pre$)$critical} point of a polynomial $f$ is defined as a
point mapped to a critical point of $f$ by some iterate $f^{\circ r}$
of the polynomial $f$ (here $r\ge 0$). Let $G_{K(f)}$ be the Green
function for $K(f)$. Call unbounded trajectories of the gradient flow
for $G_{K(f)}$ \emph{dynamic rays} of $f$. Dynamic rays of $f$ can be
of two types. All but countably many of them accumulate in $K(f)$ and
are called \emph{smooth rays} (of $f$). Otherwise a dynamic ray extends
from infinity until it \emph{crashes} at an escaping (pre)critical
point of $f$; such rays can exist only if $K(f)$ is disconnected, and
escaping critical points exist. The remaining part of $\C\sm K(f)$
consists of bounded trajectories of the gradient flow for $G_{K(f)}$
called \emph{ideal rays} (of $f$). An ideal ray can extend from one
escaping (pre)critical point to another escaping (pre)critical point or
from an escaping (pre)critical point to the Julia set. Thus, if $c$ is
an escaping (pre)critical point, then several dynamic rays crash at $c$,
and some ideal rays accumulate at $c$. We conclude that $\C\sm K(f)$ is
the union of dynamic rays and ideal rays.

Let $V_f$ be the union of all dynamic rays of $f$.
Then $V_f$ is homeomorphic to $\C\sm \cdisk$ and coincides with the basin of attraction of infinity with countably many ideal rays removed.
The \emph{B\"ottcher coordinate} is an analytic map $\phi_f:V_f\to\C$ such
that $\phi_f(V_f)$ is $\C\sm \cdisk$ united with countably many finite radial
segments that originate at the boundary of $\cdisk$ and form a null sequence.
Moreover, we can choose $\phi_f$
so that the derivative $\phi'_f(z)$ tends to a positive real number as $z\to\infty$
and $\phi_f\circ f=\phi_f^3$. The existence of B\"ottcher coordinates
was established by Douady and Hubbard in \cite{DH}. Theorem \ref{T:BH} stated
below is a consequence of the analytic dependence of the
B\"ottcher coordinate on parameters \cite{DH,BrHu}.

\begin{thm}[\cite{BH}, Proposition 2]
 \label{T:BH}
Fix $\la$, and let $\Vc$ be the union of $\{b\}\times V_{f_{\la,b}}$ over all $b\in\C$.
This set is open in $\C^2$.
The map $\Psi:\Vc\to\C^2$ given by the formula $\Psi(b,z)=(b,\phi_{f_{\la,b}}(z))$
is an analytic embedding of $\Vc$ into $\C^2$.
\end{thm}

Note that dynamic rays of $f$ are preimages under $\phi_f$ of unbounded radial segments.
Hence dynamic rays can be parameterized by \emph{angles}, i.e., elements of $\R/\Z$: the
dynamic ray $R=R_f(\theta)$ of argument $\theta\in\R/\Z$ is such that for every
$z\in R$, we have
$\phi_f(z)=re^{2\pi i\theta}$, where $r>1$.
A dynamic ray $R$ \emph{lands} at a point $z\in\C$ if $z\in K(f)$, and $z$
is the only accumulation point of $R$ in $\C$.
The complement of $V_f$ in $\C\sm K(f)$ consists of ideal rays. Every accessible point
of $K(f)$ is the landing point of some dynamic ray or some ideal ray.

Recall that $|\la|\le 1$. Since $0$ is a non-repelling fixed point
of $f\in\Fc_\lambda$, it follows from the Fatou--Shishikura
inequality \cite{fat20, shi87} that at least one critical point
of $f$ is non-escaping. Thus, each $f$ in $\Fc_\la\sm\Cc_\la$
must have two \emph{distinct} critical points, one of which
escapes to infinity while the other is non-escaping. In this case,
the non-escaping critical point of $f$ will be denoted by
$\omega_1(f)$. Let $\omega_2(f)\ne \omega_1(f)$ denote the other critical point of
$f$. Let $\omega^*_2(f)$ be the corresponding \emph{co-critical
point} (i.e., $\omega^*_2(f)$ is the unique point with
$\omega^*_2(f)\ne \om_2(f)$ and $f(\omega^*_2(f))=f(\om_2(f))$).

We write $\Pc_\la$ for the set of polynomials $f\in\Fc_\la$ such that
$[f]\in\ol\phd_3$. By \cite{bopt14}, if $f\in\Pc_\la$, then
$[f]\in\cu$. By \cite{bopt14a}, the points $\om_1(f)$ and $\om_2(f)$
can be consistently defined for all $f\in\Fc_\la\sm\Pc_\la$. Moreover,
for every component $\Wc$ of $\Fc_\la\sm\Pc_\la$, the points $\om_1(f)$
and $\om_2(f)$ depend holomorphically on $f$ as $f$ moves through
$\Wc$. Observe that if $f\in \Fc_\la \sm \Cc_\la$, then $\om^*_2(f)\in
V_f$. Then the map $\Phi_\la(f)=\phi_f(\omega^*_2(f))$ is a conformal
isomorphism between the complement of $\Cc_\la$ in $\Fc_\la$ and the
complement of the closed unit disk \cite{BH}. This isomorphism can be
used to define \emph{parameter rays}. Namely, the ray
$\Rc_\lambda(\theta)$ is defined as the preimage of the straight ray
(unbounded radial segment) $\{re^{2\pi i\theta}\ |\ r>1\}$ under
$\Phi_\lambda$. We emphasize here that for \emph{dynamic rays} we use
the notation $R(\cdot)$ (with possible subscripts and superscripts)
while for \emph{parameter rays} we use the notation $\Rc(\cdot)$ (with
possible subscripts). This is consistent with our convention to use the
calligraphic font for parameter space objects. We have
$f\in\Rc_\lambda(\theta)$ if and only if $\omega^*_2(f)$ belongs to the
ray $R_f(\theta)$, i.e., if and only if both $R_f(\theta+1/3)$ and
$R_f(\theta+2/3)$ crash into the critical point $\omega_2(f)$.

\subsection{Immediate renormalization}\label{ss:immeren}

In \cite{lyu83, MSS}, the notion of \emph{$J$-stability} was introduced
for any holomorphic family $\mathcal T$ of rational functions: a map
from $\mathcal T$  is \emph{$J$-stable} with respect to $\mathcal T$ if
its Julia set admits an equivariant holomorphic motion over some
neighborhood of the map in the family. Say that $f\in \Fc_\la$ is
\emph{stable} if it is $J$-stable with respect to $\Fc_\lambda$, otherwise we call $f$ \emph{unstable}.
The set $\Fc^{st}_\la$ of all stable polynomials $f\in \Fc_\la$ is an open
subset of $\Fc_\la$. A component of $\Fc^{st}_\la$ is called a
\emph{$(\lambda$-$)$stable component}, or a \emph{domain of
$(\la$-$)$stability}. It is easy to see that, given $\la$, the
polynomial $f_{\lambda,b}$ has a disconnected Julia set if $|b|$ is
sufficiently big. Hence, if $f=f_{\lambda,b}$ is stable and $J(f)$ is
connected, then its domain of stability is bounded. Let $\cuc_\la$ be
the set of all polynomials $f\in\Fc_\la$ with $[f]\in\cu$.

Recall some topological concepts.
Observe that, given a compact set $A\subset \C$, there is a unique unbounded complementary domain $U$ of $A$.
The set $A$ is said to be \emph{unshielded} if $A$ coincides with the boundary of $U$.
For example, a polynomial Julia set is unshielded whether it is connected or not.
For an unshielded compact set $A$, the \emph{topological hull} $\thu(A)$ of $A$ is by definition the union of
$A$ with all bounded complementary components of $A$.
Equivalently, $\thu(A)$ is the complement of $U$.

\begin{thm}[\cite{bopt14,bopt14a}]
  \label{t:stab}
  All bounded components of $\Fc_\la\sm\Pc_\la$ consist of
  stable $\cu$-polynomials. Moreover, $\Pc_\la\subset\thu(\Pc_\la)\subset\cuc_\la$.
\end{thm}

The following  classic definition and major result are due to Douady
and Hubbard \cite{DH-pl}.

\begin{dfn}\label{d:ql}
A \emph{polynomial-like} map is a proper holomorphic map $f: U\to f(U)$ of degree $k>1$, where $U$, $f(U)\subset\C$ are open subsets isomorphic to a disk, and $\ol{U}\subset f(U)$. The \emph{filled Julia set} $K(f)$ of $f$ is the set of points in $U$ that never leave $U$ under iteration.
The \emph{Julia set} $J(f)$ of $f$ is defined as the boundary of $K(f)$.
Two polynomial-like maps $f:U\to f(U)$ and $g:V\to g(V)$ are said to be \emph{hybrid equivalent} if there is a quasi-conformal map $\vp$ from a neighborhood of $K(f)$ to a neighborhood of $K(g)$ conjugating $f$ to $g$ in the sense that $g\circ\vp=\vp\circ f$ wherever both sides are defined and such that $\ol\partial\vp=0$ almost everywhere on $K(f)$.
If $k=2$, then the corresponding polynomial-like maps are said to be \emph{quadratic-like}.
Note that $U$ can always be chosen as a Jordan domain.
\end{dfn}

\begin{thmS}[\cite{DH-pl}]
Let $f :U\to f(U)$ be a polynomial-like map.
Then $f$ is hybrid equivalent to a polynomial $P$ of the same degree.
Moreover, if $K(f)$ is connected, then $P$ is unique up to $($global$)$
conjugation by an affine map.
\end{thmS}

We will need the following definition.

\begin{dfn}\label{d:plrays}
Let $f$ be a polynomial, $f^*=f|_U$ be a polynomial-like
map, and $g$ be the polynomial hybrid equivalent to $f^*$.
The curves in $\C\sm K(f^*)$ corresponding to dynamic rays of $g$
are called \emph{polynomial-like rays} of $f$. If the degree of $f^*$
is two, then we will talk about quadratic-like rays.
We will denote polynomial-like rays $R^*(\be)$, where $\be$ is the argument of the
external ray of $g$ corresponding to $R^*(\be)$.
\end{dfn}

Note that polynomial-like rays of $f$ are only defined in a bounded neighborhood of $K(f^*)$.
Observe also that the polynomial-like (qua\-dra\-tic-like) map will be always specified when we talk about polynomial-like (quadratic-like) rays, which is why we omit it from our notation.

Say that a cubic polynomial $f\in\Fc$ is \emph{immediately
renormalizable} if there are Jordan domains $U^*$ and $V^*$ such that
$0\in U^*$, and $f^*=f:U^*\to V^*$ is a
quadratic-like map (we will use the notation $f^*$ at several occasions
in the future when we talk about immediately renormalizable maps). If
$f\in \Fc_\la$ with $|\la|\le 1$ is immediately renormalizable, then the quadratic-like
Julia set $J(f^*)=J^*$ is connected. Indeed, $f^*$ is hybrid equivalent
to a quadratic polynomial $g$. Since $0\in K(f^*)$ is a non-repelling
$f$-fixed point, it corresponds to a non-repelling fixed point of $g$.
Hence, $J(g)$ and $J(f^*)$ are connected, and $g(z)=z^2+c$ with $c$ in
the filled main cardioid. The filled quadratic-like Julia set of $f^*$ is
denoted by $K(f^*)$. In \cite{bopt14a}, some sufficient conditions on
polynomials for being immediately renormalizable are obtained.

\begin{thm}[\cite{bopt14a}]\label{t:unboren}
  All polynomials in the unbounded component of $\Fc_\la\sm\Pc_\la$
  are immediately renormalizable.
\end{thm}

Theorem~\ref{t:unboren} and Theorem~\ref{t:stab} imply the following
corollary.

\begin{cor}\label{c:unboren}
All polynomials in $\Fc_\la\sm\cuc_\la$ are immediately renormalizable.
\end{cor}

\begin{proof}
By Theorem \ref{t:stab}, $\thu(\Pc_\la)\subset \cuc_\la$. Therefore, any
polynomial $f\in \Fc_\la\sm\cuc_\la$ in fact belongs to $\Fc_\la\sm
\thu(\Pc_\la)$ and therefore, by Theorem \ref{t:unboren}, is immediately
renormalizable as desired.
\end{proof}

The discussion below, following \cite{bclos}, aims at relating
qua\-dra\-tic-like rays to dynamic and ideal rays (observe that
\cite{bclos} deals with polynomials of arbitrary degree). Define an
\emph{external ray $R^e$} as a smooth dynamic ray, or a one-sided
limit of smooth dynamic rays. Slightly abusing the language, we will
call such rays either \emph{smooth} rays or \emph{one-sided} rays.
The analysis of the structure of $\C\sm K(f)$ given above shows that
every external ray $R^e$ has its initial (unbounded) part coinciding
with the dynamic ray of a well-defined argument $\theta$.
If $R^e$ is itself a smooth ray, this completely
defines its argument as $\theta$.
The ray $R^e$ is then approximated by smooth external rays from either side.
If, however, $R^e$
is not smooth, then the same analysis shows that $R^e$ is the one-sided
limit of smooth rays from exactly one side, in which case we associate
to $R^e$ the appropriate \emph{one-sided external argument}, $\theta^+$
or $\theta^-$, and denote $R^e$ by $R^e_f(\ta^-)$ or $R^e_f(\ta^+)$,
respectively. Observe that parts of $R^e_f(\ta^-)$ and $R^e_f(\ta^+)$
from infinity to a certain precritical point in the plane coincide.
If $R^e$ is a smooth ray of argument $\ta$, then we set $R^e(\ta^+)=R^e(\ta^-)=R^e$.

An external ray $R^e$ of a polynomial $f$ accumulates in a unique
component of $J(f)$. If $R^e$ accumulates in a component $E$ of $J(f)$, then
$R^e$ is called \emph{an external ray to $E$}. For a non-closed
oriented curve $l$ from infinity or a finite point in $\C$ to a bounded
region in $\C$, define its \emph{principal} (or \emph{limit}) set
$\pr(l)$, analogously to how it is done for conformal
external rays, as the set of all accumulation points of $l$ in the forward direction, i.e., as the $\om$-limit set of $l$.
Thus, the \emph{principal set $\pr(R^e)$} of $R^e$ is the set of all accumulation points
of $R^e$ in $J(f)$.
Possible intersections of external rays are described in \cite{bclos}.

\begin{lem}[Lemma 6.1 \cite{bclos}]\label{l:6.1}
If two distinct external rays $R_0^e$, $R^e_1$ have a common point, then
$R_0^e$, $R^e_1$ are both non-smooth. The intersection $L = R_0^e \cap R^e_1$
is connected
and can contain a $($pre$)$critical point only as an endpoint. Furthermore, one
and only one of the following cases holds:

\begin{enumerate}

\item $L$ is a smooth curve joining infinity and a $($pre$)$critical point;

\item $L$ is a single $($pre$)$critical point;

\item $L$ is a smooth closed arc between two $($pre$)$critical points;

\item $L$ is a smooth curve from a $($pre$)$critical point to $J(f)$ and, moreover,
the rays $R_0^e, R^e_1$ are not periodic.

\end{enumerate}

Except for the last case, the rays $R_0^e$, $R^e_1$ have their principal sets in
different components of $J(f)$.
\end{lem}

In our setting, external rays are closely related to quadratic-like
rays. To describe this relation, we need Lemma~\ref{l:pericom},
which follows immediately from \cite{BrHu}. Recall that any $f\in
\Fc_\la\sm \Cc_\la$ with $|\la|\le 1$ is immediately renormalizable.

\begin{lem}\label{l:pericom}
If $f\in \Fc_\la\sm \Cc_\la$ with $|\la|\le 1$, then $J(f^*)$ is an
invariant component of $J(f)$, and $\thu(J(f^*))=K(f^*)$ is a component of
$K(f)$.
\end{lem}

Hence, if $J(f)$ is disconnected, then we can choose neighborhoods $V^*\supset U^*\supset J(f^*)$ enclosed by components of equipotentials.
In addition to quadratic-like rays accumulating in $J(f^*)$, it is natural to consider external rays that penetrate $U^*$; by definition and by the choice of $U^*$, any such ray has its entire tail inside $U^*$ so that its principal set is located inside $U^*$.

\begin{dfn}\label{d:corres}
A quadratic-like ray $R^*$ and an external ray $R^e$ to $J(f^*)$
\emph{correspond} to each other if $\pr(R^*)=\pr(R^e)$, and
$R^*$, $R^e$ are homotopic in $\mathbb{C}\setminus K(f^*)$ among
curves with the same limit set.
\end{dfn}

Let $\psi:\C\sm K(f^*)\to \C\sm \cdisk$ be the Riemann map with $\psi(z)\sim kz$ as $z\to \infty$ with some $k>0$.
It follows that, in the situation of Definition \ref{d:corres}, the rays $\psi(R^*)\subset \C\sm
\cdisk$ and $\psi(R^e)\subset \C\sm \cdisk$ land at the same point in
$\uc=\bd(\disk)$. Observe also that no two distinct quadratic-like rays
can correspond to each other in the sense of Definition~\ref{d:corres}.
Theorem~\ref{t-extepoly} is a special case of a more general result
obtained in Theorem 6.9 of \cite{bclos}.

\begin{thm}[cf \cite{bclos}, Theorem 6.9]\label{t-extepoly}
Consider a polynomial $f\in \Fc_\la\sm \Cc_\la$ with $|\la|\le 1$.
Then every external ray $R^e$ to $J(f^*)$
corresponds to exactly one quadratic-like ray $l=\xi(R^e)$.
The mapping $\xi:R^e\mapsto l$ from external rays to $J(f^*)$ to quadratic-like rays is onto
and $\xi^{-1}(l)$ consists of finitely many rays. Moreover, if
$\xi^{-1}(l)=\{R^e_1,\dots,R^e_k\}$ with $k>1$, then one
of the following holds:
\begin{enumerate}
\item[(i)] we have $k=2$, both rays $R^e_1$, $R^e_2$ are non-smooth, meet at
a precritical point $x$, and share a common arc from $x$ to $E$;
\item[(ii)] the rays $R^e_1$, $\dots$, $R^e_k$ land at the same preperiodic point,
and include at least one pair of disjoint rays.
\end{enumerate}
\end{thm}

\begin{comment}
\begin{proof}
By Lemma~\ref{l:pericom} $J(f^*)$ is an invariant component of $J(f)$.
The rest of the theorem follows from Theorem 6.9 of \cite{bclos}.
\end{proof}
\end{comment}

\subsection{Combinatorics of the angle tripling map: an overview}
\label{ss:overview} In this subsection, we briefly describe some
results of \cite{BOPT}.

\begin{comment}
In the context of angles (e.g., speaking of an argument of an external
ray or of an arc of angles between two angles) we will still mostly use
the usual notation for angles (such as $\al$, $\be$ etc). However, in
the laminational context, we will mostly use notation $\ol\alpha$,
$\ol\be$, etc.

The angle tripling map $\si_3$ identifies with the self-map of $\uc$
taking $z$ to $z^3$. A chord of $\cdisk$ with endpoints $\ol\al$,
$\ol\be\in\uc$ will be denoted by $\ol{\al\be}$ (e.g., $\ol{0\frac 12}$
stands for the diameter of $\cdisk$ connecting the points $\ol 0$ and
$\ol{\frac 12}$). The endpoints $\ol\al$, $\ol\be$ are included into
$\ol{\al\be}$. Given a closed subset $X\subset\uc$, define \emph{holes}
of $X$ as components of $\uc\sm X$. We will write $(\ol\al,\ol\be)$ for
a hole of $X$ if $\ol\al$ and $\ol\be\in\uc$ are the endpoints of the
hole, and, as we pass from $\ol\al$ to $\ol\be$ in the positive
direction, we encounter no points of $X$.

Let $\Uf\subset\ol\disk$ be the convex hull of $\Uf'=\Uf\cap\uc$.
\emph{Holes} of $\Uf$ are defined as holes of $\Uf'$. \emph{Edges} of
$\Uf$ are defined as chords on the boundary of $\Uf$.

\end{comment}

Let $\Uf\subset\ol\disk$ be the convex hull of $\Uf'=\Uf\cap\uc$. It is
said to map under $\si_3$ in a {\em quasi-covering fashion} if,
\begin{enumerate}
\item we have $\si_3(\Uf')=\Vf'$ for some set $\Vf'$ and
    $\Vf=\ch(\Vf')$, and
\item for every hole $(\ol\al,\ol\be)$ of $\Uf$, we either have
    $\si_3(\ol\al)=\si_3(\ol\be)$, or the circular arc
    $(\si_3(\ol\al),\si_3(\ol\be))$ is also a hole of $\Vf$.
\end{enumerate}
So, the set $\Uf$ is a {\em $($stand alone$)$ invariant gap} if
$\si_3(\Uf')=\Uf'$, and $\si_3$ maps $\Uf$ in a quasi-covering fashion.
The set $\Uf$ is a {\em $($stand alone$)$ periodic gap/leaf} if there
exists $n>0$ such that $\si_3^{\circ n}(\Uf')=\Uf'$, convex hulls of
sets $\si_3^{\circ i}(\Uf')$, $i=0$, $\dots$, $n-1$ intersect
\textbf{at most} by common edges, and each such convex hull maps under
$\si_3$ in a quasi-covering fashion.

\begin{comment}

A chord $\ol{\al\be}$ is called \emph{critical} if $\si_3(\ol\al)=\si_3(\ol\be)$.
Let $\Uf$ be an invariant gap.
It will be convenient to extend the map $\si_3:\Uf'\to \Uf'$ to every edge of $\Uf$ linearly so that a critical edge maps to a point, and a non-critical edge $\ol{\al\be}$ maps to $\ol{(3\al)(3\be)}=\si_3(\al)\si_3(\be)$.
We will keep the notation $\si_3$ for this extension.
The \emph{degree} of an invariant gap $\Uf$ is defined as the number
of edges of $\Uf$ mapping \emph{onto} a non-degenerate edge of $\Uf$.
It is easy to see that the degree of an invariant gap is well-defined.
Degree two gaps are also called \emph{quadratic gaps}.

We measure arc length in $\uc$ so that the total length of the entire circle
is equal to 1.
The length of a chord $\ell$ of $\ol\disk$ is defined as the length of the
shorter circle arc in $\uc$ connecting the endpoints of $\ell$.
A hole of $\Uf$ is called a \emph{major hole} if its length is greater than or equal to $\frac13$. The edge of $\Uf$ connecting the endpoints of a
major hole is called a \emph{major edge}, or simply a \emph{major}, of $\Uf$.

\end{comment}

An invariant gap $\gf$ can have one or two majors \cite{BOPT}.
Any \emph{quadratic} invariant gap $\Uf$ has exactly one major.
It can be shown that any major edge is either critical or periodic.
If the major of $\Uf$ is critical and has no periodic endpoints, then $\Uf$ is said to be of
\emph{regular critical type}.
If the major of $\Uf$ is critical and has a periodic endpoint, then $\Uf$
is said to be of \emph{caterpillar type}.
Finally, if the major of $\Uf$ is periodic, then $\Uf$ is said to be of
\emph{periodic type}.
Since any major edge is either critical or periodic, then any quadratic invariant gap is of one of the three types just introduced.

Quadratic invariant gaps can be generated as follows.
Let $\cf$ be any critical chord.
Set $L(\cf)$ to be the longer closed arc of $\uc$ connecting the endpoints of $\cf$,
and $X(\cf)$ be the set of all points in $L(\cf)$, whose forward orbits stay in $L(\cf)$.
Let $\Uf(\cf)$ be the convex hull of $X(\cf)$ in the plane.

\begin{thm}[\cite{BOPT}]\label{t:invagap}
If $\cf$ is any critical chord, then $\Uf(\cf)$ is a quadratic invariant gap.
If one of the endpoints of $\cf$ is periodic, then $\Uf(\cf)$ is of caterpillar type.
Otherwise, $\Uf(\cf)$ is of regular critical type or of periodic type depending
on whether or not the forward orbits of the endpoints of $\cf$ are contained in $L(\cf)$.
Any quadratic invariant gap is of the form $\Uf(\cf)$ for some $\cf$.
\end{thm}

For any critical chord $\cf$, we let $\Uf_c(\cf)$ denote the convex hull
of all non-isolated points in $\Uf'(\cf)$ (here the subscript $c$ stands for ``clean'').
Then $\Uf_c(\cf)=\Uf(\cf)$ unless $\Uf(\cf)$ is of caterpillar type.
If the gap $\Uf(\cf)$ is of caterpillar type, then
$\Uf_c(\cf)$ is a quadratic invariant gap of periodic type while
the set $\Uf'(\cf)$ has isolated points and is obtained from $\Uf'_c(\cf)$ by adding the
non-periodic endpoint of $\cf$ and countably many iterated preimages of it.
The boundary of the caterpillar gap $\Uf(\cf)$ consists of $\Uf'_c(\cf)$ and
countable concatenations of edges inserted into holes of $\Uf_c(\cf)$.

We also have a description of the parameter space of all quadratic
invariant gaps.
First of all, we need to parameterize critical chords.
This is done by the map taking a point $\ol\theta\in\uc$
to the critical chord $\ol{(\theta+\frac 13)(\theta+\frac 23)}$,
which will be denoted by $\cf_\theta$.
Consider the map $\pi$ from $\uc$ to the set of all quadratic invariant
gaps taking $\ol{\theta}$ to $\Uf(\cf_\theta)$. %$\cf_\theta$.
The parameter picture of quadratic invariant gaps is somewhat similar
to that of the rotation number in an analytic family of circle diffeomorphisms.

\begin{thm}[\cite{BOPT}]
\label{T:ParDynHoles}
The mapping $\pi$ is surjective but not injective.
Non-trivial fibers of $\pi$ are exactly those of quadratic invariant gaps of
periodic type.
The fiber of the invariant quadratic gap of periodic type with major
$\ol{(\theta_1+\frac 13)(\theta_2+\frac 23)}$ is the open arc $(\ol\theta_1,\ol\theta_2)$.
\end{thm}

Nontrivial fibers of $\pi$ are the holes of a certain compact subset $Q\subset\uc$.
The convex hull $\Qf$ of $Q$ in the plane is called the
\emph{Principle Quadratic Parameter Gap}.
Holes of $\Qf$ will play an important role in our description of $\cuc_\la$.
Sometimes, we will also need the mapping $\pi_c$ taking
$\ol \ta$ to the quadratic invariant gap $\Uf_c(\cf_\ta)$.

Now we classify finite (i.e., having finitely many edges) invariant
gaps. Such gaps $\gf$ have also finitely many {\em vertices}, i.e.,
points of $\gf'$. Our classification of gaps $\gf$ mimics Milnor's
classification of hyperbolic components in slices of cubic polynomials
and quadratic rational functions \cite{M, miln93}. Say that $\gf$ is of
\emph{type A} (for \lq\lq adjacent") if $\gf$ has only one major, and
$\gf$ is of \emph{type B} (for \lq\lq bi-transitive") if $\gf$ has two
majors that belong to the same $\si_3$-orbit of edges. Finally, we say
that $\gf$ is of \emph{type D} (for \lq\lq disjoint") if $\gf$ has two
majors, whose $\si_3$-orbits are disjoint (except for common
endpoints). Every finite invariant gap must be of one of these three
types.

Let $A\subset \uc$ be a finite invariant set. Then we associate with it a unique
\emph{rotation number} $p/q$ if there is an orientation preserving
homeomorphism of the circle that conjugates $\si_3|_A$ with the
restriction of the rotation by the angle $p/q$ onto an invariant subset of
$\uc$. We also talk about rotation numbers of periodic angles and
finite invariant gaps. Clearly, not every periodic point/orbit has a
rotation number; also, to have a rotation number $p/q$, a periodic
point/orbit must be of period $q$. On the other hand, every finite
invariant gap has a rotation number by definition.
We will use the following simple count.

\begin{lem}
  \label{l:typeD-count}
  There are $q$ type D finite invariant gaps of rotation number $p/q$.
\end{lem}

Indeed, a type D finite invariant gap $\gf$ of a fixed rotation number is
determined by the number of edges that separate the two majors of $\gf$
in a counterclockwise motion from the major whose hole contains $\ol 0$
to the major whose
hole contains $\ol{\frac 12}$. %This number can be any integer between $0$ and $q-1$.
%The two majors of $\gf$ are numbered as follows: the first major is the
%one, whose hole contains $\ol 0$, and the second major is the one, whose
%hole contains $\ol{\frac 12}$.
The next lemma also easily follows from \cite{BOPT}. Observe that if an
endpoint of a hole $(\ol{\ta_1}, \ol{\ta_2})$ of $\Qf$ is such that
$\ta_1+\frac13$ (resp., $\ta_2+\frac 23$) has a well-defined rotation number
$p/q$, then the convex hull of the entire orbit of $\ol{\ta_1+\frac13\,
\ta_2+\frac23}$ is a finite invariant gap of rotation
number $p/q$.

\begin{lem}\label{l:all-angles}
Suppose that $(\ol{\ta_1}, \ol{\ta_2})$ is a hole of $\Qf$ such that
$\ta_1+\frac13$ and $\ta_2+\frac23$ are of rotation number $p/q$. Then
$\ol{\ta_1+\frac13\, \ta_2+\frac23}$ is a major of a finite invariant gap of type D.
\end{lem}

Theorem~\ref{t:typeD-quad} follows from \cite{BOPT}.
It relates type D finite invariant gaps and quadratic invariant gaps of periodic type.

\begin{thm}
  \label{t:typeD-quad}
  If\, $\Mf$ is a major of a type D finite invariant gap,
  then $\Mf$ is the major of some quadratic invariant gap of periodic type.
\end{thm}

Let us provide more detail concerning the relation between finite
invariant gaps of type D and quadratic invariant gaps; our description
here is based upon \cite{BOPT}. Let $\gf$ be a finite invariant gap of
type D.
Let $\Mf_1$ be a major of $\gf$.
Then $\Mf_1$ defines an invariant quadratic gap $\Uf_1$ of periodic type such that $\Uf_1\cap\uc$ is the set of all points of the circle whose orbits stay at the same side of $\Mf_1$ as $\gf$.
We have $\gf\subset \Uf_1$.
Similarly, the other major $\Mf_2$ of $\gf$ determines another invariant quadratic gap $\Uf_2$.

\subsection{Main results}
\label{ss:mainre} The proof of Theorem A uses methods close to those
used by J. Milnor in \cite{M00}. 

\begin{thmA}[Wakes in $\lambda$-slices]
Fix a complex number $\lambda$ such that $|\lambda|\le 1$.
For every hole $(\theta_1,\theta_2)$ of $\Qf$, the parameter rays $\Rc_\lambda(\theta_1)$
and $\Rc_\lambda(\theta_2)$ land at the same point.
\end{thmA}

Thus the parameter rays $\Rc_\lambda(\theta_1)$ and $\Rc_\lambda(\theta_2)$, together
with their common landing point, divide the plane into two (open) parts.

\begin{dfn}\label{d:wedges}
Let $(\theta_1, \theta_2)$ be a hole of $\Qf$ and
$\Wc_\lambda(\theta_1,\theta_2)$ be the component of $\C\sm
\ol{\Rc_\lambda(\theta_1)\cup \Rc_\lambda(\theta_2)}$ containing the
parameter rays with arguments from $(\theta_1, \theta_2)$. The set
$\Wc_\lambda(\theta_1,\theta_2)$ is called the (parameter) \emph{wake}
(of $\Fc_\la$). The joint landing point of the rays
$\Rc_\lambda(\theta_1)$ and $\Rc_\lambda(\theta_2)$ is called the
\emph{root point} of the parameter wake
$\Wc_\lambda(\theta_1,\theta_2)$. Let the \emph{period} of the
parameter wake $\Wc_\la(\theta_1,\theta_2)$ be
the period of $\theta_1+\frac 13$ under the angle tripling map.
%VT: $\si_3$ acts on $\uc$, and the angle-tripling map on $\R/\Z$.
\end{dfn}

Theorem B gives a dynamical description of parameter wakes. Observe
that by definition for almost all holes $(\theta_1,\theta_2)$ of $\Qf$
the corresponding major $\ol{\ta_1+\frac13 \ta_2+\frac23}$ is in
one-to-one correspondence with $(\theta_1,\theta_2)$. The only
exception is the major $\ol{0 \frac12}$, which corresponds to two holes
of $\Qf$, namely to holes $(\frac16, \frac16)$ and $(\frac23,
\frac56)$. This in turn  is related to the fact that $\ol{0 \frac12}$
as the major of an invariant quadratic gap $\Uf$ does not uniquely define $\Uf$.
A unique quadratic invariant gap with major $\ol{0 \frac12}$ contained in the upper half of the unit disk
(\textbf{a}bove $\ol{0 \frac12}$) is denoted by $\fg_a$ while a unique
quadratic invariant gap with major $\ol{0 \frac12}$ contained in the
lower half of the unit disk (\textbf{b}elow $\ol{0 \frac12}$) is denoted by $\fg_b$.
Both $\fg_a$ and $\fg_b$ have the same major $\ol{0\frac12}$.
Define the set $\Wc_\la'(\ta_1, \ta_2)$ as the wake $\Wc_\la(\ta_1, \ta_2)$ except for the holes $(\frac16, \frac13)$ and
$(\frac23, \frac56)$ for which we set

$$\Wc_\la'\left(\frac16, \frac13\right)=\Wc_\la'\left(\frac23, \frac56\right)=
\Wc_\la\left(\frac16, \frac13\right)\cup\Wc_\la\left(\frac23, \frac56\right).$$

\begin{thmB}\label{t:dynawake}
Fix a hole $I=(\theta_1, \theta_2)$ of $\Qf$.
Then the set $\Wc'_\la(I)$ coincides with the set of polynomials $f$ for which the dynamic rays $R_f(\theta_1+\frac 13)$, $R_f(\theta_2+\frac 23)$ land at the same periodic point that is either repelling for all $f\in\Wc'_\la(I)$ or equals $0$ for all $f\in\Wc'_\la(I)$.
If $f_{root}$ is the root point of $\Wc'_\la(I)$, then the dynamic rays
$R_{f_{root}}(\theta_1+\frac 13)$, $R_{f_{root}}(\theta_2+\frac 23)$
land at the same parabolic periodic point, and $f_{root}$ belongs to
$\cuc_\la$.
\end{thmB}

\begin{comment}

\begin{thmB}\label{t:dynawake}
Fix a hole $(\theta_1, \theta_2)$ of $\Qf$ and a parameter wake
$\Wc_\la(\theta_1,\theta_2)$ of $\Fc_\la$. Then, for every
$f\in\Wc_\la(\theta_1,\theta_2)$, the dynamic rays
$R_f(\theta_1+\frac 13)$, $R_f(\theta_2+\frac 23)$ land at the same
point, which is either periodic and repelling for all
$f\in\Wc_\la(\theta_1,\theta_2)$, or equals $0$ for all
$f\in\Wc_\la(\theta_1,\theta_2)$. If $f_{root}$ is the root point of
the parameter wake $\Wc_\la(\theta_1,\theta_2)$, then the dynamic
rays $R_{f_{root}}(\theta_1+\frac 13)$, $R_{f_{root}}(\theta_2+\frac
23)$ land at the same parabolic periodic point, and $f_{root}$
belongs to $\cuc_\la$.
\end{thmB}

\end{comment}

Define a \emph{limb} of the $\la$-connectedness locus $\Cc_\la$ in
the $\la$-slice $\Fc_\la$ as the intersection of $\Cc_\la$ with a
parameter wake. A compact set $X\subset\C$ is said to be \emph{full} if the
complement $\C\sm X$ is connected. Theorem C describes
some topological properties of
$\Cc_\la$ and $\cuc_\la$. Recall that
$\cuc_\la$ is the set of all polynomials $f\in\Fc_\la$ with
$[f]\in\cu$.

\begin{thmC}
  The set $\cuc_\la$ is disjoint from all parameter wakes, unless $\la=1$.
  The $\la$-connectedness locus $\Cc_\la$ is the union of $\cuc_\la$ and all
  limbs of $\Cc_\la$.
  The set $\cuc_\la$ is a full continuum.
\end{thmC}

The case $\la=1$ is addressed in Theorem \ref{t:la1}.
In this case, the two parameter wakes of period one intersect the set $\cuc_1$; the intersection is described explicitly.
Other parameter wakes are disjoint from $\cuc_1$.

\section{Proof of Theorem A}
In this section, we discuss the geometry of parameter rays in
$\lambda$-slices. In particular, we prove Theorem A. We first recall
Lemma B.1 from \cite{GM} that goes back to Douady and Hubbard
\cite{DH}.

\begin{lem}
 \label{l:rep}
Let $g$ be a polynomial, and $z$ be a repelling periodic point of $g$.
If a smooth ray with rational argument $\theta$ in the dynamical plane of $g$ lands
at $z$, then, for every polynomial $\tilde g$ sufficiently close to
$g$, the ray $\tilde R$ with argument $\theta$ in the dynamical
plane of $\tilde g$ is smooth and lands at a repelling periodic point $\tilde z$
close to $z$. Moreover, $\tilde z$ depends holomorphically on
$\tilde g$.
\end{lem}

\begin{comment}
Lemma~\ref{l:rep} holds for polynomials with disconnected Julia sets
and smooth rays as stated. However, it can be generalized to
non-smooth rays. %\cite{bclos}.
Indeed, suppose that $f\in \Fc_\la$ with $|\la|\le 1$ has a repelling
periodic cutpoint $z_f$ of $J(f)$ so that an external ray with rational
(one-sided) argument $\ta^\pm$ lands at $z_f$. Then any close by
polynomial $g$ (with connected or disconnected Julia set) has a close
by repelling periodic point $z_g$, at which the external ray with the
same one-sided argument lands [REFERENCE??]. In particular, let $\Ta$ be the set of
(one-sided) arguments of all external rays landing at a repelling
periodic point $z_f$. If $\Ta$ is finite, then any close by polynomial
$g$ (with connected or disconnected Julia set) has a close by repelling
periodic point $z_g$ such that the set of (one-sided) arguments of all
external rays of $g$ landing at $z_g$ coincides with $\Ta$.
\end{comment}

Lemma \ref{l:rep} easily implies the following lemma that we will need.

\begin{lem}
\label{l:par}
Let $g_t$, $t\in [0,1]$ be a continuous family of polynomials of the same degree,
$\theta$ be a rational angle, and let $R_t$ be a smooth ray with argument $\theta$ in the dynamical plane of $g_t$.
Denote the point, at which the ray $R_t$ lands, by $z_t$.
Suppose that the points $z_t$ are repelling for $t\in [0,1)$
but the landing point $z_1$ is not the limit of landing points $z_t$ as $t\to 1$.
Then the point $z_1$ is parabolic.
\end{lem}

Lemmas~\ref{l:rep} and ~\ref{l:par} deal with continuity of rays
landing at repelling periodic points. The situation with parabolic
periodic points
%seems to be less clear and
is studied below. The
main objects we consider are repelling petals and rays landing at
parabolic periodic points.

\subsection{Polynomials with parabolic points and their petals}
Let $g$ be a polynomial of arbitrary degree such that $0$ is a fixed parabolic
point of $g$ of multiplier 1.
Suppose that $g(z)=z+az^{q+1}+o(z^{q+1})$, where $q$ is a positive integer and $a\ne 0$.
Recall from \cite{M} that an \emph{attracting vector} for $g$ is defined as
a vector (=complex number) $v$ such that $av^q$ is a negative real number,
i.e., $v$ and $av^{q+1}$ have opposite directions.
Clearly, there are $q$ straight rays consisting of attracting vectors that
divide the plane of complex numbers into $q$ \emph{repelling sectors}.

Consider a repelling sector $S$.
Note that the set $S^{-q}=\{z\in\C\,|\, z^{-q}\in S\}$ is the complement
of the ray $\{-ta\,|\, t>0\}$ in $\C$.
Let $U$ be a sufficiently small disk around $0$.
We will write $F$ for the composition of the function $w\mapsto w^{-1/q}$ mapping
$(S\cap U)^{-q}$ onto $S\cap U$, the function $g$ mapping $S\cap U$ onto $g(S\cap U)$, and the function
$z\mapsto z^{-q}$ mapping $g(S\cap U)$ to $\C$.
We have $F(w)=w-qa+\alpha(w)$, where $\alpha(w)$ denotes a power series in $w^{-1/q}$
that converges in a neighborhood of infinity, and whose free term is zero
(note that this function is single valued and holomorphic on $S^{-q}$).
It follows that there exists a positive real number $r$ with
the property $|\alpha(w)|<\frac{|a|}2$ whenever $|w|>r|a|$.
Consider the half-plane $\Pi$ given by the inequality $\Re(w/a)>r$.
Since this inequality implies that $|w|>r|a|$, we have $F(\Pi)\supset\Pi$,
and also that the shortest distance from a point on the boundary of $\Pi$ to
a point the boundary of $F(\Pi)$ is at least $(q-\frac 12)|a|$.
The preimage of the half-plane $\Pi$ under the map $S\to S^{-q}$,
$z\mapsto z^{-q}$ is called a \emph{repelling petal} of $g$.

Every repelling sector includes a repelling petal; thus, our polynomial $g(z)=z+az^{q+1}+o(z^{q+1})$
has $q$ repelling petals. Hence there are at
least $q$ external rays landing at $0$.
A repelling petal
$P$ of $g$ has the property $g(P)\supset P$. The
dependence of the repelling petals on parameters is described in
Lemma~\ref{l:rp} proved in \cite{bopt14} (the proof follows the same lines as
the proof of Lemma 5 in \cite{BH}).

\begin{lem}
  \label{l:rp}
  Let $g_t(z)=z+a_t z^{q+1}+o(z^{q+1})$ be a continuous family
  of polynomials, in which $a_t\ne 0$, and $t$ belongs to a
  locally compact metric space that is a countable union of compact spaces.
  Then all $q$ repelling petals of $g_t$ can be chosen to vary continuously
  with respect to $t$. 
\end{lem}

\subsection{Stability of rays and their perturbations}
In this subsection, we fix $\lambda=\exp(2\pi ip/q)$ for some
relatively prime $p$ and $q$ (i.e., $\lambda$  is a root of unity). The
ratio $p/q$ is called the {\em rotation number} (of the fixed point
$0$). We discuss conditions that imply that a dynamic ray
$R_f(\theta)$ in the dynamic plane of a polynomial $f\in\Fc_\la$
landing at $0$ is stable (i.e., for $\tilde f\in\Fc_\la$ close to $f$,
the ray $R_{\tilde f}(\theta)$ also lands at $0$).

\begin{prop} [cf Proposition 3.3, \cite{bopt14}]
\label{P:Tpq}
  We have $f_{\la,b}^{\circ q}(z)=z+T_{p/q}(b)z^{q+1}+o(z^{q+1})$,
  where $T_{p/q}(b)$ is a non-zero polynomial in $b$. Moreover, the degree
  of $T_{p/q}$ is
  at most $q$ and if, for some $b$, we have $T_{p/q}(b)=0$, then
  $f_{\la,b}^{\circ q}(z)$ has $2q$ parabolic Fatou domains at $0$ forming two cycles under $f$ as well as $2q$ external rays landing at $0$.
\end{prop}

The representation for $f_{\la,b}^{\circ q}(z)$ and the fact that $T_{p/q}$ is a non-zero polynomial are proved in \cite[Proposition 3.3]{bopt14}.
The claim about the degree follows from Lemma \ref{L:polynom} below.
The last claim in Proposition~\ref{P:Tpq} follows from
\cite{bea91} (see the Petal Theorem 6.5.4 and Theorem 6.5.8) and the fact that our maps are cubic.

\begin{lem}
  \label{L:polynom}
  Let $\mathbb{V}$ be the vector space of all polynomials in $b$ and $z$ given by
  $
  p(b,z)=\sum_{n=1}^N a_n(b)z^n
  $,
  where $a_n$ is a polynomial of degree $\le n-1$.
  Then $f^{\circ r}_{\la,b}(p)\in\mathbb{V}$ for each $p\in\mathbb{V}$ and each integer $r>0$.
\end{lem}

\begin{proof}
Let $p(b,z)=a_1(b)z+\dots+a_N(b)z^N$ be a polynomial such that
the degree of $a_j(b)$ is at most $j-1$, for every $j$.
It follows that $p(b,z)^k=c_k(b)z^k+\dots+c_{Nk}(b)z^{Nk}$,
where the degree of each $c_i(b)$ is at most $i-k$ for each $i\ge k$.
Thus, $f_{\la, b}(p(b,z))=\lambda p(b, z) + b p^2(b,z) + p^3(b,z)$
can be written as $d_1(b)z+\dots+d_{3N}(b)z^{3N}$,
where the degree of each $d_i(b)$ is at most
$i-1$. The result follows.
\end{proof}

Proposition \ref{P:ray0stable} deals with rays landing at parabolic points.

\begin{prop}[\cite{bopt14} Proposition 3.4]
\label{P:ray0stable}
  Suppose that an external ray $R_{f_{\la,b_*}}(\theta)$ with periodic argument $\theta\in\R/\Z$
  lands at $0$, and $T_{p/q}(b_*)\ne 0$.
  Then, for all $b$ sufficiently close to $b_*$,
  the ray $R_{f_{\la,b}}(\theta)$ lands at $0$.
\end{prop}

\subsection{External and quadratic-like rays of polynomials in
$\Fc_\la\sm\Cc_\la$}\label{ss:exteql} In this subsection, $\lambda$
is a complex number of modulus at most 1. We study the parameter
plane $\Fc_\la$. The set $\Fc_\la\sm\Cc_\la$ is foliated by
parameter rays. Let us describe the dynamics of a polynomial
$f\in\Fc_\la\sm\Cc_\la$ choosing $f$ from the parameter ray
$\Rc_\la(\vk)$. We will establish a correspondence between the
dynamical properties of $f$ and the properties of the invariant
quadratic gap $\Uf_c(\vk)$. Let us emphasize that in what follows we
deal with both \emph{external} rays, \emph{quadratic-like} rays (defined in
Subsection~\ref{ss:immeren}), and \emph{dynamic} rays (defined in
Subsection~\ref{ss:strufc}), so it is important to distinguish
between these types of rays.

A polynomial $f\in\Fc_\la$ belongs to a \emph{parameter ray}
$\Rc_\la(\vk)$ if and only if the dynamic rays $R_f(\vk+\frac13)$ and
$R_f(\vk+\frac23)$ crash into $\omega_2(f)$; by Theorem~\ref{T:BH} then
no other dynamic ray crashes into $\omega_2(f)$. Set
$R_f(\vk+\frac13)\cup R_f(\vk+\frac23)\cup \omega_2(f)=\Gamma_f$;
clearly, $\Gamma_f$ is a curve dividing $\C$, the dynamic plane of $f$,
in two parts. Let $\Sigma_f=W_f(\vk+\frac 13, \vk+\frac 23)$ be the
part of the plane bounded by $\Gamma_f$ and containing the rays $R_f(\theta)$ for all $\theta\in(\vk+\frac 13, \vk+\frac 23)$; we call such sets $\Sigma_f$
(enclosed by two dynamic rays of $f$ landing or crashing at the same
point) \emph{dynamic wedges}. In general, a \emph{dynamic wedge}
$W_f(\al, \be)$ is a part of the plane bounded by $\ol{R_f(\al)\cup R_f(\be)}$ and containing the rays $R_f(\theta)$ for all $\theta\in (\al, \be)$,
where the dynamic rays $R_f(\al)$ and $R_f(\be)$ crash or land at the same point.

Then $\Sigma_f$ maps one-to-one onto the complement of $f(\Gamma_f)$ and
contains the dynamic rays with arguments in $\uc\sm L(\vk)$, where
$L(\vk)$ is the longer closed arc with endpoints $\vk+\frac 13$ and
$\vk+\frac 23$. By Theorem~\ref{t:unboren}, the polynomial $f$ is
immediately renormalizable. The quadratic-like filled Julia set
$K(f^*)$ is disjoint from $\Gamma_f$, since every point of $\Gamma_f$
is escaping. Clearly, $K(f^*)\not\subset \Sigma_f$ because
$f|_{\Sigma_f}$ is one-to-one while $f|_{K(f^*)}$ is generically
two-to-one. Hence, $K(f^*)\cap \overline \Sigma_f=\emptyset$. In this
subsection, we fix a polynomial $f\in \Fc_\la\sm\Cc_\la$ that belongs
to the parameter ray $\Rc_\la(\vk)$.

We need more information about the Jordan domains $U^*$ and $V^*$, for
which the map $f^*:U^*\to V^*$, the restriction of $f$ to $U^*$, is a
quadratic-like map, cf. \cite{BrHu}. The domain $V^*$ can be defined as
the connected component of the set of all point $z\in \C$ such that
$G_{K(f)}(z)<G_{K(f)}(\om_2(f))$ containing $K(f^*)$, and the domain
$U^*$ is defined as the $f$-pullback of $V^*$ in $V^*$. In particular,
$V^*$ is disjoint from $\Gamma_f$.

\begin{lem}
\label{l:not-accum}
  A dynamic ray $R_f(\al)$ does not accumulate in $K(f^*)$ if and only if there is an
  integer $k\ge 0$ such that $R_f(3^k\al)\subset \ol{\Sigma}_f$.
  In particular, a smooth ray $R_f(\al)$ does not accumulate in $K(f^*)$ if and only if there is an
  integer $k\ge 0$ such that $R_f(3^k\al)\subset \Sigma_f$.
\end{lem}

\begin{proof}
By definition, a dynamic ray $R_f(\al)$ accumulates inside $V^*$ if
and only if it penetrates $V^*$. On the other hand, $R_f(\al)$
accumulates in $K(f^*)$ if and only if all its images accumulate in
$U^*$. We claim that this is equivalent to all its images
accumulating in $V^*$. Let us show that if all images of
$R_f(\al)$ accumulate in $V^*$, then they all accumulate in $U^*$.
Indeed, otherwise there must exist an image of the ray accumulating
in $V^*$ but not in $U^*$. This would imply that the next image of
the ray would be contained in $\ol{\Sigma}_f$, a contradiction. Thus,
$R_f(\al)$ accumulates in $K(f^*)$ if and only if all its images
accumulate in $V^*$. Hence $R_f(\al)$ does not accumulate in
$K(f^*)$ if and only if there is an integer $k\ge 0$ such that
$R_f(3^k\al)$ is disjoint from $V^*$. Since rays with arguments from
$(\vk+\frac 23, \vk+\frac 13)$ penetrate $V^*$, then $3^k\al\in
[\vk+\frac 13, \vk+\frac 23]$, as desired. The case of a smooth ray
now follows because smooth rays cannot pass through $\omega_2$ and
hence cannot intersect the boundary of $\Sigma_f$.
\end{proof}

We want to describe \emph{external} rays that accumulate in $K(f^*)$.
Recall that external rays $R^e$ have one-sided arguments. If $R^e$ is a
smooth ray with argument $\ta$, then it is associated with both
one-sided external arguments $\ta^+$ and $\ta^-$. If $R^e$ is not
smooth, then $R^e$ is the one-sided limit of smooth rays from exactly
one side, and we associate to $R^e$ the appropriate \emph{one-sided}
external argument, $\theta^+$ or $\theta^-$. In what follows, we write
$R^e_f(\theta^\tau)$, where $\tau=+$ or $\tau=-$. For any set
$A\subset\uc$ and an angle $\alpha\in\R/\Z$, say that $\alpha^+$
(respectively, $\alpha^-$) is a \emph{one-sided argument $($of $\al)$} in
$A$ if $\alpha\in A$, and $\alpha$ is not isolated in $A$ from the positive
side (resp., from the negative side).
Recall that $\ol{(\vk+\frac13)(\vk+\frac23)}$ is denoted by $\cf_\vk$.

\begin{prop}\label{P:no-accum-K_b}
Consider a polynomial $f\in\Rc_\la(\varkappa)$ and its immediate
renormalization $f^*:U^*\to V^*$. Then the set of arguments of
external rays to $K(f^*)$ coincides with the union of
the set of arguments $\al^\pm$ where $\al\in \Uf'_c(\cf_\vk)$ never
maps to an endpoint of $\cf_\vk$ $($these external rays are smooth$)$
and the set of one-sided
arguments $\be^\tau$, where $\be$ eventually maps to an endpoint
of $\cf_\vk$, and $\be^\tau$ is the one-sided argument of $\be$ in
$\Uf'_c(\cf_\vk)$ $($these external rays are not smooth$)$.
\end{prop}

In particular, the (one-sided) argument of an external ray accumulating in
$K(f^*)$ always belongs to $\Uf'_c(\cf_\vk)$.

\begin{proof}
Smooth rays are exactly the rays with arguments that never map to the
endpoints of $\cf_\vk$. Moreover, smooth rays are both external and
dynamic. Thus, by Lemma~\ref{l:not-accum}, a smooth ray $R^e_f(\al)$
accumulates in $K(f^*)$ if and only if, for any $k\ge 0$, the ray
$R_f(3^k\al)$ is disjoint from $\ol{\Sigma}_f$. This is equivalent to the
fact that $3^k\al$ belongs to the interior of $L(\cf_\vk)$ for any
integer $k\ge 0$, which, in turn, is equivalent to the fact that
$\al\in \Uf'_c(\cf_\vk)$ and never maps to the endpoints of $\cf_\vk$.

Let $R^e=R^e(\al^\tau)$ be a non-smooth external ray. Since $R^e$
is non-smooth if and only if $\si_3^{\circ k}(\ol\al)$ is an endpoint of $\cf_\vk$
for the least integer $k\ge 0$, then, if all points $\ol{3^m\al}$ are in the interior of
$L(\vk)$ for $m<k$, and $R^e((3^k\al)^\tau)=f^{\circ k}(R^e)$ accumulates in $K(f^*)$, the
orbit of the principal set of $R^e$ is disjoint from $\ol{\Sigma}_f$, and
$R^e$ accumulates in $K(f^*)$. So, it suffices to consider rays with
one-sided arguments $\vk+\frac13^\pm$, $\vk+\frac23^\pm$ and to choose
those of them, that accumulate in $K(f^*)$. A simple analysis shows
that $R^e(\al^\tau)$ (where $\ol\al$ is an endpoint of $\cf_\vk$)
accumulates in $K(f^*)$ if and only if $\al^\tau$ is a one-sided
argument in $\Uf'_c(\cf_\vk)$. This completes the proof. Observe that
in the regular critical case both endpoints of $\cf_\vk$ have one-sided
argument in $\Uf'_c(\cf_\vk)$, in the periodic case neither endpoint of
$\cf_\vk$ has one-sided arguments in $\Uf'_c(\cf_\vk)$, and, in the
caterpillar case only the periodic endpoint of $\cf_\vk$ has one-sided
argument in $\Uf'_c(\cf_\vk)$.
\end{proof}

We defined the \emph{correspondence} between external rays and
quadratic-like rays in Definition~\ref{d:corres}.
\begin{comment}
By Proposition \ref{P:no-accum-K_b} and Theorem \ref{t-extepoly}, there is
a correspondence $\xi$ between the external rays $R^e_f(\theta^\tau)$, where
$\ta^\tau$ is a one-sided argument of $\Uf'_c(\cf_\vk)$,
and quadratic-like rays to $K(f^*)$.
\end{comment}

\begin{prop}
\label{p:land-endholes}
For any hole $(\ol\al,\ol\be)$ of $\Uf_c(\cf_\vk)$, the external rays
$R^e_f(\al^-)$ and $R^e_f(\be^+)$ correspond to the same quadratic-like ray to
$K(f^*)$. More precisely, there are two cases:

\begin{enumerate}
\item the gap $\Uf_c(\cf_\vk)$ is of regular critical type, the external rays
$R^e_f(\al^-)$ and $R^e_f(\be^+)$ meet at an eventual preimage $x$ of $\omega_2(f)$, and
then continue along a joint arc from $x$ to $K(f^*)$;

\item the gap $\Uf_c(\cf_\vk)$ is of periodic type with major hole $(\ol\al_*,\ol\be_*)$,
the external rays $R^e_f(\al^-)$ and $R^e_f(\be^+)$ both land at an
eventual preimage of the common periodic landing point for the rays
$R^e_f(\al^{-}_*)$ and $R_f^e(\be^{+}_*)$.
\end{enumerate}
\end{prop}

\begin{proof}
Suppose that the external rays $R^e_f(\al^-)$ and $R^e_f(\be^+)$ do not
correspond to the same quadratic-like ray to $K(f^*)$. By
Theorem~\ref{t-extepoly}, these rays correspond to distinct
quadratic-like rays, say, $T_\al\ne T_\be$, to $K(f^*)$. Then there are
quadratic-like rays $T$ in either of the two components $U^*_1$,
$U^*_2$ of $U^*\sm(T_\al\cup T_\be\cup K(f^*))$. Clearly, both
components $U^*_1$ and $U^*_2$ contain segments of quadratic-like rays
to $K(f^*)$. Therefore, by Theorem~\ref{t-extepoly}, both components
contain segments of external rays to $K(f^*)$. However, by Proposition
\ref{P:no-accum-K_b}, there are no external rays to $K(f^*)$ with
arguments in $(\al,\be)$, a contradiction. Hence, $R^e_f(\al^-)$ and
$R^e_f(\be^+)$ correspond to the same quadratic-like ray $T$ to
$K(f^*)$.

Suppose that $\Uf_c(\cf_\vk)$ has major $\Mf=\ol{\al_*\be_*}$.
First assume that it is of regular critical type.
%Then the dynamic rays $R_f(\al_*^-)$, $R_f(\be_*^+)$ crash at
%$\omega_2(f)$ and then, as follows from \cite[Lemma 6.1]{bclos},
%extend together towards $K(f^*)$.
Then the external rays $R^e_f(\al_*^-)$, $R^e_f(\be_*^+)$ meet at
$\omega_2(f)$ and then, as follows from \cite[Lemma 6.1]{bclos},
extend together towards $K(f^*)$.
If $\ol{\al\be}$ is an edge of $\Uf_c(\cf_\vk)$, then $\ol{\al\be}$ is an
iterated $\si_3$-pullback of $\Mf$, so that $\al$ and $\be$ are
appropriate preimages of $\al_*$ and $\be_*$,
and the %dynamic
external rays $R^e_f(\al^-)$, $R^e_f(\be^+)$ %crash
meet at a preimage $x$
of $\omega_2(f)$ and then extend together towards $K(f^*)$.
In fact, their union is a pullback of the union of rays $R^e_f(\al_*^-)$,
$R^e_f(\be_*^+)$.
This covers case (1) of the proposition and corresponds to case (i)
of Theorem~\ref{t-extepoly}.
Assume now that $\Uf_c(\cf_\vk)$ is of periodic type.
Then, by \cite[Lemma 6.1]{bclos}, the rays $R^e_f(\al_*^-)$,
$R^e_b(\be_*^+)$ cannot intersect, which implies that they have to land at
the same periodic point of $K(f^*)$. Since all edges of $\Uf_c(\cf_\vk)$ are
preimages of $\ol{\al_*\be_*}$, claim (2) follows.

There are two distinct cases within claim (2). If $\Uf(\vk)=\Uf_c(\vk)$
is of periodic type, the rays $R^e_f(\al_*)$ and $R^e_f(\be_*)$ are
smooth, disjoint, and land at the same periodic point. The situation is
a little more complicated if $\Uf(\vk)$ is of caterpillar type. In that
case, we may assume that $\al_*=\vk+\frac13$.
Then the dynamic ray %$R_f(\al_*^-)$
$R_f(\al_*)$ crashes (together with the dynamic ray
%$R_f(\al_*+\frac13)^+$)
$R_f(\al_*+\frac13)$) at the escaping critical point $\omega_2(f)$.
However, in this case, the external non-smooth ray $R^e_f(\al_*^-)$ still
lands at the same periodic point as the smooth external ray $R^e_f(\be_*^+)$, and both rays
correspond to the same quadratic-like ray to $K(f^*)$.
\end{proof}

Let $f\in\Fc_\la$ with $|\la|\le 1$ be immediately renormalizable.
We say that a chord $\ol{\al\be}$ \emph{generates a cut} if there are signs $\rho$ and $\tau$, each of which equals $+$ or $-$, such that the external rays $R^e(\al^\rho)$ and
$R^e(\be^\tau)$ have intersecting principal sets in $K(f^*)$.
The \emph{cut $($generated by the chord $\ol{\al\be})$} is then defined as the smallest connected closed subset of $\ol{R^e_f(\al^\rho)\cup R^e_f(\be^\tau)}$ containing
the union of \emph{dynamic} rays $R_f(\al)\cup R_f(\be)$.
We claim that if $\ol{\al\be}$ generates a cut, then this cut is unique.
In other words, if the cut $\Gamma$ generated by $\ol{\al\be}$ exists for some specific choice of the signs $\rho$ and $\tau$, then, for every other choice of the signs, the corresponding cut either does not exist or coincides with $\Gamma$.
This is a straightforward consequence of the following observation: if
$R^e(\al^+)\ne R^e(\al^-)$, then at most one of these two external rays has the principal set in $K(f^*)$; similarly for $R^e(\be^+)$ and $R^e(\be^-)$.

Denote the cut generated by $\ol{\al\be}$ by $\Ga(\ol{\al\be})$.
If the (one-sided) rays defining a cut $\Ga$ land at the same point $z$, then $z$ is called the \emph{vertex $($of $\Ga)$}.

Suppose now that $f$ lies in the parameter ray $\Rc_\la(\vk)$.
By Proposition~\ref{p:land-endholes}, there are cuts generated
by edges of $\Uf_c(\cf_\vk)$. In the regular critical case,
$\Ga(\ol{\al\be})$ is formed by dynamic rays $R_f(\al)$ and $R_f(\be)$
crashing into the same (pre-)critical point. Clearly, the cut $\Ga(\ol{\al\be})$ is disjoint from $K(f^*)$.
This corresponds to case (1) of Proposition~\ref{p:land-endholes}. In
the periodic and caterpillar cases the external rays $R^e_f(\al^-)$,
$R^e_f(\be^+)$ are disjoint but land at the same (pre)periodic point.
This corresponds to case (2) of Proposition~\ref{p:land-endholes}.
Observe that, in the caterpillar case, exactly one of the rays forming a
cut is non-smooth. %In general, if $\Ga(\ol{\al\be})$ exists but the
%external rays $R^e_f(\al^\pm)$, $R^e_f(\be^\pm)$ are disjoint, then
%$\Ga(\ol{\al\be})=\ol{R^e_f(\al^\pm)\cup R^e_f(\be^\pm)}$.
Clearly,
$\Ga(\ol{\al\be})$ separates all dynamic rays, whose arguments belong
to $(\al,\be)$, from all dynamic rays, whose arguments belong to
$(\be,\al)$.

By Proposition~\ref{P:no-accum-K_b}, the (one-sided) argument of an
external ray accumulating in $K(f^*)$ belongs to $\Uf'_c(\cf_\vk)$. No
chord connecting two points of $\Uf'_c(\cf_\vk)$ can cross an edge of
$\Uf_c(\cf_\vk)$ (say that two chords of $\ol\disk$ \emph{cross} if
they intersect in $\disk$ and do not coincide). This yields Lemma~\ref{l:cuts-cross}.

\begin{lem}
  \label{l:cuts-cross}
  Suppose that a polynomial $f$ lies in $\Rc_\la(\vk)$ with $|\la|\le 1$.
  Each edge of $\Uf_c(\cf_\vk)$ generates a cut in the dynamical plane of $f$ consisting of two rays that accumulate either on a $($pre$)$critical point $($in the regular critical case$)$ or on a $($pre$)$periodic point $($in the periodic and caterpillar cases$)$.
  In either case, the corresponding external rays correspond to the same quadratic-like ray.
  If a chord of $\disk$ generates a cut and crosses an edge of $\Uf_c(\cf_\vk)$ in $\disk$, then it coincides with this edge.
\end{lem}

The following is a partial converse of Proposition
\ref{p:land-endholes}.

\begin{prop}
  \label{p:same-extray}
  Suppose that external rays $R^e_f(\al^\epsilon)$ and $R^e_f(\be^\delta)$ correspond
  to the same quadratic-like ray to $K(f^*)$ $($here $\epsilon,\delta\in \{-, +\})$.
  Then $\ol{\al\be}$ is an edge of $\Uf_c(\cf_\vk)$.
\end{prop}

\begin{proof}
Suppose not. Let $\psi:\C\sm K(f^*)\to \C\sm \cdisk$ be the Riemann map with $\psi(z)\sim kz$ as $z\to \infty$, for some $k>0$.
Then, by definition, the rays $\psi(R^e_f(\al^\epsilon))$ and
$\psi(R^e_f(\be^\delta))$ land at the same point $z\in \uc$.
We may assume that the open wedge in the positive direction from $\psi(R^e_f(\al^\epsilon))$ to $\psi(R^e_f(\be^\delta))$ contains no points of $\uc$; otherwise we simply interchange $\al^\epsilon$ and $\be^\delta$.
Since $\ol{\al\be}$ is not an edge of $\Uf_c(\cf_\vk)$, there are infinitely many external rays to $K(f^*)$ with arguments in $(\al,\be)$.
It follows that, for any such ray $R^e$, the ray $\psi(R^e)$ can only land at $z$.
By Theorem~\ref{t-extepoly}, this implies that there are infinitely many quadratic-like rays $R^*$ such that $\psi(R^*)$ lands at $z$.
Since in fact there exist only finitely many such quadratic-like rays, we obtain a contradiction.
\end{proof}

\subsection{Landing properties}
In this subsection, we fix a hole $(\ol\theta_1,\ol\theta_2)$ of $\Qf$,
%and discuss when the dynamic rays $R_f(\theta_1+\frac 13)$ and
%$R_f(\theta_2+\frac 23)$ land at the same point.
%More precisely, we
consider a polynomial $f\in \Rc_\la(\vk)\subset \Fc_\la\sm\Cc_\la$
with $|\la|\le 1$, and study the mutual location
of the point $\ol\vk$ and the hole
$(\ol\theta_1,\ol\theta_2)$, under which the \emph{dynamic} rays
$R_f(\theta_1+\frac 13)$ and $R_f(\theta_2+\frac 23)$
can have a common landing \emph{point} in $K(f^*)$.
Note that, by Theorem \ref{T:ParDynHoles}, the hole $(\ol\theta_1,\ol\theta_2)$
defines an invariant quadratic gap $\Uf$ with periodic major
$\Mf=\ol{(\theta_1+\frac 13)(\theta_2+\frac 23)}$. We refer the reader to Subsection~\ref{ss:overview}
for the notation and the main notions and concepts we deal with in
this subsection.

\begin{lem}\label{l:land-in-k}
Suppose that the dynamic rays $R_f(\theta_1+\frac 13)$ and
$R_f(\theta_2+\frac 23)$ are contained in external rays $R^e_1$,
$R^e_2$, respectively, with the same landing point $z$. Then $z\in
K(f^*)$. Moreover, if neither $(\ta_1, \ta_2)=(\frac23, \frac56)$ nor
$(\ta_1, \ta_2)=(\frac16, \frac13)$, then the following are equivalent:

\begin{enumerate}

\item $\vk\notin [\theta_1, \theta_2]$;

\item the plane cut $R^e_1\cup \{z\}\cup R^e_2$ separates $K(f^*)$.

\end{enumerate}

\end{lem}

\begin{proof}
Since $f\in \Rc_\la(\vk)\subset \Fc_\la\sm\Cc_\la$, the dynamic rays
$R_f(\vk+\frac 13)$ and $R_f(\vk+\frac 23)$ crash at the critical point
$\om_2(f)$ and cut the plane in two wedges. One of them ($W_1$)
contains dynamic rays with arguments from $(\vk+\frac 13, \vk+\frac
23)$ and the other one ($W_2$) contains dynamic rays with arguments
from $(\vk+\frac 23, \vk+\frac 13)$. It follows that $K(f^*)\subset
W_2$. Observe also that both arcs $(\theta_1+\frac 13, \theta_2+\frac
23)$ and $(\theta_1+\frac 23, \theta_2+\frac 13)$ are longer than
$\frac13$.
Hence both angles $\ta_1+\frac13$, $\ta_2+\frac 23$ are in the arc $[\vk+\frac23, \vk+\frac13]$, the rays $R^e_1$ and $R^e_2$ are contained in $\ol{W}_2$, one of these rays lies in $W_2$, and the point $z$ belongs to $W_2$.

Let $W_{(1)}$, $W_{(2)}$ be the two wedges defined by
$R^e_1$, $R^e_2$. If $K(f^*)$ meets both these wedges, then $z\in K(f^*)$
as desired. Moreover, then $\vk\notin [\theta_1, \theta_2]$ because
otherwise $\Uf'_c(\cf_\vk)\subset [\theta_2+\frac23,
\theta_1+\frac13]$, a contradiction with the fact that $K(f^*)$ meets
both wedges defined by $R^e_1$, $R^e_2$.

Now assume that $K(f^*)\subset W_{(1)}\cup \{z\}$. We claim that
then $\vk\in [\theta_1, \theta_2]$.
Indeed, suppose otherwise. Then $K(f^*)$ is contained in $W_2$ but is
disjoint from $W_{(2)}$; hence $\si_3$ is one-to-one on
arguments of external rays landing in $K(f^*)$ except possibly for
$\vk+\frac13$, $\vk+\frac23$, a contradiction. Thus, the endpoints of
$\cf_\vk$ belong to $[\theta_1+\frac13, \theta_2+\frac23]$, and at
least one of the endpoints of $\cf_\vk$ belongs to $(\theta_1+\frac13,
\theta_2+\frac23)$. Hence, at least one of the points
$\theta_1+\frac13$, $\theta_2+\frac23$ (for definiteness suppose that
this is $\theta_2+\frac23$) belongs to $(\vk+\frac23, \vk+\frac13)$. By
Theorem~\ref{T:ParDynHoles}, the chord $\ol{(\theta_1+\frac13)
(\theta_2+\frac23)}$ is the major of an invariant quadratic gap of
periodic type whose intersection with $\uc$ is contained in
$[\ol{\theta_2+\frac23},\ol{\theta_1+\frac13}]$. Hence the orbit of
$\theta_2+\frac23$ is completely contained in the arc $[\vk+\frac23,
\vk+\frac13]$, which implies that the landing point $z$ of $R^e_2$
belongs to $K(f^*)$, as desired.
\end{proof}

Observe that if $(\ta_1, \ta_2)=(\frac23, \frac56)$ or $(\ta_1,
\ta_2)=(\frac16, \frac13)$, then the plane cut $R^e_1\cup \{z\}\cup
R^e_2$ defined in Lemma \ref{l:land-in-k} cannot separate $K(f^*)$.

The next lemma specifies Lemma \ref{l:land-in-k}.

\begin{lem}\label{l:landvk}
The following claims are equivalent.

\begin{enumerate}

\item The dynamic rays $R_f(\theta_1+\frac 13)$ and $R_f(\theta_2+\frac 23)$
are contained in a pair of external rays with the same periodic landing point.

\item One of the following holds.

\begin{enumerate}

\item The angle $\vk$ is in $[\theta_1, \theta_2]$.

\item The chord $\ol{(\theta_1+\frac 13)(\theta_2+\frac 23)}$
    is a major edge of a type D invariant gap or leaf $\gf$
    while $[\ol{\vk+\frac13},\ol{\vk+\frac23}]$ is contained in
    the closure of the hole behind the other major edge
    %$\ol{\al\be}$
    of $\gf$. Moreover, for every $\ol{\al\be}$ in the
    $\si_3$-orbit of $\ol{(\theta_1+\frac 13)(\ta_2+\frac 23)}$, the rays $R^e_f(\al^+)$ and
    $R^e_f(\be^-)$ land at a fixed point $z$; we have $z=0$ unless $(\ta_1,
    \ta_2)=(\frac23, \frac56)$ or $(\frac16, \frac23)$.
\end{enumerate}

\end{enumerate}

\end{lem}

Except for the case when $\gf=\ol{0 \frac12}$, the set $\gf$ is a gap.

\begin{proof}
We first prove that (2) implies (1). Assume that $\vk\in
[\theta_1,\theta_2]$. Then, by Proposition \ref{p:land-endholes}, the rays
$R^e_f(\theta_1+\frac 13)^-$ and $R^e_f(\theta_2+\frac 23)^+$ land at a
common point, as desired. If (b) holds, the claim follows immediately.

Assume now that (1) holds. Denote the external rays containing
$R_f(\theta_1+\frac 13)$ and $R_f(\theta_2+\frac 23)$ and landing at a
common point $z$ by $R^e_1$, $R^e_2$; then, by Lemma \ref{l:land-in-k}, we have
$z\in K(f^*)$. Suppose that $\vk\not\in [\theta_1,\theta_2]$, and prove
that (b) holds.

To begin with, consider the case when $(\ta_1, \ta_2)=(\frac23,
\frac56)$ so that $(\theta_1+\frac 13, \theta_2+\frac 23)=(0,
\frac12)$.
In this case $\gf=\ol{0 \frac12}$.
This is a leaf but it can be informally viewed as a gap with two distinct edges
$\ol{0 \frac12}$ and $\ol{\frac12 0}$.
Our assumption that $\vk\not\in [\theta_1,\theta_2]$ means that
$\vk\not\in [\frac23, \frac56]$. Since the critical cut formed by the
rays $R_f(\vk+\frac 13)$ and $R_f(\vk+\frac 23)$ cannot cross the cut
formed by the rays $R^e_1$ and $R^e_2$, then, as in the proof of
Lemma \ref{l:land-in-k}, it follows that $[\vk+\frac13, \vk+\frac23]\subset
[\frac12, 0]$. In other words, $[\ol{\vk+\frac13},\ol{\vk+\frac23}]$ is
contained in the closure of the hole behind the other major edge
$\ol{\frac12 0}$ of $\gf$ as desired. Moreover, by
Proposition \ref{p:land-endholes}, then $R^e_f(\frac 12)^-$ and $R^e_f(0)^+$ land
at the same fixed point as desired. The same can be proven in the case
when $(\ta_1, \ta_2)=(\frac16, \frac13)$.

Thus, from now on we may assume that $(\ta_1, \ta_2)\ne (\frac23,
\frac56)$ and $(\ta_1, \ta_2)\ne (\frac16, \frac13)$. Recall that we
suppose that $\vk\not\in [\theta_1,\theta_2]$, and want to prove that
then (b) holds. By Theorem~\ref{t-extepoly}, external rays $R^e_1$ and
$R^e_2$ correspond to quadratic-like rays. If they correspond to the
same quadratic-like ray, then the cut $R^e_1\cup \{z\}\cup R^e_2$ does
not separate $K(f^*)$. By Lemma \ref{l:land-in-k}, then $\vk\in
[\theta_1,\theta_2]$, a contradiction. Hence the external rays $R^e_1$
and $R^e_2$ correspond to different quadratic-like rays to $K(f^*)$.

By Lemma \ref{l:cuts-cross}, the chord $\Mf=\ol{(\theta_1+\frac
13)(\theta_2+\frac 23)}$ lies inside $\Uf_c(\cf_\vk)$, except for the
endpoints. Since the arguments of both rays are periodic, $z$ is
periodic. Moreover, since the external rays $R^e_1$ and $R^e_2$
correspond to different quadratic-like rays, $y$ is a cutpoint of
$K(f^*)$. Since $|\la|\le 1$, there are no periodic cutpoints of
$K(f^*)$ that are not fixed, and the fixed cutpoint $0$ must be
parabolic. Therefore, $z=0$ is parabolic. Since $\Mf$ is a major of
some quadratic invariant gap $\Uf$, the orbit of $\Mf$ consists of
pairwise disjoint chords; since $z$ is parabolic, $\si_3$ restricted to
them preserves their circular order. Completing the orbit of $\Mf$ to
its convex hull $\gf$, we see that $\gf$ is an invariant finite gap of
type D. Hence $[\ol{\vk+\frac13},\ol{\vk+\frac23}]$ is contained in the
closure of the hole behind a major edge of $\gf$; by definition,
$\Mf$ must be the other major edge of $\gf$. The remaining claim
that for every $\ol{\al\be}$ in the orbit of $\ol{(\theta_1+\frac 13)(\ta_2+\frac 23)}$
the rays $R^e_f(\al^+)$ and $R^e_f(\be^-)$ land at the same point follows
from Proposition \ref{p:land-endholes}.
\end{proof}

\subsection{Parameter wakes}
\label{ss:wakes} 
In Section \ref{ss:wakes}, we suppose that
$(\theta_1,\theta_2)$ is a hole of $\Qf$. Fix $\la\in\C$ with $|\la|\le
1$. Define the set $\tWc_\la(\theta_1,\theta_2)$ as the set of all
polynomials such that the \emph{dynamic} rays $R_f(\theta_1+\frac 13)$
and $R_f(\theta_2+\frac 23)$ land at the same \emph{repelling} periodic
point of $f$ (observe that then both rays are smooth). By
Lemma~\ref{l:rep}, the set $\tWc_\la(\theta_1,
\theta_2)$ is open.

Recall that, by Theorem A, for every hole $(\theta_1,\theta_2)$ of $\Qf$,
the parameter rays $\Rc_\lambda(\theta_1)$ and $\Rc_\lambda(\theta_2)$
land at the same point. Thus the parameter rays $\Rc_\lambda(\theta_1)$
and $\Rc_\lambda(\theta_2)$, together with their common landing point,
divide the plane into two (open) parts. In Definition \ref{d:wedges} we consider
a hole $(\theta_1, \theta_2)$ of $\Qf$ and denote by
$\Wc_\lambda(\theta_1,\theta_2)$ the component of $\C\sm
\ol{\Rc_\lambda(\theta_1)\cup \Rc_\lambda(\theta_2)}$ containing the
parameter rays with arguments from $(\theta_1, \theta_2)$. The set
$\Wc_\lambda(\theta_1,\theta_2)$ is called the (parameter) \emph{wake}
(of $\Fc_\la$). The joint landing point $f_{root}$ of the rays
$\Rc_\lambda(\theta_1)$ and $\Rc_\lambda(\theta_2)$ is called the
\emph{root point} of the parameter wake
$\Wc_\lambda(\theta_1,\theta_2)$. Let the \emph{period} of the
parameter wake $\Wc_\la(\theta_1,\theta_2)$ be the period of
$\theta_1+\frac 13$ under $\si_3$.

A parabolic periodic point $z$ of $f_{root}$ is not necessarily equal
to $0$; also, if $z\ne 0$, then the multiplier at $z$ is $1$. Recall
that, by Proposition \ref{P:Tpq}, if $\lambda=e^{2\pi i p/q}$ and $T_{p/q}(b)=0$,
then $f^{\circ q}_{\la, b}$ has $2q$ parabolic Fatou domains at $0$
forming two cycles under $f_{\la, b}$.

We also need an observation concerning external and dynamic rays.
Namely, if an argument $\al$ is periodic, then either
$R_f(\al)=R_f^e(\al)$ is a smooth ray, or the dynamic ray $R_f(\al)$ with argument $\al$ crashes at a \emph{critical} point.
In the latter case, the two one-sided external rays $R_f^e(\al^-)$
and $R_f^e(\al^+)$ extend the ray $R_f(\al)$ and contain infinitely
many precritical points. We will need this observation when talking
about the boundary of the set $\tWc_\la(\theta_1,\theta_2)$.

Recall that, for a hole $(\ta_1, \ta_2)$ of $\Qf$, the set $\Wc_\la'(\ta_1, \ta_2)$ is defined as the wake $\Wc_\la(\ta_1,
\ta_2)$ except for the holes $(\frac16, \frac13)$ and $(\frac23,\frac56)$ for which
$$
\Wc_\la'\left(\frac16, \frac13\right)=\Wc_\la'\left(\frac23, \frac56\right)=
\Wc_\la\left(\frac16, \frac13\right)\cup\Wc_\la\left(\frac23, \frac56\right).
$$

\begin{comment}

\begin{prop}
\label{p:W-nonspec}
If $\tWc_\la(\theta_1, \theta_2)\ne \0$,
then the parameter rays $\Rc_\la(\theta_1)$ and $\Rc_\la(\theta_2)$
land at a point $f_{root}$, where $f_{root}$ is a polynomial with a parabolic periodic point.
We have $\tWc_\la(\ta_1,\ta_2)=\Wc_\la(\ta_1, \ta_2)$ unless $(\ta_1,\ta_2)=(\frac 23,\frac 56)$ or $(\ta_1,\ta_2)=(\frac 16,\frac 13)$.
In the remaining two cases we have $\tWc_\la(\frac16, \frac13)=\tWc_\la(\frac23,
\frac56)=\Wc_\la(\frac23, \frac56)\cup
\Wc_\la(\frac16, \frac13)$ where sets $\Wc_\la(\frac16,\frac13)$ and
$\Wc_\la(\frac23,\frac56)$ are the two components of $\tWc_\la(\frac16, \frac13)=\tWc_\la(\frac23,
\frac56)$.
\end{prop}

\end{comment}

\begin{prop}
\label{p:W-nonspec} If $\tWc_\la(\theta_1, \theta_2)\ne \0$, then the
parameter rays $\Rc_\la(\theta_1)$ and $\Rc_\la(\theta_2)$ land at a
point $f_{root}$, where $f_{root}$ is a polynomial with a parabolic
periodic point. We have $\tWc_\la(\ta_1,\ta_2)=\Wc'_\la(\ta_1, \ta_2)$;
moreover, sets $\Wc_\la(\frac16,\frac13)$ and
$\Wc_\la(\frac23,\frac56)$ are the two components of $\tWc_\la(\frac16,
\frac13)=\tWc_\la(\frac23, \frac56)$.
\end{prop}

\begin{proof}
By Lemma \ref{l:cuts-cross}, if $\vk\in (\theta_1,\theta_2)$ and $f\in
\Rc_\la(\vk)$, then the corresponding dynamic rays $R_f(\theta_1+\frac
13)$ and $R_f(\theta_2+\frac 23)$ land at the same periodic point $z$
and are smooth. Assume that $\tWc=\tWc_\la(\theta_1, \theta_2)\neq
\emptyset$; then, for every $f\in \tWc$, the point $z$ is repelling.

For a polynomial $f=f_{\la,b}$ in the boundary of $\tWc$, either
one of the rays $R_f(\theta_1+\frac 13)$ and $R_f(\theta_2+\frac 23)$
crashes into a critical point, or both rays are smooth but one of the landing points fails
to be repelling (this follows from Lemma \ref{l:rep}).
Consider these cases separately.

(1) Suppose that at least one of the rays $R_f(\ta_1+\frac 13)$, $R_f(\ta_2+\frac 23)$ is not smooth.
We claim that if $(\ta_1, \ta_2)\ne (\frac23, \frac56)$ and $(\ta_1, \ta_2)\ne (\frac16, \frac23)$, then our assumption
implies that $f\in \Rc_\la(\theta_1)\cup \Rc_\la(\theta_2)$; on the other hand, if $(\ta_1, \ta_2)=(\frac23, \frac56)$ or $(\ta_1,\ta_2)=(\frac16, \frac23)$, then the non-smoothness of either $R_f(0)$ or $R_f(\frac12)$ implies that
$$f\in \Rc_\la\left(\frac23\right)\cup \Rc_\la\left(\frac56\right)\cup \Rc_\la\left(\frac16\right)\cup
\Rc_\la\left(\frac13\right).$$

Indeed, suppose that $R_f(\theta_1+\frac 13)$ is not smooth.
We claim that then $f\in \Rc_\la(\theta_1)$.
By the assumption, either $R_f(\theta_1)$ and $R_f(\theta_1+\frac 13)$ crash at the escaping
critical point (and then $f\in \Rc_\la(\ta_1-\frac13)$), or
$R_f(\theta_1+\frac 13)$ and $R_f(\theta_1+\frac 23)$ crash at the
escaping critical point (and then $f\in \Rc_\la(\ta_1)$).
We claim that the latter option takes place.

Suppose otherwise: $f\in \Rc_\la(\ta_1-\frac13)$, and the rays $R_f(\theta_1)$ and $R_f(\theta_1+\frac 13)$ crash at the escaping critical point.
Arbitrarily close to $f$, there are polynomials $g\in \tWc$.
These polynomials $g$ must satisfy one of the conditions (2)(a) and (2)(b) from Lemma \ref{l:landvk}.
However, in case $(\ta_1, \ta_2)\ne (\frac23, \frac56)$ and $(\ta_1, \ta_2)\ne (\frac16, \frac23)$, we see that (2)(b) is impossible as it would imply that the rays $R_g(\ta_1+\frac13)$, $R_g(\ta_2+\frac23)$ land at $0$, which is parabolic.
Thus, in this case polynomials $g$ belong to external parameter rays $\Rc_\la(\vk)$ with $\vk\in
(\ta_1, \ta_2)$, a contradiction with $f\in \Rc_\la(\ta_1-\frac13)$.

Suppose now that $(\ta_1, \ta_2)=(\frac23, \frac56)$ and choose $f$
from the boundary of $\tWc$ so that $R_f(\theta_1+\frac 13)=R_f(0)$ is
not smooth.
Choose $g\in \tWc$ very close to $f$.
Then the smooth dynamic rays $R_g(0)$, $R_g(\frac12)$ land at the same repelling fixed point of $g$.
By Lemma \ref{l:landvk}, this implies that $g\in \Rc_\la(\vk)$ with $\vk\in (\frac23, \frac56)$ or $\vk\in (\frac16, \frac13)$, and both cases are possible.

(2) Suppose that the point $y$, at which one of the rays $R_f(\theta_1+\frac 13)$ or
$R_f(\theta_2+\frac 23)$ lands, is a parabolic periodic point of $f$.
If $y=0$, then $\lambda=e^{2\pi i p/q}$ for some integers $p$ and $q$. We
claim that then $T_{p/q}(b)=0$. Indeed, suppose that $T_{p/q}(b)\ne 0$.
Then, by Proposition \ref{P:ray0stable}, a small neighborhood of
$f_{root}$ is disjoint from $\tWc(\theta_1, \theta_2)$. This
contradicts the fact that $f_{root}$ is on the boundary of $\tWc$.
If $y\ne 0$, and $m$ is the period of $\theta_1$ (and $\theta_2$), then, by the Yoccoz inequality, $f^{\circ m}$ has multiplier 1 at the point $y$ (cf. the proof of Theorem \ref{t:nonlin} in \cite{bopt14}).
In both cases, $b\in\Zc$, where the set $\Zc$ consists of all parameter values $b$
such that either $T_{p/q}(b)=0$, or $f^{\circ m}$ has a parabolic fixed
point of multiplier 1 different from $0$. Note that $\Zc$ is a finite
set.

Hence, if $(\ta_1, \ta_2)\ne (\frac23, \frac56)$ and $(\ta_1, \ta_2)\ne
(\frac16, \frac23)$, then the boundary of $\tWc$ lies in the union of
$\Rc_\la(\theta_1)\cup\Rc_\la(\theta_2)$ and a finite set of points.
Since $\tWc$ is open, it follows that the rays $\Rc_\la(\theta_1)$ and
$\Rc_\la(\theta_2)$ land at the same point $f_{root}$. By the previous
paragraph, the point $f_{root}$ is a polynomial with a parabolic
periodic point. Since polynomials with disconnected Julia sets located
in the wedge between $\Rc_\la(\theta_1)$ and $\Rc_\la(\theta_2)$ belong
to $\tWc$ (see Proposition \ref{p:land-endholes}; here we mean the wedge in the
positive direction from $\Rc_\la(\theta_1)$ to $\Rc_\la(\theta_2)$)),
it follows that $\tWc$ coincides with this wedge, except, perhaps, for
finitely many points of $\Zc$ removed from this wedge. However, the
existence of such removed points would contradict the maximum modulus
principle applied to the multiplier of the landing point of
$R_f(\ta_1+\frac 13)$ and $R_f(\ta_2+\frac 23)$. The arguments in the
case of $(\ta_1, \ta_2)=(\frac23, \frac56)$ or $(\ta_1, \ta_2)=
(\frac16, \frac23)$ are almost literally the same except that now
instead of one wedge we need to talk about the union of two wedges: one
in the positive direction from $\Rc_\la(\frac23)$ to $\Rc_\la(\frac56)$
and the other one in the positive direction from $\Rc_\la(\frac16)$ to
$\Rc_\la(\frac13)$.
\end{proof}

The set $\tWc_\la(\theta_1,\theta_2)$ is called a \emph{non-special $($parameter$)$ wake}.
By Proposition~\ref{p:W-nonspec}, the parameter rays
$\Rc_\la(\theta_1)$ and $\Rc_\la(\theta_2)$ land at the same point \emph{under the
assumption that there is at least one polynomial $f\in\Fc_\la$
with $R_f(\theta_1+\frac 13)$, $R_f(\theta_2+\frac 23)$ landing at
the same repelling periodic point.} Recall that the common
landing point of the rays $\Rc_\la(\theta_1)$ and $\Rc_\la(\theta_2)$ is
called the \emph{root point} of the parameter wake $\Wc_\la(\theta_1,\theta_2)$.

It remains to consider the case, where the dynamic rays $R_f(\theta_1+\frac 13)$
and $R_f(\theta_2+\frac 23)$ \emph{never} land at the same
\emph{repelling} periodic point,
no matter what polynomial $f\in\Fc_\la$ we choose.
We claim that, in this case, $\la$ must be a root of unity.
Indeed, take an angle $\vk\in (\theta_1,\theta_2)$.
Then, by Theorem \ref{T:ParDynHoles}, the chord
$\ol{(\theta_1+\frac 13)(\theta_2+\frac 23)}$ is the major of the quadratic
invariant gap $\Uf(\vk)$ of periodic type.
It follows from Proposition \ref{p:land-endholes} that, for every $f\in\Rc_\la(\vk)$,
the dynamic rays $R_f(\theta_1+\frac 13)$ and $R_f(\theta_2+\frac 23)$ land
at the same periodic point $x$.
This point is either repelling or parabolic.
By our assumption, $x$ is not repelling, hence $x$ is parabolic.
Since $f\not\in\Cc_\la$ and by the Fatou--Shishikura inequality,
there is at most one non-repelling cycle of $f$.
Thus we must have $x=0$, which means that $0$ is a parabolic point, i.e.,
the multiplier $\la$ is a root of unity.

Fix $\lambda=e^{2\pi ip/q}$ and consider $f\in\Fc_\la$. If the rays
$R_f(\theta_1+\frac 13)$ and $R_f(\theta_2+\frac 23)$ land at a point
$x$, then let $W_f(\theta_1+\3,\theta_2+\4)$ be the (dynamic) wedge
bounded by these rays and the point $x$ that contains all dynamic rays
with arguments in $(\theta_1+\frac 13,\theta_2+\frac 23)$. The boundary
of $W_f(\theta_1+\3,\theta_2+\4)$ is denoted by $\Theta_f$.
Recall that a \emph{parabolic domain of $f$ at $0$} is a Fatou component
of $f$ that contains an attracting petal of $0$.

\begin{lem}
  \label{l:land-par}
If there exists $f\in \Fc_\la\sm \Cc_\la$ such that the dynamic rays
$R_f(\theta_1+\frac 13)$ and $R_f(\theta_2+\frac 23)$ both land at $0$
and the parabolic domains of $f$ at $0$ are disjoint from
$W_f(\theta_1+\3,\theta_2+\4)$, then the convex hull $\gf$ of the set
of all points $\ol\theta$ such that the dynamic ray $R_f(\theta)$ lands
at $0$ is a finite invariant gap of type D. Moreover, if
$f\in\Rc_\la(\varkappa)$, then $\varkappa\in(\theta_1,\theta_2)$.
\end{lem}

\begin{proof}
  By \cite[Lemma 18.12]{M}, there are finitely many
  rays landing at $0$, and they are permuted with combinatorial rotation
  number $p/q$, i.e., $\gf$ is a finite invariant gap.
  Moreover, $\gf$ is of type D.
  Indeed, otherwise the angles $\theta_1+\frac 13$ and $\theta_2+\frac 23$
  would be in the same orbit under the angle tripling map.
  This contradicts Theorem \ref{T:ParDynHoles}, since the endpoints
  of the major of a quadratic invariant set cannot belong to one orbit.

  Let $f\in\Rc_\la(\varkappa)$.
  Then $\varkappa$ belongs to $(\theta_1,\theta_2)$.
  Indeed, the rays $R_f(\varkappa+\frac 13)$ and $R_f(\varkappa+\frac 23)$
  crash at the escaping critical point $\omega_2(f)$.
  Thus, both $\varkappa+\frac 13$ and $\varkappa+\frac 23$
  belong to the same major hole of $\gf$. Since
  $\ol{(\theta_1+\frac 13)(\theta_2+\frac 23)}$ is the major of an invariant
  quadratic gap, the $\si_3$-orbits of $\ta_1+\frac13$ and $\ta_2+\frac 23$ do not
  enter the arc $(\theta_1+\frac 13, \theta_2+\frac 23)$. Hence,
  $(\theta_1+\frac 13, \theta_2+\frac 23)$ is a major hole of $\gf$.
  There are two major holes of $\gf$, either giving rise to the corresponding dynamic wedge containing a critical point
  of $f$. Since, by the assumption, the parabolic domains of $f$ at $0$ are disjoint from
  $W_f(\theta_1+\3,\theta_2+\4)$ and one of these dynamic wedges contains a non-escaping critical point of $f$,
  then the escaping critical point $\omega_2(f)$ is in
  $W_f(\theta_1+\3,\theta_2+\4)$, i.e., we have $\varkappa\in (\theta_1,\theta_2)$.
\end{proof}

Consider the set $\widetilde\Wc_\la(\theta_1,\theta_2)$ of all
polynomials $f=f_{\la, b}$ such that the rays $R_f(\theta_1+\frac 13)$
and $R_f(\theta_2+\frac 23)$ land at $0$, and the attracting petals at
$0$ are disjoint from $W_f(\theta_1+\3,\theta_2+\4)$. Since there are
at most two cycles of dynamic rays landing at $0$, then $f$ can have
only one cycle of parabolic domains at $0$. By Proposition \ref{P:Tpq}, this
implies that $T_{p/q}(b)\ne 0$.

\begin{prop}
 \label{p:W-spec1}
If $\widetilde\Wc_\la(\theta_1,\theta_2)\ne \0$, then
the parameter rays $\Rc_\la(\theta_1)$, $\Rc_\la(\theta_2)$
land at the same point.
\end{prop}

\begin{proof}
 The proof is similar to that of Proposition \ref{p:W-nonspec}.
 If $\widetilde\Wc_\la(\theta_1,\theta_2)\ne \0$, then, by
 Lemma \ref{l:land-par}, the boundary of $\widetilde\Wc_\la(\theta_1,\theta_2)$
 lies in the union of the rays $\Rc_\la(\theta_1)$, $\Rc_\la(\theta_2)$
 and the polynomials $f_{\la,b}$ corresponding to roots $b$ of the polynomial $T_{p/q}$.
 Thus, the rays $\Rc_\la(\theta_1)$ and $\Rc_\la(\theta_2)$ land at the same point.
\end{proof}

Let us now define \emph{special $($parameter$)$ wakes}.

\begin{dfn}\label{d:spec-w}
Let $(\ol\theta_1,\ol\theta_2)$ be a hole of $\Qf$. Suppose that
there exists $f\in \Fc_\la\sm \Cc_\la$ such that the dynamic rays
$R_f(\theta_1+\frac 13)$ and $R_f(\theta_2+\frac 23)$ both land at
$0$ and the parabolic domains of $f$ at $0$ are disjoint from
$W_f(\theta_1+\3,\theta_2+\4)$. Then the set
$\Wc_\la(\theta_1,\theta_2)$ bounded by the parameter rays
$\Rc_\la(\theta_1)$, $\Rc_\la(\theta_2)$ and their common landing
point is called a \emph{special $($parameter$)$ wake}. Recall that
the landing point of the rays $\Rc_\la(\theta_1)$ and
$\Rc_\la(\theta_2)$ is called the \emph{root point} of the parameter
wake $\Wc_\la(\theta_1,\theta_2)$.
Recall also that points in the parameter planes $\Fc_\la$ are polynomials.
\end{dfn}

The next claim complements Proposition \ref{p:W-spec1}.

\begin{prop}\label{p:W-spec2}
If $\widetilde\Wc_\la(\theta_1,\theta_2)\ne \0$, and $f_{root}=f_{\la,b_0}$ is the common landing point of the parameter rays $\Rc_\la(\theta_1)$, $\Rc_\la(\theta_2)$, then $T_{p/q}(b_0)=0$.
Moreover,
$\widetilde\Wc_\la(\theta_1,\theta_2)$ coincides with $\Wc_\la(\theta_1,\theta_2)$,
possibly with several punctures $f_{\la,b_i}$, where $1\le i\le N$ and $T_{p/q}(b_i)=0$.
In particular, the map $f_{\la, b_i}$, for every $0\le i\le N$,
has two cycles of parabolic domains at its unique parabolic point $0$.
\end{prop}

\begin{proof}
The proof is similar to that of Proposition \ref{p:W-nonspec}.
The last claim holds by Proposition~\ref{P:Tpq}.
\end{proof}

We are now ready to prove Theorem A.

\begin{proof}[Proof of Theorem A]
 Consider any hole $(\ol\theta_1,\ol\theta_2)$ of $\Qf$.
 Choose $\varkappa\in (\theta_1,\theta_2)$ and a polynomial $f\in\Rc_\la(\varkappa)$.
 By Lemma \ref{l:landvk}, the dynamic rays
 $R_f(\theta_1+\frac 13)$ and $R_f(\theta_2+\frac 23)$ land at the same
 periodic point $z$.
 Since $\omega_2(f)$ is escaping, then Fatou-Shishikura inequality implies that the point $z$ is either repelling or $0$.
 If $z$ is repelling, then, by Proposition \ref{p:W-nonspec},
 the polynomial $f$ belongs to the non-special parameter wake bounded by the parameter
 rays $\Rc_\la(\theta_1)$, $\Rc_\la(\theta_2)$ and their common landing point.
 In particular, the parameter rays $\Rc_\la(\theta_1)$, $\Rc_\la(\theta_2)$ land
 at the same point.

Now let $z=0$. If $W_f(\ta_1+\3,\ta_2+\4)$ is disjoint from all
attracting petals at $0$, then, by Proposition \ref{p:W-spec2}, the polynomial $f$
belongs to the special parameter wake bounded by the parameter rays
$\Rc_\la(\theta_1)$, $\Rc_\la(\theta_2)$ and their common landing
point. We claim that the remaining case (where $z=0$ but the wedge
$W_f(\ta_1+\3,\ta_2+\4)$ contains some parabolic domain $\Omega$ at
$0$) is impossible. Indeed, suppose $f$ has the just listed properties.
Since $J(f)$ is disconnected, $f$ is immediately renormalizable; we let
$f^*$ be the polynomial-like restriction of $f$. The fact that the
orbits of all points in $\Omega$ converge to $0$ implies that
$\ol\Omega\subset K(f^*)$. Hence there are dynamic rays in the wedge
$W_f(\ta_1+\3,\ta_2+\4)$ that land at points of $\ol{\Omega}\subset
K^*$. However, by Lemma \ref{l:not-accum}, no dynamic ray in the wedge
$W_f(\ta_1+\3,\ta_2+\4)$ accumulates in $K(f^*)$, a contradiction.
\end{proof}

\section{Dynamical description of parameter wakes}
\label{s:dyndescr}
In this section, we expand the dynamical description of polynomials in a given parameter wake.
We will also discuss the dynamics of polynomials corresponding to
the root points of parameter wakes.

\subsection{Special parameter wakes}\label{ss:spec-par-w}
In this subsection, we assume that $\lambda=e^{2\pi i p/q}$ is a root of unity.
By definition, special parameter wakes can exist only in these $\lambda$-slices $\Fc_\la$.
We will need the following proposition.

\begin{prop}
 \label{P:landing-sp}
Let $\Wc_\la(\theta_1,\theta_2)$ be a special parameter wake, and $f_{root}$ be its root point.
Then, for every $f\in\Wc_\la(\theta_1,\theta_2)\cup\{f_{root}\}$, the rays $R_f(\theta_1+\frac 13)$ and $R_f(\theta_2+\frac 23)$ land at $0$.
\end{prop}

\begin{proof}
Consider $f=f_{\la,b}\in\Wc_\la(\theta_1,\theta_2)$ or $f=f_{root}$. If
$T_{p/q}(b)\ne 0$, then, by Propositions \ref{p:W-spec1} and \ref{p:W-spec2}, the
rays $R_{f}(\theta_1+\frac 13)$ and $R_{f}(\theta_2+\frac 23)$ land at
$0$. Suppose now that $T_{p/q}(b)=0$ but at least one of the rays
$R_{f}(\theta_1+\frac 13)$, $R_{f}(\theta_2+\frac 23)$ (say, the
former) fails to land at $0$. The ray $R_{f}(\theta_1+\frac 13)$ has to
land somewhere, say, at a point $z\ne 0$. The point $z$ is repelling or
parabolic. It cannot be repelling because then, by Lemma \ref{l:rep}, for
polynomials $g$ sufficiently close to $f$ the ray $R_g(\theta_1+\frac
13)$ will have to land at a point close to $z$ while, by
Proposition \ref{p:W-spec2}, there are polynomials $g$ arbitrarily close to $f$
such that $R_g(\theta_1+\frac 13)$ lands at $0$, a contradiction. On
the other hand, by the Fatou--Shishikura inequality, the point $z\ne 0$
cannot be parabolic either (observe that both critical points of $f$
are in parabolic domains at $0$).
\end{proof}

A choice of the combinatorial rotation number $p/q$ provides a
classification of holes of $\Qf$ into \emph{special} holes and
\emph{non-special} holes.

\begin{dfn}\label{d:spec-hol}
A hole $(\ol\theta_1,\ol\theta_2)$ of $\Qf$ is \emph{$(p/q)$-special}
if the map $\si_3$ preserves the cyclic order on the periodic orbit of
$\theta_1+1/3$, and the combinatorial rotation number of this orbit under $\si_3$
equals $p/q$. If $(\ol\theta_1,\ol\theta_2)$ is a special hole, then the
orbit of $\theta_2+2/3$ has combinatorial rotation number $p/q$ as
well.
Indeed, since the orbit of $\ol{(\ta_1+\frac 13)(\ta_2+\frac 23)}$ is on
the boundary of some invariant quadratic gap, $\si_3$ permutes this orbit
as a combinatorial rotation.
All holes of $\Qf$ that are not special are called \emph{non-special holes}.
\end{dfn}

Lemma~\ref{l:spec-wake-d} easily follows from Definition~\ref{d:spec-hol}.

\begin{lem}\label{l:spec-wake-d}
A hole $(\ol\theta_1,\ol\theta_2)$ is $p/q$-special if and only if
there exists a finite type D invariant gap $\gf$ with major
$\ol{\theta_1+\frac 13\,\theta_2+\frac 23}$. There
are exactly $2q$ special holes of $\Qf$ corresponding to the rotation number $p/q$.
\end{lem}

\begin{proof}
The first claim of the lemma is left to the reader. Now,
by Lemma \ref{l:typeD-count}, there are $q$ type D finite invariant
gaps of rotation number $p/q$. Every type D finite invariant gap of
rotation number $p/q$ has two major holes giving rise to two
$p/q$-special holes. Overall this yields $2q$ special holes as desired.
\end{proof}

By Lemma~\ref{l:spec-wake-d}, %and Lemma \ref{l:typeD-count},
there
are exactly %$q$
$2q$ special holes of $\Qf$ corresponding to the rotation number $p/q$.
Theorem~\ref{t:sp-wakes-nsp} helps to distinguish
among various parameter wakes. %special parameter wakes from non-special parameter wakes.
%Recall that $p$, $q$, and $\lambda=e^{2\pi i p/q}$, are fixed.
Given $f\in \Fc_\la$, denote by $\gf_f$ the convex hull of the set of
points $\ol\theta$ such that $R_f(\theta)$ lands at $0$.  %Recall that
%$\lambda=e^{2\pi i p/q}$ is a root of unity.

\begin{thm}[Special parameter wakes vs. special holes]
\label{t:sp-wakes-nsp}
The pa\-ra\-meter wake $\Wc_\la(\theta_1,\theta_2)$ is a special parameter wake if and only if the hole
$(\theta_1,\theta_2)$ of $\Qf$ is $p/q$-special. %, equivalently, there is a type D
%finite invariant gap of rotation number $p/q$ having $\ol{\theta_1+\frac 13}$
%and $\ol{\theta_2+\frac 23}$ among its vertices.
In particular, there are $2q$ special parameter wakes in $\Fc_\la$.
\end{thm}

\begin{proof}
Let $\Wc_\la(\theta_1,\theta_2)$ be a special parameter wake.
Consider its root point $f=f_{root}$.
Then, by Proposition~\ref{p:W-spec2}, there are two cycles of parabolic domains at $0$ and two cycles of external rays of $f$ landing at $0$.
It follows that $\gf_{f}$ contains two periodic cycles under $\si_3$.
%the union $T$ of two cycles of arguments
%of rays from two cycles of external rays above.
If $p/q=0$, these two cycles are $\ol 0$ and $\ol{\frac12}$, and $\gf_{f}=\ol{0
\frac12}$. If $p/q\ne 0$, then, by \cite{kiwi02}, there are no other rays
landing at $0$.
Hence, in this case, the set $\gf_{f}$ is a finite type D invariant
gap of rotation number $p/q$. On the other hand, by Proposition
\ref{P:landing-sp}, the points $\ol{\theta_1+\frac 13}$ and
$\ol{\theta_2+\frac 23}$ belong to $\gf_{f}$. The chord
$\ol{\theta_1+\frac 13\, \theta_2+\frac 23}$ must be an edge of
$\gf_{f}$ as otherwise an eventual image of it will cross it.
Thus, the hole $(\ol\theta_1,\ol\theta_2)$ of $\Qf$ is $p/q$-special.

Now, assume that the hole $(\ol\theta_1,\ol\theta_2)$ is $p/q$-special.
By Theorem A, the parameter rays $\Rc(\ta_1)$, $\Rc(\ta_2)$ land at the
same point, and the parameter wake $\Wc_\la(\theta_1,\theta_2)$ exists.
We claim that, in our case, the parameter wake
$\Wc_\la(\theta_1,\theta_2)$ is special. Indeed, otherwise, for
every $f\in\Wc_\la(\theta_1,\theta_2)$, the dynamic rays
$R_f(\theta_1+\frac 13)$ and $R_f(\theta_2+\frac 23)$ land at the
same repelling periodic point $z$. Suppose first that $p/q=0$ and
$\ol{\ta_1+\frac13\, \ta_2+\frac23}$ equals either $\ol{0 \frac12}$
or $\ol{\frac12 0}$. Then both rays $R_f(0)$ and $R_f(\frac12)$ land
at a repelling fixed point $z$. However, in this case $0$ is a parabolic point
with $f'(0)=1$, and at least one fixed external ray lands at $0$, a contradiction.
Hence a $0$-special hole corresponds to a special parameter wake.

Let us now assume that $p/q\ne 0$.
Let $A$ be the component of $\cdisk\sm \ol{\theta_1+\frac13\,
\theta_2+\frac23}$ disjoint from the circle arc
$(\ol{\theta_1+\frac13},\, \ol{\theta_2+\frac23})$, and let $B$ be the other component.
We claim that $\gf_f\subset A$.
Indeed, by \cite{BOPT}, the arc $B\cap\uc$ contains either $\ol{0}$ or $\ol{1/2}$, and it is easy to see that $B\cap\uc$ contains no other $\si_3$-invariant sets (recall that $p/q\ne 0$), a contradiction.
Then $\gf_f$ is a finite invariant gap of rotation number $p/q$ located in the component $A$ of $\cdisk\sm \ol{\theta_1+\frac13\,\theta_2+\frac23}$ and disjoint from the circle arc $(\ol{\theta_1+\frac13},\, \ol{\theta_2+\frac23})$, while the leaf $\ol{\theta_1+\frac13\, \theta_2+\frac23}$ rotates in $A$ with rotation number $p/q$.
We will show in the next paragraph that this is impossible for combinatorial reasons.

Let $\Uf$ be the quadratic invariant gap such that its major hole is
$(\ol{\theta_1+\frac 13},\, \ol{\theta_2+\frac 23})$. The existence of
such a quadratic gap is a consequence of Theorem \ref{T:ParDynHoles}.
It is well-known (see, e.g., \cite{BOPT}) that there is a continuous projection $\psi_{\Uf}:\Uf'\to\uc$ that identifies the endpoints of every arc in $\uc\sm\Uf'$ and that semi-conjugates the mapping $\si_3:\Uf'\to\Uf'$ with the mapping $\si_2:\uc\to\uc$.
Apply this projection to the vertices of $\gf_f$ (by definition of $\Uf$, all vertices of $\gf_f$ belong to $\Uf'$). We obtain a $\si_2$-invariant finite set of points of $\uc$ permuted by $\si_2$ as a combinatorial rotation with rotation number $p/q$.
However, a subset of $\uc$ with these properties is unique. Therefore,
it must coincide with the $\psi_\Uf$-image of the $\si_3$-orbit of
$\ol{\theta_1+\frac 13}$. It follows that one of the vertices of
$\gf_f$ must coincide either with $\ol{\theta_1+\frac 13}$ or with
$\ol{\theta_2+\frac 23}$, a contradiction. The last claim of the
theorem follows from Lemma~\ref{l:spec-wake-d}.
\end{proof}

\subsection{Root points of special parameter wakes}\label{ss:root-spec}
In this subsection, we still assume that $\lambda=e^{2\pi ip/q}$ is
a fixed root of unity. We will prove that all zeros of $T_{p/q}$
correspond to root points of special parameter wakes. In particular, it will follow that
the set $\widetilde \Wc_\la(\theta_1,\theta_2)$ introduced in Section
\ref{ss:wakes} coincides with the special parameter wake
$\Wc_\la(\theta_1,\theta_2)$, where $(\ol\theta_1,\ol\theta_2)$ is any
$p/q$-special hole of $\Qf$. Lemma~\ref{l:com-root}
describes the situation, when two special parameter wakes share a root
point.

\begin{lem}
\label{l:com-root}
 If two special parameter wakes $\Wc_\la(\theta_1,\theta_2)$ and
 $\Wc_\la(\theta'_1,\theta'_2)$ have the same root point, then
 $(\ol{\theta_1+\frac 13},\ol{\theta_2+\frac 23})$ and
 $(\ol{\theta'_1+\frac 13},\ol{\theta'_2+\frac 23})$ are the two major holes of the
 same type D finite invariant gap.
\end{lem}

\begin{proof}
Let $f=f_{root}$ be the common root point of the parameter wakes $\Wc_\la(\theta_1,\theta_2)$ and $\Wc_\la(\theta'_1,\theta'_2)$.
By Proposition \ref{P:landing-sp}, the four dynamic rays
$$
R_{f}\left(\theta_1+\frac 13\right),\quad
R_{f}\left(\theta_2+\frac 23\right),\quad
R_{f}\left(\theta'_1+\frac 13\right),\quad
R_{f}\left(\theta'_2+\frac 23\right)
$$
land at $0$. Therefore, the four arguments of these rays correspond
to the vertices of the same type D finite invariant gap $\gf$.
Clearly, these vertices are on the boundaries of major holes of $\gf$.
\end{proof}

Two majors $\Mf$ and $\Mf'$ of the same type D finite invariant gap $G$,
as well as the corresponding major holes of $G$, will be called
\emph{conjugate}.

\begin{prop}
\label{p:roots}
  Every zero $b$ of the polynomial $T_{p/q}$ corresponds to a common
  root point $f_{\la,b}$ of two special parameter wakes
  $\Wc_\la(\theta_1,\theta_2)$ and $\Wc_\la(\theta'_1,\theta'_2)$,
  where arcs $(\theta_1+\frac 13,\theta_2+\frac 23)$ and
  $(\theta'_1+\frac 13,\theta'_2+\frac 23)$ are conjugate major holes depending on $b$.
  The degree of the polynomial $T_{p/q}$ is equal to $q$.
\end{prop}

\begin{proof}
  By Lemmas \ref{t:sp-wakes-nsp} and \ref{l:com-root}, the $2q$ special parameter wakes in
$\Fc_\la$ have at least $q$ different root points. % (two special
%parameter wakes can share a root point only if they correspond to
%the same type D finite invariant gap).
By
Proposition~\ref{p:W-spec2}, each special parameter wake has a zero
of the polynomial $T_{p/q}$ as its root point. Since, by Proposition
\ref{P:Tpq}, the degree of the polynomial $T_{p/q}$ is at most $q$,
the degree of $T_{p/q}$ equals $q$, and each of the zeros of
$T_{p/q}$ corresponds to a common root point of two special
parameter wakes.
\end{proof}

Note that an alternative way of proving that the degree of the
polynomial $T_{p/q}$ is equal to $q$ may follow the methods of a paper
by Buff, \'Ecalle and Epstein \cite{bee13}, in which the authors prove
a similar statement for parameter slices of quadratic rational
functions.

\begin{thm}[Dynamics of special parameter wakes]
\label{t:dyn-sp} Assume\, that \newline the wake
$\Wc_\la(\theta_1,\theta_2)$ is a special parameter wake. A polynomial
 $f=f_{\la,b}$ belongs to $\Wc_\la(\theta_1,\theta_2)$ if and only if
 the dynamic rays $R_f(\theta_1+\frac 13)$, $R_f(\theta_2+\frac 23)$
 land at $0$, the parabolic domains at $0$ are disjoint from the wedge
 $W_f(\theta_1+\3,\theta_2+\4)$, and $T_{p/q}(b)\ne 0$. A polynomial
 $f$ is the root point of the parameter wake
 $\Wc_\la(\theta_1,\theta_2)$ if and only if $T_{p/q}(b)=0$, and the
 rays $R_{f}(\theta_1+\frac 13)$, $R_{f}(\theta_2+\frac 23)$ land at
 $0$. Moreover, then the polynomial $f$ has a parabolic point $0$ with two cycles
 of parabolic domains at $0$.
\end{thm}

\begin{proof}
Let $f=f_{\la,b}\in \Wc_\la(\theta_1,\theta_2)$. Then, by
Proposition~\ref{P:landing-sp}, the dynamic rays $R_f(\theta_1+\frac
13)$, $R_f(\theta_2+\frac 23)$ land at $0$.
By Proposition \ref{p:roots}, we have $T_{p/q}(b)\ne 0$ for all $f=f_{\la,b}$ in the
parameter wake, and, by Proposition \ref{p:W-spec2}, the parabolic
domains at $0$ are disjoint from the wedge $W_f(\theta_1+\3,\theta_2+\4)$.
On the other hand, if the dynamic rays $R_f(\theta_1+\frac 13)$,
$R_f(\theta_2+\frac 23)$ land at $0$, the parabolic domains at $0$
are disjoint from the wedge $W_f(\theta_1+\3,\theta_2+\4)$, and
$T_{p/q}(b)\ne 0$, then, by Proposition~\ref{p:W-spec2} and by
Proposition~\ref{p:roots}, we have $f\in\Wc_\la(\theta_1,\theta_2)$.
The characterization of the root point of $\Wc_\la(\ta_1,\ta_2)$ follows from Proposition \ref{p:roots}.
\end{proof}

\subsection{Non-special parameter wakes}
In this section, we assume that $\lambda$ is a complex number with $|\lambda|\le 1$.
We characterize dynamics of non-special parameter wakes in $\Fc_\la$.
To describe the properties of root points of special wakes, we need a lemma.

\begin{lem}
  \label{l:multipl}
  Suppose that $f\mapsto z_f$ is a complex-valued holomorphic function defined on
  some open subset $\Uc\subset\Fc_\la$ such that, for every $f\in\Uc$, the point
  $z_f$ is a repelling periodic point of $f$ of period $m$.
  If $f_0$ is a boundary point of $\Uc$, and $z_f$ tends to $0$ as $f\in\Uc$ tends to $f_0$,
  then $\lambda^m=1$, and $0$ is a degenerate parabolic point of $f_0$, i.e.,
  there are at least two cycles of parabolic domains at $0$.
\end{lem}

Note that if $f_0=f_{\la,b_0}$ is as in the statement of the lemma, then we
must necessarily have $T_{p/q}(b_0)=0$, where $p/q$ is the rotation number
associated with $\la$.

\begin{proof}
  Set $f=f_{\la,b}$ and $f_0=f_{\la,b_0}$.
  We have $f^{\circ m}(z)-z=(z-z_f)\varphi_f(z)$, where $\varphi_f$ is a
  polynomial function of $z$ depending holomorphically on $f\in\Uc$
  such that $\varphi_f(z_f)\ne 0$ for $f\in\Uc$.
  It is easy to see that all coefficients of the polynomial $\varphi_f$
  are algebraic functions of $b$ without any poles in the finite part of the plane.
  This follows from the Euclidean algorithm for polynomials.
  In particular, there is a well-defined limit polynomial $\varphi_{f_0}$
  of $\varphi_f$ as $f\in\Uc$ tends to $f_0$.

  On the other hand, since $0$ is a fixed point of $f$, we have $f(0)=0$.
  Note that $z_f\ne 0$ for $f\in\Uc$, since otherwise we must have $|\la|>1$,
  a contradiction with our assumption.
  Therefore, we have $\vp_f(0)=0$ for all $f\in\Uc$.
  It follows that we also have $\vp_{f_0}(0)=0$.
  Therefore, the polynomial $f_0^{\circ m}(z)-z$ is divisible by $z^2$.
  This means that $0$ is a parabolic fixed point of $f_0^{\circ m}$ with multiplier 1.
  In particular, we have $\la^m=1$.

  Suppose now that $\la=e^{2\pi ip/q}$, where $p$ and $q$ are co-prime.
  The number $m$ is then necessarily divisible by $q$ so that $m=qr$ for
  some positive integer $r$.
  We want to prove that $f_0^{\circ q}(z)-z$ is $o(z^{q+1})$.
  To this end, we note that $f^{\circ q}(z)=z+a(b)z^{q+1}+h(b,z)$, where
  $a(b)$ is a polynomial function of $b$, and $h$ is a polynomial in two variables
  divisible by $z^{q+2}$.
  This follows from the fact that $0$ is a parabolic fixed point of $f$ of
  multiplier $\lambda$.
  What we need to show is that $a(b_0)=0$.
  Assume the contrary: $a(b_0)\ne 0$.

  We have $f^{\circ m}(z)=z+ra(b)z^{q+1}+h_r(b,z)$, where $h_r$ is a polynomial
  of two variables divisible by $z^{q+2}$.
  This follows from a simple computation based on the fact that the composition
  of two polynomials $f_1(z)=z+a_1z^{q+1}+\dots$ and
  $f_2(z)=z+a_2z^{q+1}+\dots$, where, in both cases, dots denote the terms of
  order $q+2$ or higher with respect to $z$, is equal to
  $z+(a_1+a_2)z^{q+1}+o(z^{q+1})$.
  Since $(z-z_f)\vp_f(z)$ is divisible by $z^{q+1}$ as a polynomial in $z$,
  we have $\vp_f(z)=z^{q+1}\widetilde\vp_f(z)$, where $\widetilde\vp_f(z)$
  is a polynomial function of $z$, whose coefficients are algebraic functions
  of $b$ without any poles in the finite part of the plane.
  We can now substitute $f_0$ for $f$ to obtain that $f_0^{\circ m}(z)=z+z^{q+2}\widetilde\vp_f(z)$.
  This implies that $a(b_0)=0$, as desired.
\end{proof}

We now prove the following proposition, which characterizes root points of non-special parameter wakes.

\begin{prop}
 \label{P:rootpt-nsp}
Let $\Wc_\lambda(\theta_1,\theta_2)$ be a non-special parameter wake, and
$f_{root}$ its root point.
Then the rays $R_{f_{root}}(\theta_1+\frac 13)$ and $R_{f_{root}}(\theta_2+\frac 23)$
land at the same parabolic periodic point different from $0$.
\end{prop}

\begin{proof}
 Consider the function $f\in\Wc_\lambda(\theta_1,\theta_2)\mapsto z_f$,
where $z_f$ is the common landing point of $R_f(\theta_1+\frac 13)$ and
$R_f(\theta_2+\frac 23)$.
Clearly, $z_f$ is a holomorphic function of $f$.
Being a local branch of a globally defined multivalued analytic function,
the function $z_f$ has a well-defined limit $z_*$ at the root point $f_{root}=f_{\la,b_0}$.
The point $z_*$ cannot be repelling because any neighborhood of $f_{root}$ in $\Fc_\la$
contains polynomials $f$, for which the rays $R_f(\theta_1+\frac 13)$ and
$R_f(\theta_2+\frac 23)$ do not land at a common repelling point (see Lemma \ref{l:rep}).
Therefore, the point $z_*$ is neutral.

Suppose that $z_*\ne 0$.
Then there are no parabolic cycles different from those of $0$ and $z_*$.
This follows from a simple version of the Fatou-Shishikura inequality
stating that the number of cycles of parabolic basin plus the number
of irrationally neutral cycles of a degree $d$ polynomial does not
exceed $d-1$.
The ray $R_{f_{root}}(\theta_1+\frac 13)$ lands either at $0$ or at $z_*$
(in the latter case, $z_*$ must be parabolic).
If it lands at $0$, then, by Proposition \ref{P:ray0stable}, we must have
$T_{p/q}(b_0)=0$.
It follows that there are two cycles of parabolic domains of $0$,
a contradiction with the assumption $z_*\ne 0$.
The contradiction shows that the ray $R_{f_{root}}(\theta_1+\frac 13)$ lands at $z_*$.
Similarly, the ray $R_{f_{root}}(\theta_2+\frac 23)$ also lands at $z_*$.

Suppose now that $z_*=0$.
In this case, it also follows that $T_{p/q}(b_0)=0$ since $0$ becomes
a degenerate parabolic point at the moment when it merges with $z_f$ as $f\to f_{root}$.
This follows from Lemma \ref{l:multipl}.
Hence there are no parabolic cycles different from $0$.
The rays $R_{f_{root}}(\theta_1+\frac 13)$ and $R_{f_{root}}(\theta_2+\frac 23)$ must land
at parabolic points.
Therefore, they land at $0$.
By the dynamical characterization of special wakes, Theorem \ref{t:dyn-sp},
the polynomial $f_{root}$ is the root point of two special wakes,
one of which must coincide with $\Wc_\la(\theta_1,\theta_2)$.
A contradiction with our assumption that the wake $\Wc_\la(\theta_1,\theta_2)$
is non-special.
\end{proof}

We can now give a dynamical description of non-special parameter wakes,
which follows easily from Propositions \ref{p:W-nonspec} and
\ref{P:rootpt-nsp}. Recall that, for a hole $(\ta_1, \ta_2)$ of $\Qf$,
the set $\Wc_\la'(\ta_1, \ta_2)$ is defined as the wake $\Wc_\la(\ta_1,
\ta_2)$ except for the holes $(\frac16, \frac13)$ and $(\frac23,
\frac56)$ for which we set
$$
\Wc_\la'\left(\frac16, \frac13\right)=\Wc_\la'\left(\frac23, \frac56\right)=
\Wc_\la\left(\frac16, \frac13\right)\cup\Wc_\la\left(\frac23, \frac56\right).
$$

We will call sets $\Wc_\la'(\ta_1, \ta_2)$ \emph{non-special} if they correspond
to non-special parameter wakes. Observe that parameter wakes
$\Wc_\la(\frac16, \frac13)$ and $\Wc_\la(\frac23, \frac56)$ are either both special
or both non-special, so the notion of a non-special set $\Wc_\la'(\ta_1, \ta_2)$
is well-defined.

\begin{thm}[Dynamics of non-special parameter wakes]
\label{t:dyn-nsp} \mbox{A po\-ly\-no-} mial $f$ belongs to a
non-special set $\Wc'_\la(\theta_1,\theta_2)$ if and only
if the dynamic rays $R_f(\theta_1+\frac13)$, $R_f(\theta_2+\frac23)$ land at the
same repelling periodic point.
If a polynomial $f_{root}$ is the root point of $\Wc_\la(\theta_1,\theta_2)$,
then the dynamic rays $R_{f_{root}}(\theta_1+\frac 13)$, $R_{f_{root}}(\theta_2+\frac 23)$
land at the same parabolic periodic point of multiplier 1 different from $0$.
\end{thm}

Theorem B follows from Theorems \ref{t:dyn-sp} and
\ref{t:dyn-nsp}, except for its last claim, according to which, for
every root polynomial $f_{root}$, we have $f_{root}\in \cuc_\la$.
We postpone the proof of this claim until Section~\ref{s:proofc}.

\section{Fixed points and (geo)laminations}

The purpose of this section is to prove Theorem C. In the first
several subsections of it we develop the tools necessary for the
proof.

\subsection{Fixed points}
Recall some topological results of \cite{bfmot12}.
First define \emph{special pieces} of a polynomial $f$ with connected Julia set;
Loosely, these are subcontinua $X$ of $K(f)$ carved in it by \emph{exit continua} $E_1$, $\dots$, $E_n$ located on the boundary of $X$ so that $f(X)$ grows out of $X$ only
``through'' the exit continua of $X$. The definition below is a
little more special and less general than that in \cite{bfmot12}.

\begin{dfn}[Central component and exit continua]\label{d:pupie}
Let $E_1$, $\dots$, $E_n$ be a finite (perhaps empty) collection of
(possibly degenerate) continua in $J(f)$, each consisting of
principal sets of two or more external rays. Denote the union of
$E_i$ with these external rays by $\widetilde E_i$.
Suppose that there is a component $T$ of $\C\sm \cup \widetilde E_i$, whose
boundary intersects \emph{all} $E_1$, $\dots$, $E_n$.
Then the continua $E_i$ are called \emph{exit continua $($of $T)$} while $T$ is called
the \emph{central component $($of the exit continua $E_1$, $\dots$, $E_n)$}.
\end{dfn}

Observe that the collection of exit continua may be empty, in which
case the central component $T$ coincides with $\C$.

\begin{dfn}[Special pieces]\label{d:spiece}
Let $T$ be the central component of the exit continua $E_1, \dots,
E_n$. For every $i$, let the wedge $W_i$ be the component of $\C\sm
\widetilde E_i$ containing $T$. \emph{Any} non-separating continuum
$X$ is said to be a \emph{special piece $($of $f)$} if the
following holds:

\begin{enumerate}

\item $X\subset [T\cap K(f)]\cup (\cup E_i)$ is a continuum and $X$ contains $\cup E_i$;

\item each $E_i$ is either a fixed point or maps forward in such a
way that $f(E_i)\subset W_i$ (loosely, sets $E_i$ are
mapped ``towards $T$'');

\item the set $f(X)\sm X$ is disjoint from $T$ (loosely speaking,
$X$ can only grow ``through exit continua'').

\end{enumerate}

\end{dfn}

Observe that the set $[T\cap K(f)]\cup (\cup
E_i)$ above is a continuum.

\begin{thm}[\cite{bfmot12}]\label{t:fxpt}
Let $f$ be a polynomial with connected Julia set, and $X$ be a
special piece of $f$. Then $X$ contains a fixed Cremer or
Siegel point, or an invariant attracting or parabolic Fatou domain,
or a fixed repelling or parabolic point at which at least two rays
land so that $f$ non-trivially rotates these rays.
\end{thm}

In particular, Theorem~\ref{t:fxpt} applies to invariant
continua in $K(f)$ that are non-separating in $\C$ (with an empty collection of exit
continua).

A major tool for us is the notion of an \emph{impression}.

\begin{dfn}\label{d:impre}
Given a (possibly disconnected) Julia set $J$, choose an angle $\al$
and consider the set $L^-$ of all limit points of sequences $x_i\in
R_f(\al_i)$ where $R_f(\al_i)$ are smooth rays with arguments
$\al_i$, and $\al_i$ converge to $\al$ from the negative side in the
usual sense (i.e., $\al_i<\al$). Then $\imp(\al^-)=L^-\setminus
R^e_f(\al^-)$ is the \emph{one-sided impression of $\al$ from the
negative side}. Similarly we define the \emph{one-sided impression
$\imp(\al^+)$ of $\al$ from the positive side}. In case $J$ is
connected, the \emph{impression} $\imp(\al)$ is defined as $\imp(\al^-)\cup
\imp(\al^+)$.
\end{dfn}

\begin{cor}\label{c:inf-many}
Let $f$ be a polynomial with connected Julia set and let $T$ be the
central component of exit continua $E_1$, $\dots$, $E_n$, each of which
is a repelling or parabolic periodic point.
Set $Y=[T\cap K(f)]\cup (\cup E_i)$.
Suppose that $Y$ contains no periodic Cremer points, no periodic Siegel points, and no periodic attracting or parabolic Fatou domains.
Then there are infinitely many periodic points in $Y$, at each of which finitely many external rays land so that the minimal iterate of $f$ that fixes the periodic point non-trivially rotates these external rays.
\end{cor}

\begin{proof}[Proof of Corollary \ref{c:inf-many}]
Choose the iterate $f^{\circ r}$ of $f$ that fixes all rays landing at $E_1$, $\dots$, $E_n$.
It is easy to see that $Y=Y_0$ is a special piece for $f^{\circ r}$.
By Theorem \ref{t:fxpt}, there exists an $f^{\circ r}$-fixed repelling point
$y_0\in Y_0$ with several external rays of $f^{\circ r}$-period $m_0>1$ landing at $y_0$.
Observe that, by the choice of $r$, the point $y_0$ is not equal to any of the points $E_0$, $\dots$, $E_n$.
Consider one of the wedges formed by the external rays of $f$ landing at $y_0$, say, $W_1$, and set $Y_1=W_1\cap Y_0$.
It follows that $Y_1$ is a special piece for a suitable iterate of $f$, whose exit continua are periodic points of $f$.
We can now apply the same argument to $Y_1$, etc.
\end{proof}

The following is a consequence of Corollary \ref{c:inf-many}.

\begin{cor}\label{c:inf-many1}
Suppose that $Z$ is a non-separating continuum in $\C$ that is obtained as a finite union of one-sided impressions of periodic external angles and that contains no periodic Cremer points and no periodic Siegel points.
Then $Z$ is a singleton.
\end{cor}

\begin{proof}
Let us first observe that one-sided impressions cannot contain parabolic of attracting periodic Fatou domains.
Notice also that, passing to a suitable iterate of the map, we may assume that all the arguments of the external rays involved are invariant.
Then, by the previous paragraph, we can consider some iterate $f^{\circ n}$ of $f$ such that $f^{\circ n}|_Z$ has a fixed point $z$ at which several external ray land, and these external rays rotate under $f^{\circ n}$.
Since the arguments of the external rays whose impressions form $Z$ are invariant, we see that these impressions cannot contain $z$, a contradiction.
\end{proof}

\subsection{(Geo)laminations}\label{ss:geolam}

Let us recall some known facts about invariant geolaminations.
A \emph{geodesic} (also called \emph{geometric}) lamination, which we will
abbreviate as a \emph{geolamination}, is a
closed collection $\lam$ of closed chords in $\ol\disk$ that do not cross in $\disk$.
We will always assume that all points of $\uc$ (i.e., all degenerate
chords) are included into $\lam$.
The collection $\lam$ being closed means that the union $\lam^+$ of
all chords in $\lam$ is a closed subset of $\ol\disk$.
Elements of $\lam$ are called \emph{leaves}.

As in the beginning of Subsection~\ref{ss:overview}, we identify
$\uc$ with $\R/\Z$. Then the $d$-tupling map $t\mapsto d\cdot t \pmod
1$ identifies with the map $\si_d:z\mapsto z^d$ on the unit circle $\uc$
in the complex plane $\C$. If a chord $\ell$ has endpoints $\ol{\al}$
and $\ol{\be}$, then we denote this chord by $\ol{\al\be}$. We will always
extend $\si_d$ over all leaves of a given geolamination $\lam$ by
linearly mapping any leaf $\ell=\ol{\al\be}$ onto the chord
$\si_d(\ol\al)\si_d(\ol\be)=\ol{(d\al)(d\be)}$.

\begin{dfn}[Invariant geolaminations, cf \cite{bmov13}]\label{d:sibli}
A geo\-la\-mi\-na\-tion $\lam$ is said to be (sibling) $\si_d$-\emph{invariant} if:
\begin{enumerate}
\item for each $\ell\in\lam$, we have $\si_d(\ell)\in\lam$,
\item \label{2}for each $\ell\in\lam$ there exists $\ell^*\in\lam$ such that $\si_d(\ell^*)=\ell$,
\item \label{3} for each $\ell\in\lam$ such that $\si_d(\ell)$
    is a non-degenerate leaf, there exist $d$ \textbf{pairwise disjoint}
 leaves $\ell_1$, $\dots$, $\ell_d$ in $\lam$ such that
 $\ell_1=\ell$ and
 $\si_d(\ell_i)=\si_d(\ell)$ for all $i=2$, $\dots$, $d$.
\end{enumerate}
The leaves $\ell_1$, $\dots$, $\ell_d$ are called \emph{siblings} of $\ell$.
\end{dfn}

\emph{Gaps} of $\lam$ are the closures of components of $\ol\disk\sm\lam^+$.
Gaps of $\lam$ are \emph{finite} or \emph{infinite} according to whether
they have finitely many or infinitely many points in $\uc$.
We say that a $\si_3$-invariant geolamination $\lam$ \emph{co-exists} with a
stand-alone invariant quadratic gap $\Uf$ if there are no leaves of $\lam$
crossing edges of $\Uf$ in $\disk$ and different from these edges.
We say that the geolamination $\lam$ \emph{tunes} the gap $\Uf$ if all edges of
$\Uf$ are leaves of $\lam$.

Recall that a continuous map between topological spaces is said to be \emph{monotone} if the full preimage of any connected set is connected.
When talking about monotone maps, we will always assume their continuity.

\begin{dfn}[\cite{bco11}]\label{d:finest}
Given a continuum $Q$, define a \emph{finest map of $Q$ onto a
locally connected continuum $Z$} as a monotone map $\vp_Q:Q\to Z$
such that for any monotone map $t:Q\to T$ onto a locally connected
continuum $T$ there exists a monotone map $h:Z\to T$ with $t=h\circ \vp_Q$.
It is easy to see that $\vp_Q$ and $Z$ are defined up to a
homeomorphism, and we can talk about \emph{the} finest map
$\vp_Q=\vp$. A continuum $X\subset \C$ is \emph{unshielded} if it
coincides with the boundary of the unbounded component of the set
$\C\sm X$.
\end{dfn}

The finest map exists for unshielded continua.

\begin{prop}[\cite{bco11}]\label{p:thefinest}
For an unshielded continuum $Q$, there is an equivalence relation
$\sim_Q$ on $\uc$ such that the finest monotone map $\vp_Q$ maps $Q$
onto $\uc/\sim_Q$. The map $\vp_Q$ can be extended onto
the whole complex plane $\C$ so that $\vp_Q:\C\to \C$ is one-to-one
outside of $Q$; in what follows, when talking about $\vp_Q$, we mean
the extended map $\vp_Q:\C\to \C$.
\end{prop}

The following terminology will be used in the rest of the paper.

\begin{dfn}[\cite{bco11}]\label{d:lamigen}
The equivalence relation $\sim_Q$ is called the \emph{finest lamination
$($of $Q)$}. Let $p_{\sim_Q}:\uc\to\uc/\sim_Q$ be the canonical
projection. Given $y\in Q$, let \emph{$\sim_Q$-class
generated by $y$} be the $\sim_Q$-class $p^{-1}_{\sim_Q}(\vp_Q(y))$.
Full preimages of points under $\vp_Q$ are called \emph{fibers} of
$\vp_Q$.
\end{dfn}

If $Q=J(f)$ is the connected Julia set of a polynomial $f$ of degree
$d$, this construction is compatible with the dynamics of $\si_d$.

\begin{dfn}[\cite{bco11}]\label{d:dynappl}
Set $\sim_{J(f)}=\sim_f$ and call it the \emph{lamination generated
by $f$}. Denote $\uc/\sim_f$ by $J_\sim(f)$ and call it the
\emph{topological Julia set}. If $f^*$ is a polynomial-like map, then,
similarly, one defines the \emph{lamination generated by $f^*$}.
Define a geolamination $\lam_f$ as follows: a chord $\ell$ is a leaf
of $\lam_f$ if $\ell$ is a boundary chord of the convex hull of some $\sim_f$-class.
Call $\lam_f$ the \emph{geolamination generated by $f$}.
\end{dfn}

Theorem~\ref{t:monmodel} shows that the finest map onto the connected
Julia set preserves the dynamics.

\begin{thm}[\cite{kiwi97, bco11}]\label{t:monmodel}
If a polynomial $f$ has connected Julia set, then the map
$\vp_{J(f)}:\C\to \C$ semiconjugates $f$ and a branched covering map of
the plane $f_\sim:\C\to \C$. On $J(f)$, the map $\vp_{J(f)}$ collapses
all impressions of angles to points; it is one-to-one on the boundaries
of Fatou domains eventually mapped onto an attracting or parabolic
periodic Fatou domain and maps these boundaries onto Jordan curves.
Moreover, $\vp_{J(f)}$ is one-to-one on the set of all
$($pre$)$periodic points $x\in J(f)$ such that the $\sim_f$-class
generated by $x$ is finite.
\end{thm}

Theorem \ref{t:monmodel} justifies the following terminology.

\begin{dfn}\label{d:topoly}
The map $f_\sim:\C\to \C$ from Theorem \ref{t:monmodel} is called a
\emph{topological polynomial}.
\end{dfn}

Note the following immediate consequence of Theorem
\ref{t:monmodel}: if points $\ol\alpha$ and $\ol\beta$ are periodic or preperiodic
under $\si_d$ and belong to the same finite $\sim_f$-class, then the
rays $R_f(\al)$, $R_f(\be)$ land at the same point.
By \cite{bco11}, if $\ol\al$ belongs to the $\sim_f$-class corresponding to a
point $x\in J_{\sim}(f)$,
then the impression of $R_f(\al)$ is contained in $\vp_f^{-1}(x)$; otherwise,
the impression of $R_f(\al)$ is disjoint from $\vp_f^{-1}(x)$.

Infinite gaps of the geolamination $\lam_f$ may be of two types.
Firstly, they may be convex hulls of infinite $\sim_f$-classes.
Secondly, they may correspond to bounded Fatou components of $f$ as follows.
Let $U$ be a bounded Fatou component of $f$ such that $\vp_f$ does not
collapse the boundary of $U$.
Then $\vp_f$ maps the boundary of $U$ to some Jordan curve $S_U$.
For each $y\in S_U$ consider the convex hull in the unit disk of $p_{\sim_f}^{-1}(y)$.
Then take the closure of the component of the unit disk minus the union of these
convex hulls whose boundary intersects all of them; denote this closure
by $\Uf$. % be the convex hull of the union of the $\sim_f$-classes constituting $S_U$.
Then $\Uf$ is an infinite gap of $\lam_f$ \emph{corresponding} to $U$.
The correspondence between Fatou components of $f$ and gaps of $\lam_f$ just
described will be heavily used in the sequel.

\begin{prop}\label{p:cs}
A periodic fiber $Q$ of $\vp_f$ that corresponds to an infinite
$\sim_f$-class must contain a Cremer or a Siegel periodic point.
Moreover, if $Q$ has a a repelling periodic cutpoint, then
$Q$ contains at least two irrationally indifferent periodic points.
\end{prop}

\begin{proof}
Suppose that $Q$ contains neither Cremer nor Siegel periodic points.
Since $Q$ is a fiber of $\vp_f$, by Theorem~\ref{t:monmodel}, the continuum $Q$
cannot contain a Fatou domain eventually mapped to a parabolic or
attracting periodic Fatou domain. Thus, Corollary~\ref{c:inf-many}
applies and shows that there are infinitely many points in $Q$, at
which the landing rays rotate in a non-trivial fashion. However, by
\cite[Proposition 40]{bco11}, the fiber $Q$ can only contain a
finite collection of periodic cutpoints.
Finally, let $x_0$ be a repelling periodic cutpoint of $Q$.
Then, arguing as above, we see that every component of $Q\sm\{x_0\}$
contains a Cremer or a Siegel periodic point.
This proves the last part of the proposition.
\end{proof}

\section{The invariant quadratic gap $\Uf(f)$}\label{s:qugap}

In this section, we consider an immediately renormalizable
polynomial $f\in \Fc_\la$ with $|\la|\le 1$ (in particular, by Theorem \ref{t:unboren} any
polynomial $f$ in the unbounded component of $\Fc_\la\sm \Pc_\la$ is immediately renormalizable), and define
an invariant quadratic gap $\Uf(f)$ associated with $0$; by $f^*:U^*\to V^*$, we
denote the corresponding polynomial-like map. When $J(f)$ is
disconnected, the gap that is the convex hull of all arguments of
rays to $K(f^*)$ was introduced in Subsection~\ref{ss:exteql}
(though using different approach based upon the fact that $J(f)$ is
disconnected in an essential way). Our approach here is necessarily
different as we deal with immediately renormalizable polynomials and
their quadratic-like Julia sets in both connected and disconnected
cases. Once we introduce $\Uf(f)$, it will be easy to see that the
gap introduced in the disconnected case in
Subsection~\ref{ss:exteql} does coincide with $\Uf(f)$.

\begin{lem}\label{l:uniqj}
Let $f$ be a complex cubic polynomial with a non-repelling fixed point $a$.
Suppose that there exists a quadratic-like Julia set $J^*$ such that $a\in J^*$.
Then $J^*$ is unique.
\end{lem}

\begin{proof}
Suppose that $U_1$ and $U_2$ are Jordan disks such that $f^*_1=f|_{U_1}$ and $f_2^*=f|_{U_2}$ are quadratic-like maps with filled Julia sets $K^*_1$, $K^*_2$, both containing the non-repelling fixed point $a$.
Observe that, by definition, the map $f:U_1\cap U_2\to f(U_1\cap U_2)$ is proper.
Let $U$ be the component of $U_1\cap U_2$ containing $a$.
Then $f:U\to f(U)$ is also proper.
Moreover, by definition $f(\bd(U))\cap \ol{U}=\0$.
Since $f(a)=a$, it follows from the Maximum Modulus Principle that $\ol{U}\subset f(U)$.
Replacing $f(U)$ with a slightly smaller Jordan disk $V$ and redefining $U$ as the component of $f^{-1}(V)$ containing $a$, we obtain a polynomial-like map $f^*=f|_U$.

If $f^*$ is univalent, then, by the Schwarz Lemma, $a$ is repelling, a contradiction.
Hence the degree of $f^*$ is greater than one, and hence equals two.
Then, by Theorem 5.11 from \cite{mcm94}, it follows that $J_1^*=J_2^*$.
\end{proof}

To define the quadratic invariant gap $\Uf(f)$ associated with $K(f^*)$, we use (pre)periodic points of $f$.
Observe that, by definition, $K(f^*)$ is a component of $f^{-1}(K(f^*))$.

\begin{dfn}\label{d:uf}
Let $\widehat X(f)=\widehat X$ be the set of all $\si_3$-(pre)periodic
points $\ol\al\in\uc$ such that $R_f(\al)$ lands in $K(f^*)$. Let
$X(f)=X$ be the closure of $\widehat X$. Let $\Uf(f)$ be the convex
hull of $X$. Let $\widetilde K(f^*)$ be the component of
$f^{-1}(K(f^*))$ different from $K(f^*)$ (such a component of
$f^{-1}(K(f^*))$ exists because $f|_{K(f^*)}$ is two-to-one). Let
$Y(f)=Y$ be the closure of the set of all preperiodic points
$\ol\al\in\uc$ with $R_f(\al)$ landing in $\widetilde K(f^*)$.
\end{dfn}

The gap $\Uf(f)$ is the combinatorial counterpart of the set $K(f^*)$.
Clearly, $X\ne \0$ and the map $\si_3:X\to X$ is such that any point of
$X$ is the image of two or three different points of $X$.
Denote the length of any circle arc $T$ by $|T|$.
The lengths of circle arcs are measured in radians$/2\pi$ so that the total length of the unit circle is equal to 1.

\begin{lem}
\label{l:exact2-1}
The map $\si_3|_{\widehat X}$ is exactly two-to-one.
Moreover, $Y$ lies in the closure of a hole $I=(\ol\al,\ol\be)$ of $X$.
\end{lem}

\begin{proof}
The set $Y$ lies in the closure $[\ol\al, \ol\be]$ of a hole $I=(\ol\al, \ol\be)$ of $X$ because otherwise $K(f^*)$ and $\widetilde K(f^*)$ are non-disjoint.
Let $\ol\ta\in \widehat X$ and let $z\in K(f^*)$ be the landing point of $R_f(\ta)$.
If $z$ is critical, then we may assume that $R_f(\ta+\frac13)$ lands at $z$, and hence
$\ol{\ta+\frac13}\in \widehat X$; if $z$ is not critical, then we can find a point $z'\in K(f^*)$ with the same image as $z$ and may again assume that, say, $R_f(\ta+\frac13)$ lands at $z'$, and hence $\ol{\ta+\frac13}\in \widehat X$.
Thus the map $\si_3|_{\widehat X}$ is at least two-to-one.
Now, suppose that $\ta$, $\ta+\frac13$, $\ta+\frac23\in \widehat X$.
Let $z\in K(f^*)$ be the landing point of $R_f(\ta)$.
Choose a point $\tilde z\in \widetilde K(f^*)$ with the same image as $z$; then there exists $\ol{\tilde\ta}\in Y$ such that $R_f(\tilde\ta)$ lands at $\tilde z$.
Thus, $\ol{\tilde\ta}\notin \widehat X$ whereas $\ol{\tilde\ta}$ must coincide with either
$\ol{\ta+\frac13}$ or $\ol{\ta+\frac23}$, a contradiction.
\end{proof}

From now on, in the setting when $f$ is immediately renormalizable,
$I=(\ol\al, \ol\be)$ is the hole of $X$ whose closure $[\ol\al,
\ol\be]$ contains the set $Y$.

\begin{lem}
  \label{l:XY}
  We have that $\widetilde K(f^*)$ is disjoint from $U^*$.
  The gap $\Uf(f)$ is an invariant quadratic gap of regular critical or periodic type, and $I$ is its major hole.
\end{lem}

\begin{proof}
The fact that $\widetilde K(f^*)$ is disjoint from $U^*$ follows
because otherwise points of $\widetilde K(f^*)\cap U^*$ must belong
to $K(f^*)$, a contradiction. Suppose that $|I|<\frac13$. If
$\ol\al\in \widehat X$, then points $\ol\al+\frac13$,
$\ol\al+\frac23$ belong to $(\ol\be,\ol\al)$ and cannot belong to
$Y\subset [\ol\al, \ol\be]$. Hence in this case three points in
$\widehat X$ have the same $\si_3$-image, a contradiction with
Lemma \ref{l:exact2-1}. If $\ol\al\notin \widehat X$, we can find a point
$\ol\ta\in \widehat X$ so close to $\al$ that points $\ol\ta,
\ol\ta+\frac13, \ol\ta+\frac23$ belong to $(\ol\be, \ol\al)$, which
similarly leads to a contradiction. Hence $|I|\ge \frac13$.

We claim that if $|I|=\frac13$, then neither $\ol\al$ nor $\ol\be$ is
periodic. Otherwise let $\ol\al$ be periodic and set $\ol{\al\be}=\cf$.
Since $X$ is invariant, $X\subset \Uf'(\cf)$. By \cite{BOPT} (see
Subsection~\ref{ss:overview}), the gap $\Uf(\cf)$ is of caterpillar type,
$\ol\be$ is isolated in $\Uf'(\cf)$ and thus isolated in $X$. Thus,
$\ol\be\in \widehat X$, which implies that $\ol\al\in \widehat X$ by
definition and because an eventual $\si_3$-image of $\ol\be$ equals $\ol\al$.
Since $\ol\al+\frac23\in (\ol\be, \ol\al)$, we have $\ol\al+\frac23\notin Y$, and hence the landing point of $R_f(\al+\frac23)$ must also belong to $K(f^*)$.
Therefore, $\ol\al+\frac23\in \widehat X$, a contradiction with Lemma \ref{l:exact2-1}.
We conclude that if $|I|=\frac13$, then $\Uf(f)$ is of regular critical type.

Assume that $|I|>\frac13$. Choose a critical chord $\cf$ with
endpoints in $[\ol\al, \ol\be]$; clearly, we can always choose $\cf$
to be a critical chord with no periodic endpoint. Then $X\subset
\Uf'(\cf)$, where $\Uf(\cf)$ is the invariant quadratic gap defined
in Subsection~\ref{ss:overview}; by \cite{BOPT}, the set $\Uf'(\cf)$
has no isolated points, and, for any point $\ol\ga\in \Uf'(\cf)$, the
backward orbit of $\ol\ga$ in $\Uf'(\cf)$ is dense in $\Uf'(\cf)$ (see
Subsection~\ref{ss:overview} for details). Hence $X=\Uf'(\cf)$ is a
quadratic invariant gap of periodic type as desired.
\end{proof}

One-sided impressions were introduced in Definition \ref{d:impre}.

\begin{prop}\label{p:impinj}
If $f$ is a cubic polynomial with connected Julia set $J(f)$ then
one-sided impressions are subcontinua of $J(f)$.
\end{prop}

\begin{proof}
Let us show that $\imp(\ga^-)$ is a continuum.
Recall that $\phi_f$ is the B\"ottcher coordinate on $\C\sm K(f)$.
For $\eps>0$ and $r>1$, let $T(\eps,r)$ be the closure of the set $\phi_f^{-1}(S(\eps,r))$, where $S(\eps,r)$ consists of all points $\rho\, e^{2\pi i\ta}$ with $\rho<r$ and $\ta\in (\ga-\eps,\ga)$.
Clearly, each set $T(\eps, r)$  is a continuum.
By definition, $\imp(\ga^-)$ is the intersection of a nested family of sets $T(\eps, r)$ taken for $\eps\to 0$ and $r\to 1$.
It follows that $\imp(\ga^-)$ is a continuum.
Moreover, since points of such closures need more and more time to escape to infinity as $r\to 0$, it follows that $\imp(\ga^-)\subset J(f)$, as desired.
\end{proof}

\begin{prop}
  \label{p:uflaf}
Suppose that $f\in\Fc_\la$ with $|\la|\le 1$ has a disconnected Julia set.
Let $\ol{\al\be}$ be any edge of $\Uf(f)$.
Then the principal sets of $R^e(\al^-)$ and $R^e(\be^+)$ intersect.
Moreover, if $\al$ and $\be$ are periodic, then the rays $R^e(\al^-)$ and $R^e(\be^+)$ land at the same point of $K(f^*)$.
\end{prop}

\begin{proof}
Recall that $f^*:U^*\to V^*$ is a quadratic-like restriction of $f$.
Let $\psi:\C\sm K(f^*)\to \C\sm\ol\disk$ be a conformal isomorphism. By
way of contradiction let us assume that the principal sets of $R^e(\al^-)$ and $R^e(\be^+)$ are disjoint.

Since $J(f)$ is disconnected, the set $K(f^*)$ is a component of the filled Julia set $K(f)$.
Then $R^e_f(\al^-)$, $R^e_f(\be^+)$ correspond to different quadratic-like rays $R^*(\al)$ and $R^*(\be)$ (recall that, by Theorem~\ref{t-extepoly}, for every external ray $R^e$ to $J(f^*)$, there exists a unique quadratic-like ray $R^*$ such that $\pr(R^e)=\pr(R^*)$, and these rays are homotopic in $\C\sm K(f^*)$ among all curves with the same limit set).
We can define a quasi-conformal isomorphism $\psi^*$ between $U^*\sm K(f^*)$ and some topological annulus $A$, whose inner boundary is the unit circle.

The map $\psi^*$ is the composition of the straightening map (taking $U^*$ to some Jordan neighborhood of $K(g)$, where $g$ is the quadratic polynomial hybrid equivalent to $f^*$) and the B\"ottcher coordinate for $g$.
The quadratic-like rays $R^*(\al)$ and $R^*(\be)$ are mapped to different radial rays under $\psi^*$.
These two radial rays disconnect the annulus $A$.
Each complementary component has infinitely many radial rays with periodic arguments.
The associated quadratic-like rays correspond to external rays (by Theorem \ref{t-extepoly}) that must all land in $K(f^*)$, a contradiction with the definition of $\Uf(f)$.
Finally, if an argument of a one-sided external ray is periodic, then the one-sided ray in question must land.
This implies the last claim of the lemma.
\end{proof}

In the rest of Section~\ref{s:qugap}, we suppose that $J(f)$ is connected.
Let us relate the gap $\Uf(f)$ and the lamination $\sim_f$. For an edge
$\ol{\al\be}$ of $\Uf(f)$, we let $\Delta(\al,\be)$ stand for the set
$$
\imp(\al^-)\cup\imp(\be^+)\cup R^e_f(\al^-)\cup R^e_f(\be^+).
$$
This set is clearly a closed set.

Let us now quote Theorem \ref{t:conn-extepoly}, which is proven in
\cite{bopt16} and is similar to Theorem \ref{t-extepoly} but applicable to the
case when the Julia set is connected. Given a set $T\subset \uc$ and a
map $h:T\to \uc$ we say that $h$ is \emph{monotone extendable} if it
has a monotone extension $m:\uc\to \uc$.

\begin{thm}[\cite{bopt16}, Lemma 4.9]\label{t:conn-extepoly}
Let a polynomial $f$ of degree $d$ have a connected Julia set and
$K(f^*)\subset K(f)$ be a polynomial-like connected filled Julia set of
degree $k$. Let $\mathcal B_f(K(f^*))=\mathcal B$ be the set of all
angles $\ta\in \uc$ such that the external ray $R^e_f(\ta)$ lands in
$J(f^*)$. Then there is a monotone extendable continuous map
$\psi:\mathcal B\to \uc$ such that:
\begin{enumerate}
\item the set of all $($pre$)$periodic angles from $\psi(\mathcal B)$
    is dense in $\uc$;
\item for every $\ta\in \mathcal B$, the polynomial-like ray
    $R^*(\psi(\ta))$ lands at the same point $y_\ta$ as
    $R^e_f(\ta)$ and is homotopic to $R^e_f(\ta)$ outside of
    $K(f^*)$ under a homotopy that fixes $y_\ta$;
\item we have $\psi\circ \si_d(\ta)=\si_k\circ \psi(\ta)$ for
    every $\ta\in \mathcal B$.
\end{enumerate}
\end{thm}

By Theorem \ref{t:conn-extepoly}, the $\psi$-images of $\si_d$-(pre)periodic points of $\mathcal B$
are $\si_k$-(pre)periodic. Since for any connected polynomial Julia set $J$ a dense
in $\uc$ set of (pre)periodic angles gives rise to a dense in $J$ set of their landing
points, it follows that the set of landing points of all (pre)periodic angles
from $\mathcal B$ is dense in $J(f^*)$.

\begin{prop}\label{p:uflaf1}
If $\ol{\al\be}$ is an edge of $\Uf(f)$, then $\imp(\al^-)\cap\imp(\be^+)\cap J(f^*)$ is non-empty.
Moreover, either $J(f^*)\subset\imp(\al^-)\cup \imp(\be^+)$ or, otherwise, $J(f^*)\sm [\imp(\al^-)\cup \imp(\be^+)]$
is contained in the unbounded component of $\C\sm \Delta(\al,\be)$ containing external rays with arguments from $(\be,\al)$.
In particular, $\ol\al \sim_f \ol\be$.
\end{prop}

Recall that, since $J(f)$ is connected, the lamination $\sim_f$ generated by $f$ is defined.
Unshielded compacta $A$ and their topological hulls $\thu(A)$ were defined in Subsection~\ref{ss:immeren}.
The topological hull of a compact unshielded subset of the plane is the union of topological hulls of its components.

\begin{proof}
Set $T=\imp(\al^-)\cap J(f^*)$.  We claim that the topological hull
$\thu(T)=T'$ of $T$ coincides with the intersection of the topological
hull $\thu(\imp(\al^-))$ of $\imp(\al^-)$, and $K(f^*)$. Indeed, let
$S$ be a bounded complementary component of $T$.
Then $S\subset \thu(\imp(\al^-))\cap K(f^*)$.
Hence $T'\subset \thu(\imp(\al^-))\cap K(f^*)$. On the other
hand, if $z\in \thu(\imp(\al^-))\cap K(f^*)$, then either $z\in T$, or
$z$ belongs to a Fatou domain of $J(f)$ that is also a bounded
complementary domain of both $\imp(\al^-)$ and $K(f^*)$. In either case
$z\in T'$ as desired. Thus, $T'=\thu(\imp(\al^-))\cap K(f^*)$.

We claim that $T$ is a continuum. Recall that one-sided impressions are
continua. We will use Theorem \ref{t:conn-extepoly}, including the notation
from that theorem. Observe that any bounded complementary component of
$J(f^*)$ is a Fatou component of $J(f)$; similarly, any bounded
complementary component of $\imp(\al^-)$ is a Fatou component of
$J(f)$. Therefore, bounded complementary components of either $J(f^*)$
or $\imp(\al^-)$ are disjoint from $J(f^*)\cup \imp(\al^-)$. Moreover,
if a bounded complementary domain of $\imp(\al^-)$ and a bounded
complementary domain of $J(f^*)$ are non-disjoint, then they coincide.

If $T$ is disconnected, then $T'$ is disconnected too. Divide the
components of $T'$ in two groups whose unions $A$ and $B$ are disjoint
non-separating compact sets. Choose open Jordan disks $U\supset A$ and
$V\supset B$ such that $\ol{U}$ and $\ol{V}$ are disjoint closed Jordan disks.
Then, by Moore's theorem \cite{dave86}, a map $\Psi$ that collapses $\ol{U}$ and
$\ol{V}$ maps $\C$ to a space homeomorphic to the plane, in which the
continuum $\Psi(\thu(\imp(\al^-)))$ and the continuum $\Psi(K(f^*))$
intersect over a two-point set. By Theorem 61.4 from \cite{mun00}, this
implies that $\Psi(\thu(\imp(\al^-)))\cup \Psi(K(f^*))$ is a separating
continuum, which implies that $T'=\thu(\imp(\al^-))\cap K(f^*)$ is a
separating continuum. Since both $\thu(\imp(\al^-))$ and $K(f^*)$ are
non-separating, there exists a bounded complementary component $W$ of
$T'$ and an open (in relative topology of $\bd(W)$) set $E\subset
\bd(W)\cap K(f^*)$.
Then, by Theorem \ref{t:conn-extepoly}, we can find a (pre)periodic angle $\ga\in \mathcal B$ such that the (pre)periodic
polynomial-like ray $R^*(\psi(\ga))$ has a terminal interval contained in $W$, lands in $E$ and is homotopic to the external ray $R^e_f(\ga)$ by a homotopy outside $K(f^*)$ fixing its landing point, a contradiction.
Thus, $T$ is connected.
Similarly, $\imp(\be^+)\cap J(f^*)$ is a continuum.

By way of contradiction, assume that the continua $\imp(\al^-)\cap
J(f^*)$ and $\imp(\be^+)\cap J(f^*)$ are disjoint. By choosing $U^*$
sufficiently tight around $K(f^*)$ we may assume that there are Jordan
disks $V_\al$, $V_\be$ with
$$
\imp(\al^-)\cap J(f^*)\subset V_\al\subset \ol{V_\al}\subset U^*,
\quad \imp(\be^+)\cap J(f^*)\subset V_\be\subset \ol{V_\be}\subset U^*
$$
such that $\ol{V_\al}\cap \ol{V_\be}=\0$.
By Theorem \ref{t:conn-extepoly}, we can find a pair of (pre)periodic angles
$\psi(\ga)$, $\psi(\eta)$ in $\psi(\mathcal B)$ such that
$R^*(\psi(\ga))\cup R^*(\psi(\eta))\cup J(f^*)$ separates
$\ol{V_\al}\sm K(f^*)$ from $\ol{V_\be}\sm K(f^*)$. Then $R_f(\ga)\cup
R_f(\eta)\cup J(f^*)$ also separates $\ol{V_\al}\sm K(f^*)$ from
$\ol{V_\be}\sm K(f^*)$.
It follows that $\{\ga,\eta\}$ separates $\{\al,\be\}$ in $\R/\Z$. Thus
$\ga$ or $\eta$ belongs to $(\al,\be)$; assume that it is $\ga$. The
ray $R_f(\ga)$ is periodic and lands in $J(f^*)$, a contradiction with
the definition of $\Uf(f)$.

Let us prove the remaining claims of the lemma.
If $J(f^*)$ has any points outside of $\imp(\al^-)\cup \imp(\be^+)$, then these points
cannot belong to the unbounded component of $\C\sm \Delta(\al,\be)$ that contains external rays with arguments from $(\al, \be)$.
On the other hand, no point of $J(f^*)$ can belong to a bounded component of
$\C\sm \Delta(\al,\be)$ because otherwise some points of $J(f^*)$ do
not belong to the closure of $\C\sm K(f)$.
Hence $J(f^*)\sm (\imp(\al^-)\cup \imp(\be^+))$ is contained in the unbounded component of
$\C\sm \Delta(\al,\be)$ containing external rays with arguments from
$(\be, \al)$.

The last claim of the lemma follows from the fact that $\imp(\al^-)\cap \imp(\be^+)\ne \0$ and from \cite{bco11, kiwi97}.
\end{proof}

\section{Proof of Theorem C}\label{s:proofc}

This section is mainly devoted to a proof of Theorem C; along the way we also complete the proof of Theorem B and show that root polynomials belong to $\cuc_\la$.
First we show that parameter wakes in $\Fc_\la$ are disjoint from $\cuc_\la$ where $|\lambda|\le 1$.
Recall that if $f$ is immediately renormalizable, then we denote by $f^*$ its
quadratic-like restriction upon the appropriate planar domain $U^*$.

\subsection{Parameter wakes are disjoint from $\cu$}

In this subsection, $\lambda$ is a fixed complex number such that
$|\lambda|\le 1$. Recall, that the \emph{main cubioid} $\cu$ is the
set of classes $[f]\in\Mc_3$ with the following properties: $f$ has
a non-repelling fixed point, $f$ has no repelling periodic cutpoints
in $J(f)$, and all non-repelling periodic points of $f$, except at
most one fixed point, have multiplier 1. Recall that $\cuc_\la$ is
the set of all polynomials $f\in\Fc_\la$ with $[f]\in\cu$. We will
need Theorem~\ref{t:fxpt} when proving that certain polynomials
$f\in\Fc_\la$ do not belong to $\cuc_\la$.

We will prove that all parameter wakes in the
$\lambda$-slice $\Fc_\lambda$ are disjoint from $\cuc_\la$, for
$\la\ne 1$. Recall, that the \emph{$\la$-connectedness locus
$\Cc_\lambda$} is the set of all $f\in\Fc_\la$ such that $K(f)$ is
connected. A \emph{limb} of $\Cc_\la$
is the intersection of $\Cc_\la$ with a parameter wake.

\begin{lem}\label{l:perpt-limb}
Suppose that $f$ lies in a limb of $\Cc_\la$.
Then either $f$ has a repelling periodic cutpoint in $K(f)$, or $f$ belongs to
a special wake $\Wc_\la(\theta_1,\theta_2)$ of some period $k$, the
dynamic wedge $W_f(\ta_1+\frac13, \ta_2+\frac23)$ contains a non-repelling periodic point $x\ne 0$ of period $k$ of multiplier
different from 1 such that for every $i$, $0<i<k$, the point $f^{\circ i}(x)$ belongs to the dynamic wedge $W_f(\si_3^{\circ i}(\ta_1), \si_3^{\circ i}(\ta_2))$.
\end{lem}

Recall that $W_f(\al,\be)$ is the wedge bounded by the dynamic rays
$R_f(\al)$, $R_f(\be)$ and the point $z$ on which they land or crash;
it contains external rays with arguments from $(\al, \be)$.

\begin{proof}
By our assumptions, the filled Julia set $K(f)$ is connected, and $f$
lies in some parameter wake $\Wc_\la(\theta_1,\theta_2)$ of some period
$k$. The rays $R_f(\theta_1+\frac 13)$ and $R_f(\theta_2+\frac 23)$
land at the point $x_1$, which is either a repelling periodic point (if
the parameter wake is non-special), or the parabolic point $0$ (if the
parameter wake is special). If the parameter wake is non-special, we
are done. Hence we may assume that the parameter wake is special. Set
$X_1=\ol{W_f(\ta_1+\3,\ta_2+\4)}\cap K(f)$. Then $X_1$ is a special
piece for $f^{\circ k}$ with exit continuum $\{0\}$. By
Theorem~\ref{t:fxpt} applied to $X_1$ and $f^{\circ k}$, the set $X_1$
contains either (1) a non-repelling $f^{\circ k}$-fixed point in $X_1$,
whose multiplier is different from 1, or (2) an $f^{\circ k}$-invariant
parabolic domain with the $f^{\circ k}$-fixed point on its boundary and
multiplier $1$, or (3) a repelling $f^{\circ k}$-fixed cutpoint of
$X_1$. Observe that case (3) is impossible because we assume that $f$
does not have repelling periodic cutpoints.

Let us show that case (2) is impossible. Suppose that there is an
$f^{\circ k}$-invariant parabolic domain $\Omega$ in $X_1$. Let $x$ be
the corresponding parabolic $f^{\circ k}$-fixed point of multiplier
$1$. Since the parameter wake $\Wc_\la(\theta_1,\theta_2)$ is special,
then, by Theorem~\ref{t:dyn-sp}, we have $x\ne 0$. Since the multiplier
at $x$ is 1, there exists an external $f^{\circ k}$-fixed ray
$R_f(\ga)$ landing at $x$ with $\ga\in (\ta_1, \ta_2)$. We claim that
this is impossible for combinatorial reasons. Indeed,  the arc
$[\ta_1+\frac13, \ta_2+\frac23]$ maps onto its image under the angle
tripling map so that the arc $[3\ta_1, 3\ta_2]$ is covered twice while
the rest of the circle is covered exactly once. Since the map is
two-to-one on $\Omega$, it follows that all external rays landing at
points from $\Omega$ must have arguments that map to $[3\ta_1,
3\ta_2]$. However, a simple analysis shows that the only two fixed
angles under the $k$th iterate of the angle tripling map in $[3\ta_1,
3\ta_2]$ are $3\ta_1$ and $3\ta_2$ themselves, a contradiction with the
existence of $\ga$.

Finally, consider case (1). Then since $x$ is non-repelling, then by
Theorem 4.3 from \cite{bclos} there exists a critical point $d$ such
that no (pre)periodic cut separates $d$ and $x$. Hence $x$ belongs to
the strip bound by the two cuts: the cut formed by the rays
$R_f(\ta_1+\frac23)$, $R_f(\ta_2+\frac23)$ (these rays land on the point
$0$) and the cut formed by the rays $R_f(\ta_2+\frac13),
R_f(\ta_1+\frac23)$ (these rays land at the preimage of $0$ not equal to $0$).
It follows that for every $i$, $0<i<k$, the point $f^{\circ i}(x)$ belongs
to the dynamic wedge $W_f(\si_3^{\circ i}(\ta_1)$, $\si_3^{\circ i}(\ta_2))$ as claimed.
\end{proof}

We are now ready to prove the following theorem.

\begin{thm}
  \label{t:wakes-cu}
  Parameter wakes in $\Fc_\la$, where $\la\ne 1$, are disjoint from $\cuc_\la$.
\end{thm}

\begin{proof}
Suppose that $\la\ne 1$.
Let $f\in\Wc_\la(\theta_1,\theta_2)$.
  If $K(f)$ is disconnected, then, by definition, $f\notin\cuc_\la$.
  Suppose that $K(f)$ is connected.
  By Lemma \ref{l:perpt-limb}, the map $f$ has a non-repelling periodic
  point $x\ne 0$, whose multiplier is different from 1, or
  a periodic repelling cutpoint in $K(f)$.
  In the latter case, we have (by definition of $\cuc_\la$) that $f\notin\cuc_\la$.
  In the former case, since $\la\ne 1$, we again see that,
  by definition of $\cuc_\la$, we have $f\notin\cuc_\la$.
\end{proof}

More details on the mutual location of $\cuc_1$ and the wakes in
$\Fc_1$ are given in Section \ref{ss:fc1}.

\subsection{Properties of polynomials from $\Cc_\la\sm \cuc_\la$}

We need the \emph{Separation Lemma} of Jan Kiwi \cite{Ki} inspired by
\cite{GM}; a convenient version is given below (in fact, a stronger
statement is proved in \cite{Ki}).

\begin{lem}[Separation Lemma \cite{Ki}]
\label{l:sep}
  Suppose that $f$ is a polynomial with connected Julia set.
  Then there is a finite collection of periodic dynamic rays for $f$
  $($landing at periodic repelling or parabolic points$)$
  such that the closure of the union of these rays divides the plane into parts
  with the following property: every part contains at most one non-repelling
  periodic point of $f$ or one parabolic basin of $f$.
\end{lem}

A cut formed by periodic dynamic rays from Lemma \ref{l:sep} will be called
a \emph{K-cut}. Let $\Upsilon$ be the closure of the union of the
periodic rays from Lemma \ref{l:sep}; every component of
$\C\sm\Upsilon$ can contain at most one non-repelling periodic point of
$f$ or parabolic basin of $f$.  Let $x\ne y$ be periodic non-repelling
points of $f$ and let $\Ga$ be a K-cut. Say that $\Ga$ \emph{separates
domains at $x$} from $y$ if $\Ga$ separates $x$ and $y$ and, if $x$ is
attracting or parabolic, then $\Ga$ separates all attracting or parabolic domains at $x$
from $y$.

We need a few laminational results from \cite{BOPT}.

\begin{lem}[\cite{BOPT}]\label{l:noreguc1}
If a geolamination $\lam$ has an invariant quadratic gap $\Uf$ of
regular critical type, then, for every chord $\ol{\al\be}$ of the unit
disk such that $\si_3^{\circ n}(\ol{\al\be})$ does not cross an edge of
$\Uf$ for all $n$, we have that $\ol{\al\be}$ is eventually mapped to a
subset of $\Uf$.
\end{lem}

We will use the following corollary of Lemma \ref{l:noreguc1}.

\begin{cor}\label{c:noreguc1}
If $f\in \Cc_\la$ with $|\la|\le 1$ is immediately renormalizable, then
there are no cuts $\Ga(\ol{\al\be})$ with $($pre$)$periodic vertex $x$
such that $\ol{\al\be}$ crosses an edge of $\Uf(f)$. Moreover, if there
exists a periodic cut with vertex $x\ne 0$, then $\Uf(f)$ is of
periodic type, and, for any periodic cut $\Ga(\ol{\al\be})$ with vertex
$x$, the chord $\ol{\al\be}$ does not cross the siblings of the major
$\Mf$ of $\Uf(f)$.
\end{cor}

\begin{proof}
Suppose that there exists a cut $\Ga(\ol{\al\be})$ with (pre)periodic
vertex $x$ such that $\ol{\al\be}$ crosses an edge of $\Uf(f)$. Then,
by definition, there are points of $J(f^*)$ in either component of
$\C\sm \Ga(\ol{\al\be})$, which implies that $x\in J(f^*)$, and hence
$\ol\al$, $\ol\be\in \widehat X\subset \Uf'(f)$, a contradiction
(indeed, by the assumption at least one of $\ol\al$, $\ol\be$ must not
belong to $\Uf'(f)$).

Suppose now that $x\ne 0$ and that $\Uf(f)$ is of regular critical
type. Let $\hf$ be the convex hull of the finite set $\hf'\subset \uc$
of points $\ol\ga$, where $\ga$ are arguments of rays landing at $x$.
By the previous paragraph, images of $\hf$ do not cross edges of
$\Uf(f)$. Then, by Lemma \ref{l:noreguc1}, the set $\hf$ is eventually
mapped inside $\Uf(f)$. However, no image of $x$ can be a cutpoint of
$J(f^*)$ since $|\la|\le 1$ and, hence, there are no periodic cutpoints
of $J(f^*)$ different from $0$. This contradiction proves the second
claim.

Now, let $\Gamma(\alpha\beta)$ be a periodic cut with vertex $x$; set
$\ell=\ol{\al\be}$. By the above, no image of $\ell$ crosses edges of
$\Uf$. Suppose that an image $\widetilde\ell$ of $\ell$ crosses the
sibling $\Mf^*$ of $\Mf$ that is disjoint from $\Uf$. If
$\widetilde\ell$ is disjoint from $\Uf$ then $\si_3(\widetilde\ell)$
crosses an edge of $\Uf$, a contradiction. Otherwise an endpoint of
$\widetilde\ell$ must coincide with an endpoint of $\Mf$. Then either
$\si_3(\widetilde\ell)\subset \Uf$, and all leaves from the orbit
$\widetilde\ell$ are contained in $\Uf$, a contradiction ($\ell$ is
periodic!), or $\si_3(\widetilde\ell)$ crosses an edge of $\Uf$, again
a contradiction. Hence no image of $\ell$ crosses $\Mf^*$ as desired.
\end{proof}

Below, we prove a sequence of lemmas in which we study polynomials
$f\in \Cc_\la\sm\cuc_\la$. Our aim is to prove an analog of
Proposition \ref{p:uflaf}. We will do this in a step by step fashion.
Theorem \ref{t-extepoly} and other tools applicable in the disconnected case do not apply anymore.
Instead we have to use various planar tools
combining continuum theory and dynamics of (geo)laminations.

We now investigate the situation, where $\Uf(f)$ is of periodic type.

\begin{dfn}[cf. \cite{BOPT}]\label{d:periot}
Let $\Uf$ be a quadratic invariant gap of periodic type. Denote its
major by $\ol{\ta_1+\frac 13\,\ta_2+\frac 23}=\Mf$ and assume that it
is of period $k$ with sibling $\ol{\ta_2+\frac13\,
\ta_1+\frac23}=\Mf^*$. Let $\Vf'$ be the set of all points $\ol\al$ of
the circle such that, for every $n$, we have $3^n\al\in
[3^n(\ta_1+\frac 13),3^n(\ta_2+\frac 23)]$, and let $\Vf$ be the convex
hull of $\Vf'$. In this setting, $\Uf$ is called the \emph{senior gap}
while $\Vf$ is called the \emph{vassal gap} (of $\Uf$). If $\Uf=\Uf(f)$
is defined by an immediately renormalizable polynomial $f\in \Cc_\la$,
we also consider the continuum $I(f)=I=\imp^-(\ta_1+\frac 13)\cup
\imp^+(\ta_2+\frac 23)$ (the fact that $I$ is a continuum follows from Proposition \ref{p:uflaf1}).
\end{dfn}

In Theorem \ref{t:periocan}, specific (geo)laminations with a quadratic
invariant gap $\Uf$ of periodic type are studied. Recall that, given a
geolamination $\lam$, the map $\si_3$ extends to all leaves of $\lam$
so that leaves are mapped to leaves or points in $\uc$. We also extend
$\si_3$ to all gaps of $\lam$ so that gaps are mapped to gaps, leaves
or points and so that the obtained extension is a continuous self-map
of the closed unit disk.

\begin{thm}[\cite{BOPT}]\label{t:periocan}
There exists a unique invariant lamination $\sim_\Uf$ with a given
senior gap $\Uf$ of periodic type and such that the vassal gap
$\ch(\Vf')=\Vf$ of period $k$ is a gap of $\sim_\Uf$. The gap $\Vf$
maps two-to-one onto itself under $\si_3^{\circ k}$, all other gaps of
$\sim_\Uf$ are one-to-one pullbacks of $\Uf$ or $\Vf$ and all leaves of
the corresponding geolamination $\lam_\Uf$ are edges of these gaps.
Also, there are no points of period $k$ located between the major $\Mf$
and its sibling $\Mf^*$ except for the endpoints of $\Mf$.
\end{thm}

Theorem \ref{t:periocan} implies the following corollary.

\begin{cor}\label{c:periocan}
Let $\Uf$ be a senior gap of periodic type with major $\Mf$ and vassal
gap $\Vf$. Let $\ell$ be a chord. Then one of the following holds.
\begin{enumerate}
\item An eventual $\si_3$-image of $\ell$ crosses an edge of $\Uf$.
\item An eventual $\si_3$-image of $\ell$ is contained in $\Uf$.
\item An eventual $\si_3$-image of $\ell$ separates $\Mf$ from
    $\Mf^*$ in $\disk$.
\end{enumerate}
\end{cor}

\begin{proof}
Set $\ell=\ol{\al\be}$. Assume that neither (1) nor (2) holds. Let an
eventual image of $\ell$ intersect $\Uf$. Then, by the assumption, this
image is a chord having a common endpoint with an edge of $\Uf$ but
otherwise located outside $\Uf$.
Suppose that $\Mf=\ol{\ta_1+\frac 13\,\ta_2+\frac 23}$.
By properties of senior gaps of
periodic type, we may assume that $\al=\ta_1+\frac 13$ and $\ta_1+\frac
13<\be<\ta_2+\frac 23$ while, by our assumptions, no image of $\ell$
has an endpoint $\ta_1+\frac23$ or an endpoint $\ta_2+\frac13$. This
implies that an eventual image $\widetilde\ell$ of $\ell$ will have one
endpoint $\ta_1+\frac 13$ and the other endpoint belonging to
$(\ta_2+\frac13,\ta_1+\frac23)$. It follows that $\si_3^{\circ
k}(\widetilde\ell)$ will have an endpoint $\ta_1+\frac 13$ and the
other endpoint in $(\ta_2+\frac 23,\ta_1+\frac 13)$, which implies that
either (1) or (2) must hold, a contradiction. Thus we may assume that
eventual images of $\ell$ are disjoint from $\Uf$.

We claim that then (3) holds. Indeed, suppose first that $\ell$ is
eventually mapped inside $\Vf$. Then the claim follows from the fact
that $\si_3^{\circ k}|_{\Vf}$ is semiconjugate to $\si_2$ and the
well-known (and easy) fact that an eventual $\si_2$-image of any
non-degenerate chord separates $0$ and $\frac12$ or coincides with
$\ol{0\frac12}$. Suppose that $\ell$ is never mapped inside $\Vf$.
Combining this and the fact that all images of $\ell$ are disjoint from
$\Uf$ we see that $\ell$ crosses a leaf $\ell_1$ of $\lam_\Uf$. If no
image of $\ell$ separates $\Mf$ and $\Mf'$, then this implies that the
images of $\ell$ continue crossing the corresponding images of
$\ell_1$. However, $\ell_1$ is eventually mapped to $\Mf$, a
contradiction.
\end{proof}

Lemma \ref{l:nocutinab} studies periodic cuts of $f$.

\begin{lem}\label{l:nocutinab}
Suppose that $f\in\Cc_\la$ is immediately renormalizable, and that
$\Uf(f)=\Uf$ is of periodic type with major $\Mf=\ol{\ta_1+\frac
13\,\ta_2+\frac 23}$ of period $k$. Suppose that there exists a
periodic cut $\Ga(\ell_1)=\Ga$ with vertex $\ne 0$ corresponding to a
chord $\ell_1$ that does not belong to the orbit of $\Mf$. Then there
is a cut $\Ga(\ell_2)$ such that $\ell_2$ is separated from $\Uf$ by
$\Mf$ and such that $\si_3^{\circ k}(\ell_2)$ is separated from $\Uf$
by $\ell_2$.
\end{lem}

\begin{proof}
Let us show that all images of $\ell_1$ is disjoint from $\Uf$. Assume
otherwise. Then by Corollary \ref{c:noreguc1}, the iterated images of $\ell_1$ do not cross
edges of $\Uf$. Hence an eventual image $\ell'_1$ of $\ell_1$ has
either $\ta_1+\frac 13$ or $\ta_2+\frac 23$ as its endpoint. On the
other hand, by Corollary \ref{c:noreguc1}, $\ell'_1$ cannot cross $\Mf^*$, the
sibling of $\Mf$ disjoint from $\Uf$. By Theorem \ref{t:periocan} except for
$\ta_1+\frac 13$ or $\ta_2+\frac 23$ there are no other periodic points
of period $k$ between $\Mf$ and $\Mf^*$. Hence $\ell'_1=\Mf$, a
contradiction.

Since $\ell_1$ and all its images are disjoint from $\Uf$, then, by
Corollary \ref{c:periocan}, there is an image of $\ell_1$ separating $\Mf$ and
$\Mf^*$. We may assume that $\ell_1$ separates $\Mf$ and $\Mf^*$.
Recall that $k$ is the period of $\Mf$. Let $\ol\al_1$ and $\ol\be_1$
be the endpoints of $\ell_1$ chosen so that
$$
\ta_1+\frac 13<\al_1<\be_1<\ta_2+\frac 23.
$$
Since both arcs of $\uc$ between $\Mf$ and $\Mf^*$ are mapped onto the
major  hole of $\Mf$ homeomorphically under $\si_3^{\circ k}$, there is
a unique angle $\al_2$ in $(\ta_1+\frac 13,\al_1)$ with
$3^k\al_2=\al_1$, and there is a unique angle $\be_2$ in
$(\be_1,\ta_2+\frac 23)$ with $3^k\be_2=\be_1$. Consider the chord
$\ell_2=\ol{\al_2\be_2}$; then we have $\si_3^{\circ
k}(\ell_2)=\ell_1$. By Corollary \ref{c:noreguc1}, no pullback of $\ell_1$ can
cross an edge of $\Uf(f)$. Hence the cut that maps onto $\Ga(\ell_1)$
under $f^{\circ k}$ and contains $R_f(\al_2)$ coincides with
$\Ga(\ell_2)$. Hence $\ell_2$ is the desired pullback of $\ell_1$.
\end{proof}

\begin{lem}\label{l:case0}
Let $f\in \Cc_\la\sm \cuc_\la$, where $|\la|\le 1$. Moreover, suppose
that there is one cycle of parabolic domains at $0$ and two cycles of
external rays landing at $0$. Let $\gf$ be the gap whose vertices are
arguments of all external rays landing on $0$. Then $\gf$ is of type D.
There exists a unique major $\Mf=\ol{\ta_1+\frac13, \ta_2+\frac23}$ of
$\gf$ such that the quadratic invariant gap defined by $\Mf$ coincides
with $\Uf(f)$, and the polynomial $f$ belongs to a special wake
$\Wc_\la(\theta_1,\theta_2)$
\end{lem}

\begin{proof}
The fact that $\gf$ is of type D is immediate. Moreover, by the
assumptions $0$ is parabolic. Since $f\in \Cc_\la\sm \cuc_\la, |\la|\le
1$, then $f$ is immediately renormalizable. It follows that $\Uf(f)$
contains arguments of all external (pre)periodic rays landing on points
from the cycle of parabolic domains at $0$. Denote by
$\Mf=\ol{\ta_1+\frac13, \ta_2+\frac23}$ the major of $\gf$ coming from
the cycle of edges of $\gf$ corresponding to the cycle of planar wedges
at $0$ that do not contain parabolic domains at $0$. Then the wedge
$W_f(\ta_1+\frac13, \ta_2+\frac23)$ contains no points of $J(f^*)$. It
follows that the quadratic invariant gap defined by $\Mf$ coincides
with $\Uf(f)$. Moreover, by definition the polynomial $f$ belongs to a
special wake $\Wc_\la(\theta_1,\theta_2)$ as desired.
\end{proof}

We are ready to prove a lemma describing dynamics of polynomials $f\in
\Cc_\la\sm \cuc_\la, |\la|\le 1$.

\begin{lem}\label{l:locate}
Let $f\in \Cc_\la\sm \cuc_\la$, where $|\la|\le 1$. Then $f$ is
immediately renormalizable, $\Uf(f)$ is of periodic type, and, if
$\Mf=\ol{\ta_1+\frac 13\,\ta_2+\frac 23}$ is the major of $\Uf(f)$,
then the rays $R_f(\ta_1+\frac 13)$, $R_f(\ta_2+\frac 23)$ land at the
same point $z\in K(f^*)$. Moreover, either $z\ne 0$ is repelling, or
$z=0$ is parabolic with exactly one cycle of parabolic Fatou domains at
$0$ and two cycles of external rays landing on $0$. In particular, $f$
belongs to one of the parameter wakes.
\end{lem}

\begin{proof}
Since $f\in \Cc_\la\sm \cuc_\la$, then, by Corollary \ref{c:unboren}, the polynomial $f$ is immediately renormalizable.
Consider first the case when there exists a periodic
cut $\Ga(\ell_1)$ with vertex $x\ne 0$. By Corollary \ref{c:noreguc1}, the gap
$\Uf(f)$ is of periodic type. Assume that $\Mf=\ol{\ta_1+\frac
13\,\ta_2+\frac 23}$ is the major of $\Uf(f)$. We claim that the rays
$R_f(\ta_1+\frac 13)$, $R_f(\ta_2+\frac 23)$ land at the same point. By
way of contradiction suppose otherwise. Then the cut $\Ga(\ell_1)$
satisfies conditions of Lemma \ref{l:nocutinab}. This implies that there
exists a cut $\Ga(\ell_2)$ with properties from that lemma. Consider
the first $k-1$ iterated $\si_3$-pullbacks of the chord $\ell_2$ whose
endpoints are located in holes of $\Uf(f)$ behind the periodic edges.
The corresponding cuts together with a suitably chosen equipotential
bound a Jordan disk such that the restriction of $f$ to this Jordan
disk is a quadratic-like map. Let $\widehat J$ be the corresponding
quadratic-like Julia set. Then, by Lemma \ref{l:uniqj}, we have $\widehat
J=J(f^*)$, and so the landing point $z_1$ of $R_f(\ta_1+\frac 13)$ and
the landing point $z_2$ of $R_f(\ta_2+\frac 23)$ belong to $J(f^*)$.

Suppose that $z_1\ne z_2$. Consider the component $W$ of $\C\sm
[R_f(\ta_1+\frac 13)\cup R_f(\ta_2+\frac 23)\cup J(f^*)]$ containing
external rays with arguments from $(\ta_1+\frac 13, \ta_2+\frac 23)$.
Since $z_1\ne z_2$, there are preperiodic points of $J(f^*)$ that
belong to $\bd(W)$ and such that some quadratic-like rays landing at
these points (or at least parts of these quadratic-like rays near
$K(f^*)$) are contained in $W$. Then, by Theorem \ref{t:conn-extepoly}, there
are repelling preperiodic points $y\in J(f^*)\cap \bd(W)$ at which the
corresponding external rays with arguments from $(\ta_1+\frac 13,
\ta_2+\frac 23)$ land. However, by definition of $\widehat X$, this is
impossible. Thus $z_1=z_2=z\in K(f^*)$, as desired.

Let us prove the claims of the lemma concerning the point $z$. First
assume that $z\ne 0$. Then, if $z$ is parabolic, it would imply that
points of a parabolic domain at $z$ will belong to $K(f^*)$, a
contradiction. Hence $z$ is repelling as stated in the lemma.  Observe
that by definition in this case $f$ belongs to the parameter wake
$\Wc_\la(\theta_1,\theta_2)$ that is non-special.

Assume now that $z=0$. Then by definition $z$ must be parabolic. We
claim that there is exactly one cycle of parabolic domains at $0$.
Indeed, if there are two cycles of parabolic Fatou domains at $0$ then,
clearly, there are no quadratic-like Julia sets containing $0$, which
implies that $f$ is not immediately renormalizable, a contradiction.
Thus, there is exactly one cycle, say, $\mathcal P$, of parabolic
domains at $0$. By way of contradiction suppose that there is exactly
one cycle of external rays landing at $0$. Then each wedge formed by
these rays contains one parabolic domain from this cycle. However the
wedge associated to the major $\Mf=\ol{\ta_1+\frac 13\,\ta_2+\frac 23}$
(i.e., formed by the external rays with arguments $\ta_1+\frac 13$ and
$\ta_2+\frac 23$) has $z=0$ as its vertex and by construction cannot
contain any parabolic domains from $\mathcal P$, a contradiction. Hence
there is exactly one cycle of parabolic Fatou domains at $0$ and there
are exactly two cycles of external rays landing on $0$ as claimed. Then
the rest follows from Lemma \ref{l:case0}.

Assume now that there are no periodic cuts with non-zero vertex. In
particular, there are no periodic cuts whose vertex is a repelling
periodic point. Since $f\notin \cuc_\la$, it follows that there exist
two distinct periodic non-repelling points of multipliers not equal to
$1$. In particular, there exists a non-repelling periodic point $y\ne
0$ of multiplier not equal to $1$. If $y$ were parabolic then a
periodic cut with non-zero vertex would exist, a contradiction. Hence
$y$ is not parabolic. By Lemma \ref{l:sep} (Kiwi's Separation Lemma) there
are periodic cuts separating $0$ (or possibly existing parabolic
domains at $0$) from points of the orbit of $y$. Since the only
periodic cuts of $f$ are cuts with vertex $0$, it follows that $0$ is
parabolic, and cuts at $0$ separate parabolic domains at $0$ from the
points of the orbit of $y$. This is only possible if there are two
cycles of external rays landing on $0$ and forming two cycles of wedges
at $0$: one cycle contains a cycle of parabolic domains at $0$ and the
other cycle of wedges contains the entire orbit of $y$. Then, again,
the desired claims follow from Lemma \ref{l:case0}.
\end{proof}

Lemma \ref{l:locate} implies a few claims from our main theorems, which are not proven yet. 
To begin with, recall that the last claim of Theorem B
states that if $f_{root}$ is the root point of a parameter wake
$\Wc_\la(\theta_1,\theta_2)$, then $f_{root}$ belongs to $\cuc_\la$.
Indeed, otherwise $f_{root}$ would have properties listed in Lemma \ref{l:locate}, and it clearly does not.

Now, the first claim of Theorem C is that the set $\cuc_\la$ is
disjoint from all parameter wakes, unless $\la=1$. Indeed, suppose that
$\la\ne 1$ and $f\in \cuc_\la$ belongs to a parameter wake. Since $f\in
\cuc_\la$ then by definition $f$ has no periodic cutpoints of its Julia
set. Moreover, since $f\in \cuc_\la$ then $J(f)$ is connected. 
Hence by Lemma \ref{l:perpt-limb} $f$ has a non-repelling periodic point $x\ne 0$
with multiplier not equal to 1. Since $\la\ne 1$, $f$ has at least two
non-repelling periodic points with multiplier not equal to 1, a
contradiction with $f\in \cuc_\la$.

To complete the proof of Theorem C, it remains to prove that
$\cuc_\la$, where $|\la|\le 1$, is a full continuum. The set $\Cc_\la$
is a full continuum \cite{BH}; this is very similar to the fact that
the standard Mandelbrot set is a full continuum \cite{DH}. By Theorem
\ref{t:wakes-cu} and Lemma \ref{l:locate}, the set $\cuc_\la$ is obtained
from the full continuum $\Cc_\la$ by removing all limbs. Note that, if
we remove finitely many limbs from $\Cc_\la$, then we are left with a
full continuum; indeed, a limb does not separate $\Cc_\la$. Being the
intersection of a nested sequence of full continua, the set $\cuc_\la$
is also a full continuum. This concludes the proof of Theorem C.

\subsection{Root points of non-special wakes}
We now complete the characterization of non-special wakes by providing
conditions on polynomials necessary and sufficient for being root
points of non-special wakes. Recall the following notation: a unique
quadratic invariant gap with major $\ol{0 \frac12}$ contained in the
upper half of the unit disk (\textbf{a}bove $\ol{0 \frac12}$) is
denoted by $\fg_a$ while a unique quadratic invariant gap with major
$\ol{0 \frac12}$ contained in the lower half of the unit disk
(\textbf{b}elow $\ol{0 \frac12}$) is denoted by $\fg_b$. We begin with
a laminational claim concerning majors of quadratic invariant gaps of
$\si_3$.

\begin{lem}\label{l:major-dyn}
Let $\Mf$ be a major of a quadratic invariant gap. Suppose that $\Mf$
is of period $q$. Then the only major of a quadratic invariant gap of
$\si_3$ in the orbit of $\Mf$ is $\Mf$ itself.
Moreover, if there exists a finite stand alone periodic gap $G$ of period $q$ with edge $\Mf$, then $G$ is a triangle, and
the only edge of $G$ whose $\si_3$-orbit contains a major of a quadratic invariant gap of
$\si_3$ is $\Mf$ itself.
\end{lem}

Recall that by a finite stand alone gap $G$ of period $q$ we mean a gap that maps onto itself under $\si_3^{\circ q}$ so that $G$, $\si_3(G)$, $\dots$, $\si_3^{\circ q-1}(G)$ are pairwise disjoint.

\begin{proof} The first claim is immediate if $\Mf$ is invariant (i.e.,
if $\Mf=\ol{0 \frac12}$). Suppose that $\Mf$ is not invariant (i.e., $q>1$).
Then, by \cite{BOPT}, the chord $\Mf$ is the only leaf in the orbit of
$\Mf$ that divides the circle into two arcs of length $>\frac13$.
Since, by \cite{BOPT}, every major of an invariant quadratic gap of $\si_3$ divides the circle into two arcs of length $>\frac13$, we see that $\Mf$ is the only major of an
invariant quadratic gap of $\si_3$ in the orbit of $\Mf$.

Suppose now that there exists a stand alone gap $G$ of period $q$
containing $\Mf$. Clearly, this is impossible if $\Mf=\ol{0 \frac12}$,
thus we may assume that $\Mf\ne \ol{0 \frac12}$ and $q>1$.
Then all vertices of $G$ are of period $q$.
By \cite{kiwi02}, this implies that $G$ is a triangle of period $q$ such
that $\Mf$ is one of its sides, and all vertices of $G$ are of period $q$.
By the previous paragraph, the only major in the orbit of $\Mf$ is $\Mf$ itself.
Denote by $\ell'$ and $\ell''$ the two remaining edges of $G$; also, denote by
$\Uf$ the quadratic invariant gap with major $\Mf$ (since $\Mf\ne \ol{0 \frac12}$, the gap $\Uf$ is unique).

Let us show that the orbit of $\ell'$ does not contain the major of an
invariant quadratic gap (similar arguments show that neither does
$\ell''$). Suppose that there is some $i$, $0\le i<q$, such that
$\si_3^{\circ i}(\ell')$ is the major of a quadratic invariant gap of
$\si_3$. Then $\si_3^{\circ i}(\ell')$ is the unique leaf in the orbit
of $\ell'$ that divides the circle in two arcs of length $>\frac13$. By
definition, we have one of the following two cases:
\begin{enumerate}
\item the leaf $\ell'$ and all its images are contained in $\Uf$, or
\item the leaf $\ell'$ and all its images are contained in the closures
    of the holes of $\Uf$ behind the corresponding images of $\Mf$.
\end{enumerate}
Let us show that in either case $i>0$. Indeed, consider case (1). Then
the quadratic invariant gap generated by $\ell'$ is located on the same
side of $\Mf$ and hence must be strictly contained in $\Uf$. This
contradicts the fact that on both quadratic invariant gaps in question
the map $\si_3$ is exactly two-to-one and the fact that the set of all
preimages of every point is dense in the basis of the corresponding
gap. Similarly one can consider case (2). Hence $i>0$; consider now
cases (1) and (2) separately using the fact that $i>0$. Observe that
non-major leaves from the orbit of $\ell'$ are contained in one of the
two components into which $\si_3^{\circ i}(\ell')$ divides the closed
unit disk, namely, the component containing $\ell'$.

(1) Use the ``projection'' $\psi$ that collapses all edges of $\Uf$ to
points and semiconjugates $\si_3|_{\Uf}$ to $\si_2|_{\uc}$. Then the
$\psi$-image of $G$ equals the $\psi$-image of $\ell'$. Moreover, the
orbit of $\ell'$ (equivalently, of $G$) ``projects'' to a
$\si_2$-periodic orbit consisting of $q$ leaves $\psi(\ell')$,
$\si_2(\psi(\ell'))=\psi(\si_3(\ell'))$, $\dots$. Observe that since,
by \cite{BOPT}, the gap $\Uf$ contains no concatenations of edges, and
the set $\Uf\cap \uc$ is a Cantor set, then $\psi(\ell')$ and all its
$\si_2$-images are non-degenerate leaves.

Moreover, the fact that all triangles in the $\si_3$-orbit of $G$ are disjoint implies that the leaves $\psi(\ell')$, $\si_2(\psi(\ell'))=\psi(\si_3(\ell''))$, $\dots$ are disjoint. Indeed, it is clear that these leaves are pairwise unlinked.
If two of them are concatenated, then a triangle $\si_3^{\circ k}(G)$ would have one of its sides coinciding with $\si^{\circ k}_3(\ell')$ while the remaining vertex would be an
endpoint of another image of $\ell'$.
This would contradict the claim that all distinct images of $G$ are pairwise disjoint.

Recall that since $\si_3^{\circ i}(\ell')$ is a major of some quadratic invariant gap, it is the unique leaf in the orbit of $\ell'$ that divides the circle in two arcs of length $>\frac13$.
In particular, the arc not containing $\ell'$ must be longer than $\frac13$.
Moreover, the rest of the orbit of $\ell'$ is contained in one of the two components into which $\si_3^{\circ i}\ell'$ divides the closed unit disk, namely the component containing $\ell'$.
It follows that $\psi(\si^{\circ i}_3(\ell'))$ is a $\si_2$-periodic leaf of period $q$ whose entire orbit is contained in one component of $\cdisk\sm\psi(\si^{\circ i}_3(\ell'))$, namely in the component bounded by $\psi(\si^{\circ i}_3(\ell'))$ and an arc of length strictly less than $\frac12$.
However, then it follows from well-known properties of $\si_2$ that leaves $\psi(\ell')$,
$\si_2(\psi(\ell'))=\psi(\si_3(\ell'))$, $\dots$ cannot be pairwise disjoint, a contradiction.

(2) Recall that in this case $\ell'$ and all its $\si_3$-images are
contained in the closures of the holes of $\Uf$ behind the
corresponding images of $\Mf$. Since all holes of $\Uf$ behind
$\si_3$-images of $\Mf$ are of length $<\frac13$ except for the major
hole behind $\Mf$, then the only potential major in the $\si_3$-orbit
of $\ell'$ is $\ell'$ itself (recall that the major of a quadratic
$\si_3$-invariant gap has to divide $\uc$ into two arcs of length at
least $\frac13$), contradicting the fact that $i>0$. 

Thus, in either case the orbit of $\ell'$ does not contain the major of
an invariant quadratic gap in its orbit; similar arguments show that
neither does $\ell''$. This completes the proof.
\end{proof}

We are ready to prove the theorem describing root polynomials of
non-special wakes.

\begin{thm}
  \label{t:root-nsp}
  Suppose that $\Wc_\la(\ta_1,\ta_2)$ is a non-special wake, $f\in \Fc_\la$. Then the
  following holds.
\begin{enumerate}
  \item Suppose that $(\ta_1, \ta_2)\ne (\frac23, \frac56)$ and
      $(\ta_1, \ta_2)\ne (\frac16, \frac13)$. Then the dynamic rays
      $R_f(\ta_1+\frac 13)$, $R_f(\ta_2+\frac 23)$ land at the same
      periodic parabolic point $z\ne 0$ of multiplier 1 if and only
      if $f$ is the root point of $\Wc_\la(\ta_1,\ta_2)$.
  \item Suppose that either $(\ta_1, \ta_2)=(\frac23, \frac56)$ or
      $(\ta_1, \ta_2)=(\frac16, \frac13)$. Then the dynamic rays
      $R_f(0)$ and $R_f(\frac12)$ land at the same periodic parabolic
      point $z\ne 0$ of multiplier 1 if and only if $f$ is the root
      point of either $\Wc_\la(\frac23, \frac56)$ or
      $\Wc_\la(\frac16, \frac13)$.
  \end{enumerate}
\end{thm}

Observe that in case (1) the arc $(\ta_1, \ta_2)$ is in one-to-one
correspondence with the major $\Mf=\ol{(\ta_1+\frac 13)\,(\ta_2+\frac
23)}\ne \ol{0 \frac12}$ of a quadratic invariant gap $\Uf$ not equal
$\fg_a$ or $\fg_b$. However, in case (2) both arcs $(\ta_1, \ta_2)=
(\frac23, \frac56)$ or $(\ta_1, \ta_2)=(\frac16, \frac13)$ are
associated with the same major $\ol{0 \frac12}$ that can serve as the
major of either $\fg_a$ or $\fg_b$.

\begin{proof}
The ``if'' part in both cases follows from Theorem \ref{t:dyn-nsp}. We
need to prove the ``only if'' part, i.e. to prove that if the dynamic rays
$R_f(\ta_1+\frac 13)$, $R_f(\ta_2+\frac 23)$ land at the same periodic
parabolic point $z\ne 0$ of multiplier 1, then $f$ is the root point of
$\Wc_\la(\ta_1,\ta_2)$ (in case (1)) or the root point of either
$\Wc_\la(\frac23, \frac56)$ or $\Wc_\la(\frac16, \frac13)$ (in case (2)). The
arguments are alike, so we separate the two cases only later in the
proof.

Suppose that the dynamic rays $R_f(\ta_1+\frac 13)$, $R_f(\ta_2+\frac
23)$ land at the same periodic parabolic point $z\ne 0$ of multiplier
1. Let $q$ be the minimal period of $z$. Since the multipliers of the
periodic points of $g\in\Fc_\la$ are branches of certain multivalued
analytic functions of $g$, there exists a stable component
$\Uc\subset\Fc_\la$ bounded by a real analytic curve such that,
\begin{itemize}
\item we have $f\in\bd(\Uc)$;
\item for all $g\in\Uc$, there is an attracting periodic point $z_g$ of period $q$ depending analytically on $g$;
\item we have $z_g\to z$ as $g\to f$.
\end{itemize}
By Theorem C, the domain $\Uc$ is contained in a parameter wake
$\Wc(\ta^*_1,\ta^*_2)$ and $f$ belongs to the closure of $\Wc(\ta^*_1,\ta^*_2)$.
Then either $f\in \Wc_\la(\ta^*_1,\ta^*_2)$ and the dynamic rays $R_f(\ta^*_1+\frac 13)$,
$R_f(\ta^*_2+\frac 23)$ land at the same \emph{repelling} periodic point $z^*$, or $f$ is the root point of
$\Wc_\la(\ta^*_1,\ta^*_2)$ and the dynamic rays $R_f(\ta^*_1+\frac 13)$,
$R_f(\ta^*_2+\frac 23)$ land at the same \emph{parabolic}
periodic point $z^*$. Consider cases.

First, let $(\ta_1,\ta_2)=(\ta^*_1,\ta^*_2)$. Then %by definition %Theorem \ref{t:dyn-nsp},
$f\notin \Wc_\la(\ta_1,\ta_2)$, as otherwise, since $\Wc_\la(\ta_1,\ta_2)$ is non-special,
the external rays $R_f(\ta_1+\frac13)$ and $R_f(\ta_2+\frac23)$ land at
a repelling periodic point, a contradiction. Hence then $f$ is the root
polynomial of $\Wc_\la(\ta_1,\ta_2)$. This corresponds to case (1).

Suppose now that $(\ta_1,\ta_2)\ne (\ta^*_1,\ta^*_2)$. Set
$\Mf=\ol{(\ta_1+\frac 13)\,(\ta_2+\frac 23)}$ and
$\Mf^*=\ol{(\ta^*_1+\frac 13)\,(\ta^*_2+\frac 23)}$, and again consider
various cases.

(a) The polynomial $f$ is the root polynomial of a non-special wake
$\Wc_\la(\ta^*_1,\ta^*_2)$. Then $z^*\ne 0$ is parabolic. By the
Fatou--Shishikura inequality, $z$ and $z^*$ belong to the same orbit.
By Lemma \ref{l:major-dyn}, we have $\Mf=\Mf^*$.
Since $(\ta_1,\ta_2)\ne (\ta^*_1,\ta^*_2)$, it follows from \cite{BOPT}
that the only way it can happen is when $(\ta_1,\ta_2)=(0, \frac12)$
and $(\ta^*_1,\ta^*_2)=(\frac12, 0)$ or vice versa. This corresponds to
case (2).

(b) The polynomial $f$ is the root polynomial of a special wake $\Wc_\la(\ta^*_1,$ $\ta^*_2)$.
Then, by Theorem \ref{t:dyn-sp}, there are two cycles of parabolic domains at $0$, which contradicts the assumption about $f$ and $z$.

(c) Suppose now that $z^*$ is repelling. Then $f$ is immediately
renormalizable, and $\Uf(f)$ is the quadratic invariant gap with major
$\Mf^*$. It follows from \cite{BOPT} that $\Mf$ crosses the interior of
$\Uf(f)$, hence $z$ separates $K(f)$. In particular, $z\in K(f)$, a
contradiction with the fact that $K(f)$ contains no parabolic points
different from $0$ (see \cite[Theorems A and B]{bopt16}).
\end{proof}

\subsection{The slice $\Fc_1$}
\label{ss:fc1} 
The case $\la=1$ is an exception in Theorem C. Thus it
requires  separate attention. The slice $\Fc_1$ consists of polynomials
$f=z^3+bz^2+z$. The point $0$ is a fixed parabolic point with
multiplier 1 of every polynomial in $\Fc_1$.
Figure \ref{fig:Fc0} shows the parameter slice $\Fc_1$ in which several
parameter rays and several wakes are shown.

\begin{figure}
  \includegraphics[width=8cm]{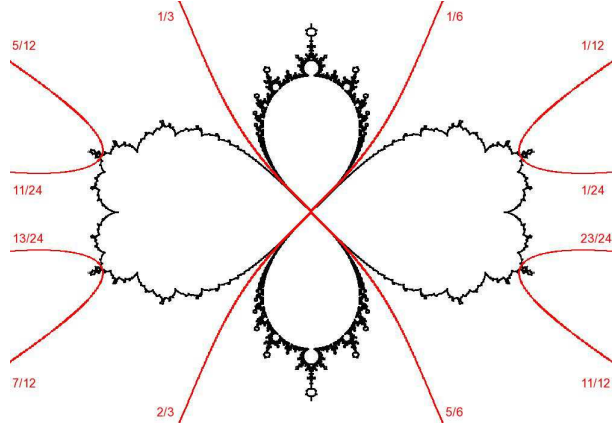}
  \caption{Parameter slice $\Fc_1$ with some rays}
  \label{fig:Fc0}
\end{figure}

The mutual position of
$\cuc_1$ and the wakes in $\Fc_1$ is described below:

\begin{thm}
  \label{t:la1}
  Parameter wakes in $\Fc_1$ of period greater than 1 are disjoint from $\cuc_1$.
  A polynomial $f\in\Fc_1$ belongs to the set
  $$\left(\Wc_1\left(\frac 16,\frac 13\right)\cup\Wc_1\left(\frac 23,\frac 56\right)\right)\cap\cuc_1$$
  if and only if $f$ is affinely conjugate to the root point of some wake $\Wc_{\la^*}(\frac 16,\frac 13)$ with $|\la^*|\le 1$ and $\la^*\ne 1$.
\end{thm}

\begin{proof}
Let $\Wc_1(\ta_1,\ta_2)$ be a parameter wake of period $k>1$. Take
$f\in\Wc_\la(\theta_1,\theta_2)$. If $K(f)$ is disconnected, then, by
definition, $f\notin\cuc_1$. Suppose that $K(f)$ is connected. Then by
Lemma \ref{l:perpt-limb} there are two possibilities at least of which
holds. First, there may exist a repelling periodic cutpoint of $K(f)$.
Then, by definition, $f\notin \cuc_1$, as desired. Second, there may exist
a non-repelling $k$-periodic point $x\ne 0$. Hence there are at least
two non-repelling periodic points, $x$ and $f(x)$, whose multiplier is
different from 1, and so again $f\notin \cuc_1$ as desired.

Suppose now that $f\in\Fc_1$ is affinely conjugate to the root point
$f_{root}$ of some wake $\Wc_{\la^*}(\frac 16,\frac 13)$, where
$|\la^*|\le 1$ and $\la^*\ne 1$. By Lemma \ref{l:locate} and the
description of the root points of wakes given in Theorems
\ref{t:dyn-nsp}, \ref{t:dyn-sp}, we have $[f_{root}]\in\cu$, hence
$f\in\cuc_1$.

Finally, let us prove that if $f\in\Fc_1$ belongs to the set
$$
\left(\Wc_1\left(\frac 16,\frac 13\right)\cup\Wc_1\left(\frac
23,\frac 56\right)\right)\cap\cuc_1,
$$
then $f$ is affinely conjugate to
the root point of some wake $\Wc_{\la^*}(\frac 16,\frac 13)$ with
$|\la^*|\le 1$ and $\la^*\ne 1$. Note that the rays $R_f(0)$ and
$R_f(\frac 12)$ land at $0$ (the common landing point is a parabolic
fixed point of multiplier 1, hence it must be $0$). Since $f\in \cuc_1$,
there are no repelling periodic cutpoints in $K(f)$. Hence, by Lemma
\ref{l:perpt-limb}, the polynomial $f$ has a non-repelling fixed point $x\ne 0$ with
multiplier $\la^*\ne 1$. Consider a translation of $\C$ that moves $x$
to $0$, and let $f^*$ be the cubic polynomial obtained as the
conjugation of $f$ by this translation. Then $f^*\in\Fc_{\la^*}$, and
$x^*=-x$ is a fixed point of $f^*$ with multiplier 1. Since the rays
$R_f(0)$ and $R_f(\frac 12)$ land at $0$, the rays $R_{f^*}(0)$ and
$R_{f^*}(\frac 12)$ land at $x^*$. It follows from Theorem
\ref{t:root-nsp} that $f^*$ is the root point of the wake
$\Wc_{\la^*}(\frac 16,\frac 13)$ or of the wake $\Wc_{\la^*}(\frac
23,\frac 56)$.
\end{proof}

\section*{Acknowledgements}
The first named author was partially supported by NSF grant DMS--1201450.
The second named author was partially supported by NSF grant DMS-0906316.
The third named author was partially supported by the RFBR grant 16-01-00748a.
The study has been funded by the Russian Academic Excellence Project `5-100'.
The research comprised in Theorem C was funded by RScF grant 14-21-00053.

\end{document}